\documentclass[11pt,leqeq]{article}
\usepackage{amssymb,amsthm}
\usepackage{amsmath}
\usepackage{amscd}
\usepackage{xy}
\xyoption{all}
\usepackage[dvips]{graphicx}
\newtheorem{thm}{Theorem}%[section] (If you want theorem numbered
\newtheorem{lem}[thm]{Lemma}%               with section number.  Same
\newtheorem{cor}[thm]{Corollary}%       goes for lemmas, etc.)
\newtheorem{prop}[thm]{Proposition} %--> \begin\end{theorem,lemma,...}
\newtheorem{rem}[thm]{Remark}

\newtheorem{defn}[thm]{Definition}
\newtheorem{exmp}[thm]{Example}

\date{}
\begin{document}
\setlength{\baselineskip}{16pt}
\title{Combinatorial Micro-Macro Dynamical Systems}
\author{Rafael D\'\i az \ \ and \ \ Sergio Villamar\'\i n}

\maketitle

\begin{abstract}
The second law of thermodynamics states that the entropy of an isolated system is almost always increasing. We propose
combinatorial formalizations of the second law  and explore their conditions of possibilities.
\end{abstract}

\section{Introduction}

The second law of thermodynamics is one of the pillars of modern science, enjoying a fundamental status comparable only
to that of the law of conservation of energy. The range of applicability of the second law extends well-beyond the confines
of its original formulation  within thermodynamics and statistical mechanics \cite{rf, j2, j3, op, es}, and for that
 reason we further refer
to it simply as the second law. Despite the efforts of many distinguished researchers a definitive mathematical
formulation of the second law,  as clear say as the symplectic geometry (Poisson brackets) formulation of the conservation
of energy law, has not yet been achieved. With this long term goal in mind  we propose combinatorial embodiments of the second law,
which give rise to  combinatorial problems interesting in their own right. \\

There are several equivalent formulations of  second law, as well as some formulations
whose equivalence is not fully understood. The reader will find in \cite{fr, gil, g, j1, l, lie, dr, ru} mathematically 
inclined introductions to the subject from quite different viewpoints.
Our departure point is the following formulation of  the Clausius' second law due to Boltzmann: \
"the entropy of an isolated system is almost always increasing."  To make sense of this statement several
precisions are in order:\\

\noindent - As formulated the second law    is meant to cover transitions between equilibria macrostates,
as well as transitions between non-equilibrium macrostates \cite{j15}.
The word "increasing" is taken in the weak sense, i.e. entropy tends to  grow or to remain constant.
A  system is isolated if it doesn't interchange neither
matter nor heat with its surroundings. We typically think of it as a system enclosed in an insulating box;
the universe as whole is an isolated system \cite{ca, rp}. Entropy is understood in the Boltzmann's sense, i.e. as  a logarithmic measure of the number of
micro-realizations of a macrostate. The Boltzmann approach to the second law has been studied by a number of authors, among them
\cite{cer, chi, de2, fr1, fr2, gol, gol2, j, lebo2, lebo, lebo3, lebo1}.  For equilibria macrostates Boltzmann and Clausius entropies
agree up to an additive constant, in the thermodynamic limit \cite{j, es}.\\

\noindent -   The restriction of the second law to transitions between equilibria macrostates is quite well-understood in the thermodynamic
limit.  A straightforward approach to this case has been developed by Jaynes \cite{j} where the key facts are, first,
the equilibrium macrostate can be identified with the probability distribution of maximum Shannon entropy under
the available constraints, and, second,  Bolztmann and Shannon entropies agree up to a positive multiplicative factor. The second law is thus a consequence of the
obvious fact that if constrains are lifted the maximum Shannon entropy increases.  Theorem \ref{j2l} provides a
combinatorial analogue for the latter argument within the Boltzmann entropy context. It is thus the increase of entropy for  transitions between non-equilibrium macrostates, and in particular
the transitions of such kind that arise as a non-equilibrium macrostate gradually approaches the
equilibrium, that remains an open  problem. The Clausius entropy is defined only for
equilibrium macrostates, so the actual problem is to extend the notion of entropy away from the equilibrium.
Bolztmann entropy is such an extension, and thus it make sense to ask under what conditions, if any, it possess the properties
expected from the second law within a combinatorial setting.\\

\noindent -  In contrast with energy,  an always conserved quantity, entropy increases almost surely.
Strictly decreasing entropy is not ruled out, on the contrary,
it is actually predicted by the second law; otherwise the word "almost" should be removed from
the law according to Occam's razor principle. Decreasing entropy is however,  according to the second law,
an extremely low probability event, turning it into a non-option for most practical purposes. The second law
itself does not provide bounds for the probability of decreasing entropy, neither  it provides estimates
for the  rate of entropy production.\\

We  also consider, see Section \ref{op}, combinatorial embodiments for stronger versions of the second law, such as
the following formulation  due to Gibbs in which the equilibrium plays a main role:
"entropy tends to strictly increase until the system reaches the equilibrium, i.e. the state of highest entropy,
and then it remains in the equilibrium for a very long period of time." Indeed we are going to propose several
combinatorial properties each covering some aspect of the second law;  we are not claiming that any of
this properties is the ultimate combinatorial formulation of the second law, but we do claim that understanding
these properties, both individually and collectively, provides a deep insight towards grasping the second
law as a combinatorial statement.\\

Although Boltzmann himself was aware of the combinatorial  nature of the
second law, arising from coarse-graining, it seems that this idea with powerful potential applications hasn't had the impact that
it deserves. The recent works of Niven \cite{ni, ni2} may be regarded as fundamental steps towards making the connection combinatorics/second
law more explicit. It is our believe that the combinatorial approach to the second law should be pursued in all its depth. In this
contribution, we lay down some foundational ideas, discuss the main problems of study, and establish some basic results in the
combinatorial approach to the second law. The main advantage of working within a combinatorial context is that we can rigorously define and compute with certain objects
whose higher dimensional analogues  may be elusive. For example, the set of invertible dynamical systems on a finite set
is just the set of permutations on it. Also combinatorial methods often lead to algorithms suitable for numerical computation,
allowing hypothesis and conjectures to be probed. Moreover, it is expected that by considering finite sets of
large cardinalities the combinatorial models can be used to understand infinite phase spaces. Thus combinatorial models
may be useful both as a conceptual guide, and as a computational tool for attacking the more involved cases allowing
infinitely many microstates. Along this work we argue that the combinatorial viewpoint  leads to a picture of the
second law as a subtle balance among six principles:\\

\noindent \textbf{Micro/Macro Duality.} Boltzmann entropy relies on the distinction between microstates and macrostates.
 The micro-macro divide gives rise to dual interpretations. An ontological interpretation where microstates are
 primordial entities,  and  macrostates are what the observer measures when the system is in a given microstate.
 This approach is often referred by phrases such as "subjective or anthropomorphic macrostates" \cite{j, j1, ro}.
 A phenomenological interpretation where macrostates are primordial, being what is actually accessible to the scientist, and microstates are
 theoretical constructs whose non-observable individual behaviour is postulated so that it gives rise to the observable
 behaviour of  macrostates. In Boltzmann's days it was microstates that were regarded as
 subjective or anthropomorphic, just as today some microstates beyond the standard model are often regarded
 as lacking an objective basis; in the last few years we have  witness the Higgs'  field transition
 from theoretical construct to experimental fact. It seems that the subjective/objective knowledge qualification
 correlates weakly with the micro/macro scale division.  The choice of interpretation leads to different but
 ultimately equivalent mathematical models, in their common domain of reference.\\

\noindent \textbf{Proportionality.} The idea is that probabilities are proportional to possibilities, the more microstates
  within a macrostate the higher its probability.  Entropy grows simply because the are more microstates with higher
  entropy than microstates with lower entropy. The proportionality principle may be thought as an application to microstates
  of the Laplace principle of insufficient reason: a probability is uniform unless we
  have reasons to claim the contrary. The probabilistic symmetry of microstates arises from the usual methodological division between
  law of motion and initial conditions, where a theory provides the evolution law for microstates but leaves
  the choice of initial microstate to the applied scientist. The proportionality principle is so
  intuitively appealing that it is tempting   to identify it with the second law itself, as some authors seem to do.
  However the further principles  introduced below show the need to complement and restrict the applicability of
  the proportionality  principle in order to understand  the second law.\\

  \noindent  \textbf{Large Differences.} In science once a scale is fixed the relevant numbers are
  often of comparable size. In the realm of the second law however the normal is just the opposite:
  huge differences in numbers, so pronounced indeed, that they are reminiscent of the mathematical distinction
  between measure zero and full measure sets. With huge differences
  low entropy microstates properties  likely have a negligible impact  on the global properties of a system.
  However simply disregarding low or decreasing entropy microstates
  is like disregarding the rational numbers  because they have zero Lebesgue measure. In the realm
  of large differences, small may be huge: suppose a microstate have probability   $10^{-10^{23}}$ of being
  non-equilibrium,  a probability so low that studying such microstates seems pointless; nevertheless if
  the total number of microstates is say $10^{10^{23} +10^{10}},$ then there are about  $10^{10^{10}}$ non-equilibrium
  microstates leaving plenty of room for interesting behaviour.   Assuming large differences  the main obstacle towards
  the "nowhere to go but up" effect are constant entropy microstates. Large differences imply
  a dominant equilibrium but in general it is a much stronger condition.  We talk about large rather than infinite
  differences, as one of the main aims of the combinatorial approach is to estimate the transition point where
  differences become dominant.\\

  \noindent  \textbf{Continuity.} Proportionality implies that starting from generic initial conditions
  the equilibrium will eventually be reached,   but against all empirical evidence, it also implies that at any time
  the most likely move for a microstate is to jump to the equilibrium. Unrestricted
  proportionality violates the law of gradual changes, a most cherished principle of physics. Continuity
  places restrictions on  proportionality in a couple of ways:
  it limits the allowed dynamics on microstates,  and it demands that macrostates couple to the dynamics in  such a way
  that sudden long jumps in entropy are unlikely, although not completely rule out.\\

  \noindent  \textbf{Microstates Asymmetry.}  Reversible systems have as many entropy decreasing
  as entropy increasing microstates, indeed this is the basic fact behind the  Loschmidt's
  paradox.  Within our combinatorial formalizations  of the second law  equal increases and decreases in
   entropy by itself does not give rise to contradictions, but it does point towards a fundamental fact:
  the second law is a sufficient reason to break the probabilistic   symmetry of microstates,
  a fact materialized with the introduction of not reversion invariant macrostates. As a rule one may expect
  the equilibrium to be reversion invariant, but it is  quite unnatural to demand this property for all macrostates;
  in particular entropy itself may not be invariant under reversion. The outshot is that any mathematical formalization of the second
  law must in some way or another  break the probabilistic symmetry of microstates. Microstates asymmetry
  plays a major role in our combinatorial renderings of the fluctuation theorems.\\

 \noindent  \textbf{Localization to Orbits.} From Gibbs' viewpoint  properties formalizing
  the second law should apply orbitwise, allowing a relative small number of microstates
  to live in badly behaved orbits. Again several more o less related reasonable properties
  may be proposed. As an example we are going to consider a particularly powerful one:
  the existence of a reversion invariant equilibrium such that most microstates on each orbit belong to
  the equilibrium. In such cases the equilibrium reaching time  is a strictly
  decreasing not reversion invariant function on non-equilibrium  microstates. Looking at the macrostates
  associated to this function one obtains, under reasonable hypothesis on  the image of the equilibrium down sets,
  a micro-macro dynamical systems with strictly increasing entropy on non-equilibrium macrostates, i.e. for such systems
  irreversibility arises naturally from reversibility,  the origin of any microstate is a low entropy microstate, and
  the longer the (past and future) history of a microstate, the lower the entropy of its origin.\\

Let us describe in details a standard construction given rise to combinatorial models from  familiar
smooth models through a couple of coarse-graining procedures.
Let $\ (M, \omega)\ $ be  a compact symplectic manifold
and $\  M  \longrightarrow B \ $ be a coarse-graining map  with $ \ B \ $ a finite set.  For $\ n \in \mathbb{N}_{\geq 1}\ $
the Hamiltonian map $\ H_n:M^n \longrightarrow \mathbb{R} \ $ generates the dynamics
$\ \phi_t: M^n \longrightarrow M^n \ $ via the identity
$\ \omega_n(\dot{\phi}, \ ) = dH_n,  \ $ where  $\ \omega_n \ $ is the product symplectic structure
on $M^n. \ $ For  $\ u \geq 0 \ $ consider the energy shell $\ H_n^{-1}(nu) \subseteq M^n \ $ and its image
$\ B_u^n \ $ under the  coarse graining map $\ M^n  \longrightarrow B^n. \ $ Assume we have  a second
coarse-graining map $\ \mathrm{prop}_B \longrightarrow A, \ $
where $\ \mathrm{prop}_B \ $ is the space of probability distributions on $\ B, \ $ and $\ A \ $ is another finite set.
We obtain  the chain of maps
$$ H_n^{-1}(nu) \ \longrightarrow  \ B_u^n \ \longrightarrow \
\mathrm{prop}_B  \ \longrightarrow A.\ $$
In this work we focus on the (composition) map
$\ B_u^n  \longrightarrow A \ $ since it only involves finite sets;
so $\ B_u^n \ $ will be our  set of microstates and $\ A \ $ will be our set of macrostates.
Under reasonable hypothesis $\ B_u^n \ $ inherits a measure and
a stochastic dynamics from the corresponding structures on $\ M_u^n  \ $ via
the map $\ M_u^n \longrightarrow  B_u^n .\ $ In this work however we  only consider the case where
the induced measure is uniform and the dynamics is deterministic. Although one should really
start with a stochastic dynamics on microstates we refrain to do so for several reasons.
First, it is worth it to see random processes  arising straight out of fully deterministic processes;
second, the deterministic case is interesting in itself and deserves its own study; third, studying
the deterministic case should be though as preparation for dealing with the more general stochastic case.\\

In Section \ref{mmds} we introduce micro-macro dynamical systems and formulate some of the main problems in
the combinatorial approach to the second law, e.g. counting the number of strict decreases in  entropy
for arbitrary permutations and partitions on finite sets.  The partition of microstates into macrostates gives
us the notion of Boltzmann entropy, and also a probability distribution and a stochastic dynamics
on macrostates. In Section \ref{atrs} we formalize  the notions of  reversible
micro-macro dynamical systems and  global arrow of time.   We provide a couple of  general construction showing that
there are plenty of (invariant, equivariant) reversible micro-macro dynamical systems, and provide formulae for
these systems.  We show that each (invariant, equivariant)  reversible micro-macro dynamical systems
can be canonically decomposed into four components, one coming from the constructions just mentioned, and the other
ones  quite easy to grasp. We also discuss fluctuation theorems \cite{sb, c, de1, de3, de2, es, s}  for combinatorial  micro-macro dynamical systems,
and study with a global arrow of time, with  emphasis on systems with the
same number of strict increases and strict decreases in entropy. \\

In Section \ref{spmms}  we review some of the structural operations on micro-macro dynamical system such as the product,
disjoint union, restriction, coarse-graining, meet and joint; and introduce five general constructions of
micro-macro phase spaces. We formulate an analogue of the asymptotic equipartition theorem applicable for
micro-macro phase spaces, and provided a couple of interesting examples of  coarse-graining.
In Section \ref{fsl} we consider the applicability,  within our combinatorial framework, of the
second  law with the world "almost" removed, i.e. we study invertible micro-macro dynamical
systems with no strictly decreasing entropy, and show that a generic system has  low probability of having this property.
This case is nonetheless interesting because we are able to fully explore for it the dual viewpoints:
the partition-based viewpoint where  macrostates are fixed and the dynamics vary, and the permutation-based
viewpoint where the dynamics is fixed and  macrostates vary. As the two viewpoints lead to
equivalent results, we obtain an interesting combinatorial identity. In Section \ref{mfsl} we
consider invertible  micro-macro dynamical systems with the highest possible number of strict decreases
in entropy. We introduce a sharp upper bound with a simple combinatorial meaning on
the number of such decreases, adopting a partition-based viewpoint, and
provide conditions on a partition implying that any  permutation coupled to it defines a
system satisfying a combinatorial formalization of the second law. \\

In Section \ref{sllip2} we introduce a pair of new combinatorial  formulations of the arrow of time,
define the jump of a map from a set provided with a partition to itself, and study
combinatorial formulations of the second law for zero jump systems. In Section \ref{op} we adopt Gibbs' viewpoint and
study  combinatorial formalizations the second law through properties localized to orbits. We introduce equilibrium bound systems and study the
equilibrium reaching time for such systems.  In Section \ref{sllip} we reformulate some of
the main problems in the combinatorial approach to the second law in terms of sums over integer points in convex
polytopes; this approach allows to fully analyze some simple but revealing cases and  opens the door
for numerical computations. In the final Section \ref{tl} we consider thermodynamic limits.
Although based on the previous sections,  readers familiar with maximum entropy methods may feel at home
with the techniques and results of this section.  At various points through out this work we  use the language
of category theory but only basic notions are required \cite{law, mc}.  \\

\section{Micro-Macro Dynamical Systems}\label{mmds}

Let $\ \mathrm{set}\ $ be the category of finite sets and maps, and $\ \mathrm{map}\ $ be the category of morphisms in $\ \mathrm{set}, \ $ i.e.
objects in $\ \mathrm{map}\ $ are functions between finite sets. A morphism in $\ \mathrm{map}\ $ from $\ f_1: X_1 \longrightarrow A_1\ $ to $\ f_2: X_2 \longrightarrow A_2\ $  is given by  maps $\ k:X_1 \longrightarrow X_2\ $ and
$\ \bar{k}:A_1 \longrightarrow A_2\ $  such that $\ \overline{k}f_1 = f_2k, \ $ i.e. the following diagram commutes
\[\xymatrix @R=.4in  @C=.8in
{X_1 \ar[r]^{k} \ar[d]_{f_1} &  X_2 \ar[d]^{f_2}
\\ A_1 \ar[r]^{\overline{k}} & A_2  } \]

\begin{defn}{\em \
\begin{enumerate}
  \item A micro-macro phase space is a tuple $\ (X,A,f) \ $ where  $\ f:X \longrightarrow A\ $ is a surjective map,
   $\ X \ $ is the set of microstates, and $\ A \ $ is the set of macrostates. If $\ f(i)=a\ $ we say that the microstate $\ i \ $ belongs to the macrostate $\ a, \ $ and write $\ i\in a\ $ instead of $\ i \in f^{-1}(a) \ $ whenever $ f  $ is understood.

  \item A micro-macro dynamical system is a tuple $\ (X,A, f, \alpha) \ $ where $\ (X,A,f) \ $ is a micro-macro phase space, and $\ \alpha:X \longrightarrow X\ $ is a map defining the dynamics on  microstates, i.e. it sends a microstate $\ i \ $ to the microstate $\ \alpha(i) \ $ in a unit of time. We let $\ \mathrm{mmds} \ $ be the category of micro-macro dynamical systems.

\item A micro-macro dynamical system  $\ (X,A, f, \alpha) \ $ is called invertible
 if the map  $\ \alpha:X \longrightarrow X\ $ is bijective. We let $\ \mathrm{immds} \ $ be the full subcategory of $\ \mathrm{mmds}\ $  whose objects are micro-macro dynamical systems with invertible dynamics.
 \end{enumerate}
}
\end{defn}

The category of micro-macro phase spaces is the full subcategory  $\ \mathrm{surj} \ $ of $\  \mathrm{map} \ $ whose objects are surjective maps. A morphism $ \  (X_1,A_1,f_1,\alpha_1)  \longrightarrow  (X_2,A_2,f_2,\alpha_2) \ $   in   $\  \mathrm{mmds} \  $ is  a morphism $\ k \ $ of micro-macro phase spaces such that the following diagram commutes: \[\xymatrix @R=.4in  @C=.8in
{X_1 \ar[r]^{k} \ar[d]_{\alpha_1} &  X_2 \ar[d]^{\alpha_2}
\\ X_1 \ar[r]^{k} & X_2  } \] Given a finite set $ \ X, \  $ we let $\ \mathrm{Par}X\ $ and $\ \mathrm{S}_X\ $ be,
respectively, the
 set of partitions and permutations on $\ X. \ $  A partition on $\ X \ $ is a family of non-empty disjoint subsets
 (called blocks) of $\ X \ $ with union equal to $\ X.\ $   A surjective map $\ f:X \longrightarrow A\ $ defines
 the partition on $\ X \ $ given by $\ \pi  =  \{f^{-1}(a) \ | \  a \in A \}, \ $
  which gives rise to the surjective map $ \  \overline{(\ )}:X  \longrightarrow  \pi  \  $
  sending $\ i \in X \ $ to the block  containing it.  The maps $\ f:X \longrightarrow A\ $ and $\ \overline{(\ )}:X \longrightarrow \pi\ $ are isomorphic objects in the category $\ \mathrm{surj}, \ $ i.e. we have a commutative isomorphism triangle
\[
\xymatrix @R=.3in  @C=.5in
{& X \ar[dl]_{f}\ar[dr]^{\overline{(\ )}} & \\
A \ar[rr]  & & \pi   }
\]
where the bottom arrow sends $\ a \ $ to $\ f^{-1}(a). \ $
So, up to isomorphism, a surjective map and a partition on $\ X\ $  define the same structure,
thus any   micro-macro dynamical system $ \ (X,A,f, \alpha)$ is isomorphic
to a micro-macro dynamical system of the form $\ (X, \pi, \overline{(\ )}, \alpha),\ $   henceforth denoted  by
$\ (X, \pi, \alpha).\ $ Next we show that $ \ \mathrm{immds} \ $ is a coreflective subcategory of $\ \mathrm{mmds}.$

\begin{prop}{\em \ The inclusion functor $\ i: \mathrm{immds} \longrightarrow \mathrm{mmds}\ $ has a right
 adjoint functor $\ i_{\ast}: \mathrm{mmds} \longrightarrow \mathrm{immds}\ $ given on objects by
$ \  i_{\ast}(X,A,f,\alpha) =  (X_{r}, f(X_r),f,\alpha) , \ $
where  $\ X_r \ $ is the set  of recurrent microstates $\ \{ i \in X \ | \ \alpha^n(i)=i \ \  \mbox{for some} \ \ n>0 \}, \ $
and the restrictions of $\ f \ $ and $ \ \alpha \ $ to  $ \ X_r \ $ are denoted with the same symbols.}
\end{prop}

\begin{proof}
We need to show that there is a natural bijection
$$\mathrm{immds}\big((X_1,A_1,f_1,\alpha_1), i_{\ast}(X_2,A_2,f_2,\alpha_2) \big) \ \simeq \
 \mathrm{mmds}\big((X_1,A_1,f_1,\alpha_1), (X_2,A_2,f_2,\alpha_2) \big).$$
Indeed if $\ \alpha_1 \ $ is invertible, then the image of a morphism in the right-hand
set above is necessarily contained in $ \ (X_2)_r, \ $ and therefore it is also a morphism in the left-hand set above.
\end{proof}

\begin{defn}\label{bs}
{\em Let $\ (X,A,f,\alpha) \ $ be a micro-macro dynamical system.  The following structures arise on $ \ A \ $ and $\ X\ $:
\begin{enumerate}
  \item A measure $\ |\ |:A \longrightarrow \mathbb{N}\ $ given by $\ |a|  = |f^{-1}(a)|, \ $ inducing
  the measure $\ |\ |:X \longrightarrow \mathbb{N}\ $ given by $\ |i|=|f(i)|, \ $ and
  the probability measure $\ p:A \longrightarrow  [0,1] \ $ given by $\ p_a   =  \frac{|a|}{|X|}.$
  \item A stochastic map $\ T:A \longrightarrow A \ $ with $\ T_{ab} \in [0,1] \ $ giving the probability
  that a macro-state $\ b \ $ moves to a macro-state $\ a \ $ in a unit of time
 $$T_{ab}  \ = \  \frac{ |\{i\in b\ | \ \alpha(i) \in a \}|}{|b|}. \ \ \ \ \mbox{Note that}\ \ T_{ab}\geq 0 \ \ \
  \mbox{and}
\ \ \ \sum_{a\in A}T_{ab}=1.\ $$
\item  The uniform probability on $\ A \ $ and the uniform probabilities on each block
   $\ a\in A \ $ induce the probability measure $\ q \ $ on $\ X \ $ given by $\ q(i)=\frac{1}{|A||i|}.$

\end{enumerate}}
\end{defn}

\begin{defn}{\em Let $\ (X,A,f,\alpha)\ $ be a micro-macro dynamical system.

\begin{enumerate}
  \item The Boltzmann entropy on macrostates $\ S:A \longrightarrow \mathbb{R}\ $ is given by
  $\ S(a) =   \mathrm{ln}|a|  .$
  \item The Boltzmann entropy  on microstates $\ S:X \longrightarrow \mathbb{R} \ $ is given by
  $\ S(i)   =   \mathrm{ln}|i|.$
  \item The Boltzmann entropy of $\ A \ $ is given by
  $ \ \displaystyle S(A) =   \sum_{a\in A}S(a)p_a  =   \frac{1}{|X|}\sum_{i \in X}S(i).  $
  \item The Shannon entropy of $\ p \ $ is given by
  $\ \displaystyle  H(p)=   -\sum_{a\in A} \mathrm{ln}(p_a)p_a  =   \mathrm{ln}|X|   -   S(A)  .$
  \item The Shannon entropy of the stochastic map $\ T: A \longrightarrow A\ $ is given by
  $$ \displaystyle H(T)   =  \sum_{b \in A}H(T_{\bullet b})p_b   =   -\sum_{a,b \in A}\mathrm{ln}(T_{ab})T_{ab}p_b .$$
\end{enumerate}
}
\end{defn}

\begin{rem}
{\em Boltzmann's actual definition of the entropy of a macrostate is  $\ k\mathrm{ln}|a|, \ $ where $\ k\ $ is the
Boltzmann constant. For simplicity we  set $\ k=1,\ $ or equivalently, work with the logarithmic function
$\  \mathrm{log}_{e^{1/k}}(x).\ $ The  entropy of a partition with respect to an automorphism has been studied
in ergodic theory \cite{aa}. Although related to our constructions,  we will not use this notion.
Shannon entropy and Boltzmann entropy play complementary roles as
$$ \  H(p) \ + \ S(A)     \ = \  \mathrm{ln}|X|.  \ $$ Shannon entropy $ \ H(p)\  $ measures the mean
uncertainty in choosing  a macrostate. Boltzmann entropy $\ S(A) \ $ measures the mean uncertainty in choosing
a microstate given that a macrostate has already been chosen.
}
\end{rem}

Given a micro-macro phase space $\ (X,A, f) \ $ the set of equilibria macrostates
$\ A^{\mathrm{eq}} \subseteq A \ $ is the set of macrostates with maximum Boltzmann entropy.
In the applications, usually $ \ A^{\mathrm{eq}} \ $ has a unique element called the equilibrium.
We let $\  X^{\mathrm{eq}}   \subseteq X  \ $ be the set of microstates in an equilibrium macrostate,
and $ \ X^{\mathrm{neq}} = X \setminus  X^{\mathrm{eq}}\ $ be the set of non-equilibrium microstates.
For $\ L \subseteq X \ $ set $\ L^{\mathrm{eq}} = L \cap X^{\mathrm{eq}} \ $ and
$\ L^{\mathrm{neq}} = L \cap X^{\mathrm{neq}}. $ \\

Let $ (X,A,f, \alpha) $ be a micro-macro dynamical system. The sets of microstates where entropy
is decreasing, increasing, and constant are respectively given by:
$$ D   =   DX   =   \{i \in X \  |  \ S(\alpha(i)) < S(i) \}, \ \ \ \ I   =   IX   =   \{i \in X \  | \  S(\alpha(i)) > S(i)\}, \ $$
$$C  =   CX   =  \{ i \in X \  | \  S(\alpha(i)) = S(i)  \}. \ $$
More generally, for $\ L  \subseteq  X \ $ set $ \  DL=D\cap L, \ \ IL=I\cap L, \ $ and $ \ CL=C \cap L.\ $
We have that
$$ \frac{|DL|}{|L|} \ + \ \frac{|IL|}{|L|} \ + \ \frac{|CL|}{|L|} \ = \ 1 .$$

Throughout this work we use a parameter $\ \varepsilon \ $  allowed to be in the interval $\ [0,1] \ $
in order to exhaust all logical possibilities, but meant to be a fairly small positive real number.
Next we introduce our first formalization of the second law.   We will subsequently provide further formalizations demanding stronger conditions making the
  systems more closely resemble those likely to be relevant in nature.

\begin{defn}\label{1m}
{\em A micro-macro dynamical system $\ (X,A,f, \alpha)\ $ satisfies  property $\ \mathrm{L}_1(\varepsilon) \ $
  if and only if  $\ \ \displaystyle \frac{|D|}{|X|}  \leq  \varepsilon, \ \ $ and in this case we write
  $\ (X,A,f, \alpha) \in  \mathrm{L}_1(\epsilon).\ $ A  sequence $\ (X_n, A_n, f_n, \alpha_n)\ $ of micro-macro dynamical systems satisfies  property $\ \mathrm{L}_1, \ $
  and we write $\ (X_n, A_n, f_n, \alpha_n) \in \mathrm{L}_1, \ $ if for any $\ \varepsilon > 0 \ $ there exits $\ N \in \mathbb{N}\ $ such
  that $\ (X_n, A_n, f_n, \alpha_n)\in \mathrm{L}_1(\varepsilon) \ $ for $\ n \geq N.$
 }
\end{defn}

Next result allows to understand property $\ \mathrm{L}_1(\epsilon)\ $  in terms of macrostates.
Given $\ b \in A \ $ we set
$\ A_{<b} = \{a\in A \ | \ |a| < |b| \} \ \ $ and $\ \ A_{\leq b} = \{a\in A \ | \ |a| \leq |b| \} .\ $

\begin{prop}\label{nt}
{\em Let $\ (X,A,f, \alpha) \ $  be a micro-macro dynamical system. We have that:
\begin{enumerate}
  \item
$ \displaystyle \frac{|D|}{|X|}   =   \sum_{S(a)<S(b)}T_{ab} p_b, \ \ \ \ \ \ \frac{|I|}{|X|}  =
\sum_{S(a)>S(b)}T_{ab} p_b,  \ \ \ \ \ \mbox{and} \ \ \ \ \ \frac{|C|}{|X|}  =
\sum_{S(a)=S(b)}T_{ab} p_b. \ \ \ \ $
  \item $(X,A,f, \alpha) \in \mathrm{L}_1(\varepsilon) \ \ $ if and only if $\ \ \displaystyle \sum_{S(a)<S(b)}T_{ab} p_b \ \leq\  \varepsilon.$
\item If $ \displaystyle \ T_{ab} \leq  \frac{\varepsilon}{|A_{<b}|} \ $ for $\ |a| < |b|, \ $ then $\ (X,A,f, \alpha) \in \mathrm{L}_1(\varepsilon). \ $
\end{enumerate}
}
\end{prop}

\begin{proof}Item 2 is a direct consequence of item 1. We show the leftmost identity in item 1:
 {\footnotesize $$\sum_{S(a)<S(b)}T_{ab}p_b  \   =   \ \sum_{S(a)<S(b)}\frac{ |\{i\in b\ | \ \alpha(i) \in a \}|}{|b|}\frac{|b|}{|X|} \ =  \ \frac{1}{|X|}\Big| \coprod_{S(a)<S(b)}\{i\in b\ | \ \alpha(i) \in a \} \Big| \  =  \ \frac{|D|}{|X|}.$$}
Item 3 is shown as follows:
$$\frac{|D|}{|X|} \  =  \ \sum_{S(a)<S(b)}T_{ab} p_b \   \leq  \ \varepsilon \sum_{S(a)<S(b)} \frac{p_b}{|A_{<b}|}  \  =  \  \varepsilon \sum_{b} \frac{|A_{<b}|p_b}{|A_{<b}|} \  \leq  \ \varepsilon.$$
\end{proof}

According to Jaynes \cite{de2, j} a transition  on macrostates $ \ b \rightarrow a \ $
is experimentally reproducible if and only if $\ T_{ab}\ $ is nearly equal to $1$,
meaning that the images under $\ \alpha\ $ of almost all microstates in $\ b \ $
 lie in the macrostate $\ a.\ $ Accordingly, the stochastic map $\ T \ $ is experimentally
  reproducible if and only if it is nearly deterministic, i.e. if and only if there
  is a map $\ t:A   \longrightarrow A\ $ such that  $\ T_{t(b)b}\ $ is nearly equal to $\ 1, \ $
  say $\ T_{t(b)b} \geq  1- \epsilon \ $ for $\  \varepsilon \geq 0 \ $ fairly small.

\begin{prop}{\em Let $\ (X,A,f, \alpha)\ $ be an invertible micro-macro dynamical system.
\begin{enumerate}
  \item If the entropy  $\ H(T)\ $ of the stochastic map $\ T :A \longrightarrow A \ $ is nearly vanishing,
   then  property $\ \mathrm{L}_1(\varepsilon) \ $ holds for suitable $\ \varepsilon >0 \ $ specified below.
  \item If $\ H(T)=0, \ $ then property $\ \mathrm{L}_1(0) \ $ holds.
\end{enumerate}
}
\end{prop}

\begin{proof}
Under the hypothesis of item 1, $\ \alpha \ $ induces a map $\ \alpha:A \longrightarrow A\ $
such that  $\ T_{\alpha(b)b}\ $ is nearly equal to $\ 1, \ $ say $\ \ T_{\alpha(b)b} \geq   1- \varepsilon \ $ with $\  \varepsilon > 0 \ $ fairly small. Then
$$\frac{|\alpha(b)|}{|b|}\ \geq  \ \frac{ |\{i\in \alpha(b)\ | \ \alpha^{-1}(i) \in b \}|}{|b|} \  =
 \ \frac{ |\{i\in b\ | \ \alpha(i) \in \alpha(b) \}|}{|b|}\  =  \ T_{\alpha(b)b} \ \geq  \ 1 - \varepsilon,$$
and thus $\ S(\alpha(b)) =   \mathrm{ln}|\alpha(b)|  \geq  \mathrm{ln}|b| +  \mathrm{ln}(1-\epsilon)  =   S(b)  + \mathrm{ln}(1-\varepsilon).  \ $
Assuming in addition  that $\ \varepsilon \  $ is small enough  that $\ S(a)<S(b) \ $ implies that $\ S(a)<S(b)+ \mathrm{ln}(1-\varepsilon) ,\ $
then by Proposition \ref{nt} we have that:
$$ \frac{|D|}{|X|}\   =  \ \sum_{S(a)<S(b)}T_{ab} p_b \ \leq \ \varepsilon \sum_{b} p_b \ = \ \varepsilon .$$
Item 2 follows from item 1, since in this case we can actually set $\ \varepsilon=0.$
\end{proof}

The following result is a direct consequence of the definitions.

\begin{lem}\label{at}
{\em Let $\ (X,A,f,\alpha)\ $ be an invertible micro-macro dynamical system. We have that
1) $\alpha  D  =  I_{\alpha^{-1}}, \ \ \alpha I  =  D_{\alpha^{-1}}, \ \ \alpha C =  C_{\alpha^{-1}} .\ \ $
2) $|D|  =  |I_{\alpha^{-1}}|, \ \ |I| =  |D_{\alpha^{-1}}|, \ \   | C|  =  |C_{\alpha^{-1}}| .\ \ $
3) $ |D|  +   |D_{\alpha^{-1}}|   +    |C| =    |X|. \ $
}
\end{lem}

Next we show that a micro-macro dynamical system and its inverse satisfy property
 $ \ \mathrm{L}_1(\epsilon)\ $  if and only if entropy is nearly constant.

\begin{prop}
{\em Let $\ (X,A,f, \alpha)\ $ be an invertible micro-macro dynamical system.
\begin{enumerate}
  \item If $\ (X,A,f,\alpha) \in \mathrm{L}_1(\varepsilon_1) \ $ and
$\ (X,A,f,\alpha^{-1}) \in \mathrm{L}_1(\varepsilon_2), \ $ then $\ \displaystyle \frac{|C|}{|X|}\geq   1- \varepsilon_1 - \varepsilon_2.$
  \item If $\  \displaystyle \frac{|C|}{|X|}\geq  1- \varepsilon, \  $ then  $\ (X,A,f,\alpha) \in \mathrm{L}_1(\varepsilon) \ $ and
$\ (X,A,f,\alpha^{-1}) \in \mathrm{L}_1(\varepsilon). \ $
\item If $\ |I|=|D|, \ $ then $\ (X,A,f,\alpha) \in \mathrm{L}_1(\varepsilon)\ \ $ if and only if  $ \ \ \displaystyle \frac{|C|}{|X|} \geq   1- 2\varepsilon.$
\end{enumerate}
}
\end{prop}

\begin{proof}Recall that $\ |D_{\alpha^{-1}}|  =  |I|. \ $ The hypothesis of item 1 implies that
$$\varepsilon_1 \ + \ \varepsilon_2 \ + \ \frac{|C|}{|X|} \ \geq \
 \frac{|D|}{|X|} \ + \  \frac{|I|}{|X|} \ + \ \frac{|C|}{|X|} \ = \ 1.$$ Under the hypothesis of item 2 we  have that
$ \ \displaystyle \frac{|D|}{|X|} \ + \  \frac{|D_{\alpha^{-1}}|}{|X|} \ \leq \ \varepsilon .\ $ Item 3 follows from
the identity $\ \displaystyle 2\frac{|D|}{|X|} \ + \ \frac{|C|}{|X|} \ = \ 1.$
\end{proof}

The set of isomorphism classes of invertible micro-macro dynamical systems on a set
 $\ X \ $ of micro-states can be identified with the quotient set $\ (\mathrm{Par}X\times S_X)/S_X \ $ where:
\begin{enumerate}
\item $\beta \in S_X\ $ acts on  $ \ \pi \in \mathrm{Par}X \ $  by   $ \   \beta\pi{} =  \{\beta a \ | \  a\in{}\pi{} \}.$
\item $\beta \in S_X \ $ acts on $\alpha \in S_X \ $  by conjugation $ \  \beta(\alpha)  =  \beta\alpha \beta^{-1}.$
\item $S_X\ $ acts diagonally on $\ \mathrm{Par}X\times S_X.$
\end{enumerate}
Isomorphic invertible micro-macro dynamical systems have strictly decreasing entropy sets of
the same cardinality, thus we get the map
$$|D|: (\mathrm{Par}X\times S_X)/S_X  \longrightarrow  [0,d_X] \ \ \ \   \mbox{given by} \ \ \ \
|D|(\pi,\alpha)   =  |D_{\pi, \alpha}| = \big|\{\ i \in X \ | \  S(\alpha(i))< S(i) \ \}\big|,$$
where $\ d_X \ $ is maximum number of strict decreases for a micro-macro dynamical system on $ X. \ $
We show in Section \ref{mfsl}  that $\ \displaystyle d_X =  |X|  - \underset{l \vdash |X|}{\mathrm{min}}\
\underset{1 \leq i \leq |X|}{\mathrm{max}}\ il_i, \ $
where $l$ runs over the numerical partitions of $|X|: \ l=(l_1,...,l_{|X|}) \ $  and
  $ \ \displaystyle \sum_{i=1}^{|X|}il_i  =  |X|. \ $
The first 40 entries of the sequence $\ d_n \ $ are:
$\ 0, \ 0, \ 1, \ 2, \ 2, \ 3,  \ 4, \ 4, \ 5, \ 6, \ 7, \ 8, \ 9, \ 10, \ 11, \ 11, \ 12, \ $
$13, \ 14, \ 15, \ 16, \ 16, \ $$
17, \ 18, \ 19, \ 20, \ 21,  \ 22, \ 23, $
$\ 24, \ 25, \ 26, \ 27, \ 27, \ 28, \ 29, \ 30, \ 31, \ 32, \ 33. \ $
In Section \ref{fsl} we consider micro-macro dynamical systems with always increasing entropy, i.e. systems in $\ |D|^{-1}(0). \ $  In Section \ref{mfsl} we consider  micro-macro dynamical systems with the maximum number  of strict decreases  allowed, i.e. systems in $\ |D|^{-1}(d_X). \ $

\begin{lem}\label{me}{\em
The uniform probability on $\ \mathrm{Par}X\times S_X\ $ induces a probability  on $\ (\mathrm{Par}X\times S_X)/S_X\ $ for which the expected value of $\ |D| \ $ is given by
$$\overline{|D|} \  =  \ \frac{1}{|X|!B_{| X|}}\sum_{(\pi, \alpha)\in \mathrm{Par}X\times S_X}|D_{\pi,\alpha}|,$$
where $\ B_n \ $ are the Bell numbers given by
$\ \ \displaystyle B_{n+1}=\sum_{k=0}^{n}{n \choose k}B_k\ \ $ and $\ \ B_1=1.$}
\end{lem}

\begin{exmp}{\em  For $\ X=[3] \ $  there are three non-uniform partitions $ \ 1|23, \ 2|13, \ 3|12, \ $
and for each of these partitions there are four permutations with $\ |D|=1. \ $ All other choices lead
to$\ D = \emptyset. \ $  Therefore
$\ \displaystyle \overline{\frac{|D|}{|X|}} =  \frac{12}{90}=0.13 . \ $ See Figure \ref{t}.

\begin{figure}[t]
    \centering
    \includegraphics[width=11cm, height=5cm]{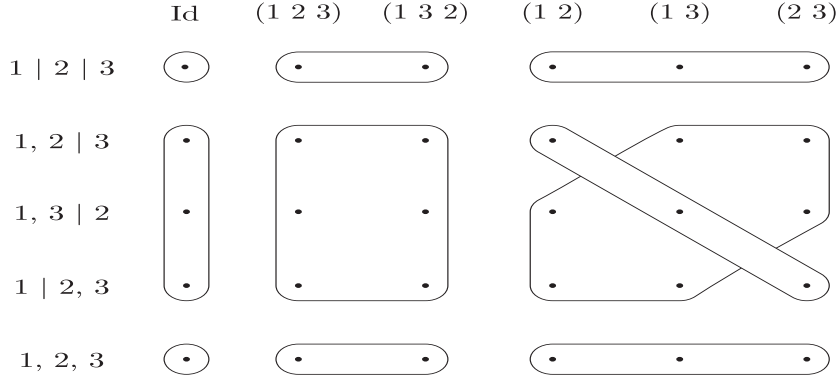}
    \caption{Orbits of the $\ S_3$-action on  $\ \mathrm{Par}[3]\times S_3. \ $}
    \label{t}
\end{figure}
}
\end{exmp}

Next we show that, in average, a random invertible micro-macro dynamical system has as many strict increases as strict decreases in entropy.

\begin{thm}{\em   The random variables
$\ \displaystyle |I| \  $ and $\  \displaystyle |D| \ $ on $\ (\mathrm{Par}X\times S_X)/S_X\ $ have the same mean.}
\end{thm}

\begin{proof}Consider the uniform probability on $\ \mathrm{Par}X\times S_X. \ $ The result follows from  Lemma \ref{me} since
$\ \ |D_{\pi,\alpha}|  =   |I_{\pi,\alpha^{-1}}|\ \ \ \mbox{for} \ \ \ (\pi,\alpha) \in \mathrm{Par}X\times S_X. \ $
\end{proof}

Next we show that whenever $\ \pi \ $ has a dominant equilibrium, then property
$\ \mathrm{L}_1(\varepsilon) \ $ holds for all invertible systems of the form $\  (X, \pi, \alpha).$

\begin{thm}\label{i1}
{\em  Let $\ (X, \pi) \ $ be a micro-macro phase space and
$\ \delta_1 \leq  \delta_2 \leq \varepsilon (\delta_1+1)\ $ in $ \ \mathbb{R}_{\geq 0}\ $ be such that
 $\ \delta_1 |X^{\mathrm{eq}}|  \leq  |X^{\mathrm{neq}}|  \leq \delta_2 |X^{\mathrm{eq}}| . \ $
 Then
 $\  (X, \pi, \alpha) \in \mathrm{L}_1(\varepsilon) \ $ for any permutation $\ \alpha \in \mathrm{S}_X.$
}
\end{thm}

\begin{proof}Under the given hypothesis we have that
$$\frac{|D|}{|X|} \ = \ \frac{|\alpha D|}{|X|} \ \leq \ \frac{|X^{\mathrm{neq}}|}{|X^{\mathrm{neq}}|  + |X^{\mathrm{eq}}|}
\ \leq \ \frac{\delta_2|X^{\mathrm{eq}}|}{(\delta_1+1)|X^{\mathrm{eq}}|} \  \leq \ \varepsilon.$$
\end{proof}

\section{Reversible Systems and the Arrow of Time}\label{atrs}

In this section we begin to formalize the arrow of time concept  \cite{ca, lebo2, rp} within our
combinatorial framework; stronger formalizations will be developed subsequently. We also introduce reversible
micro-macro dynamical systems and study some of their main properties. Let $\ (X,A,f,\alpha) \ $ be a  micro-macro dynamical system,  $ i \in X, \  $ and  $ \ N, M \in \mathbb{N}. \ $
Entropy defines a $\ [N,M]$-arrow of time around $ \ i \ $ of time length $\ N+M+1, \ $ if it is strictly increasing
at the microstates   $\ \alpha^{n}(i)\ $ for all $\ n \in  [-N,M].\ $
In this case the $\ \alpha$-orbit of $\ i \ $ must have at least $\ N+M+2 \ $ elements. As with the second law itself, it is convenient to introduce a less strict condition for the arrow
of time, allowing a relative small number of decreases or constant entropy among the microstates
$\ \alpha^{n}(i). \ $
\begin{defn}\label{att}
{\em  Entropy defines an $\ [\varepsilon,N,M]$-arrow of time around $\ i \ $ of time length $N+M+1$  if the following inequality holds
$$ \frac{ \big|\big\{n\in [-N,M] \ \big| \  S(\alpha^{n+1}(i))  \leq  S(\alpha^{n}(i))  \big\} \big|}{N+M+1} \
  \leq \  \varepsilon.\ $$
}
\end{defn}

Definition \ref{att} can be extended for other functions on $\ X \ $ in place of entropy,
in particular, one can apply it to negative entropy.  The foregoing considerations motivate our next definition.

\begin{defn}{\em   Entropy defines a global $\varepsilon$-arrow of time on  $\ (X,A,f,\alpha),\ $ and we write
  $\ (X,A,f,\alpha)\in \mathrm{GAT}(\varepsilon), \ $ if
$$\frac{|DX^{\mathrm{neq}} \sqcup  CX^{\mathrm{neq}} | }{|X^{\mathrm{neq}}|} \ \leq \
 \varepsilon, \ \ \ \ \mbox{or \ equivalently} \ \ \ \
  \frac{| IX^{\mathrm{neq}} | }{|X^{\mathrm{neq}}|}  \geq   1 -  \varepsilon.$$
}
\end{defn}

\

\begin{thm}{\em Let $\ (X,A,f,\alpha)\ $ be a micro-macro dynamical system such that $\ T_{ab}  \leq
\frac{\varepsilon}{|A_{\leq b}|} \ $ whenever $\ |a| \leq |b| \ $ and $ \ b\in  A^{\mathrm{neq}},\ $ then
$\ (X,A,f,\alpha)\in \mathrm{GAT}(\varepsilon). $
}
\end{thm}
\begin{proof} We have that
$$\frac{| DX^{\mathrm{neq}} \sqcup  CX^{\mathrm{neq}} |}{|X^{\mathrm{neq}}|} \   =
 \ \frac{|X|}{|X^{\mathrm{neq}}|} \sum_{S(a)\leq S(b), \ b \in A^{\mathrm{neq}} }T_{ab} p_b \  \leq  \
 \frac{\varepsilon |X|}{|X^{\mathrm{neq}}|} \sum_{S(a)\leq S(b), \ b \in A^{\mathrm{neq}}}\frac{p_b}{|A_{\leq b}|}
\ = $$
$$ \frac{\varepsilon |X|}{|X^{\mathrm{neq}}|}\sum_{b \in A^{\mathrm{neq}}} \frac{|A_{\leq b}|}{|A_{\leq b}|}p_b \  =
 \ \frac{\varepsilon |X||X^{\mathrm{neq}}|}{|X^{\mathrm{neq}}||X|} \  =   \  \varepsilon.$$
\end{proof}

\begin{defn}
{\em Let $ \ \varepsilon_1, \varepsilon_2 \in [0,1].\ $
 A  micro-macro dynamical system $\ (X,A,f, \alpha)\ $ satisfies property $\ \mathrm{L}_2(\varepsilon_1,\varepsilon_2), \ $  and we write $\ (X,A,f, \alpha) \in  \mathrm{L}_2(\varepsilon_1,\varepsilon_2),\ $ if it satisfies $ \mathrm{L}_1(\varepsilon_1) $ and $ \mathrm{GAT}(\varepsilon_2).$
 A  sequence $ (X_n, A_n, f_n, \alpha_n)$ of micro-macro dynamical systems satisfies property $\ \mathrm{L}_2 ,\ $
  and we write $\ (X_n, A_n, f_n, \alpha_n)\in \mathrm{L}_2, \ $ if for any $\ \varepsilon_1, \varepsilon_2 > 0 \ $ there exits $\ N \in \mathbb{N}\ $ such
  that $\ (X_n, A_n, f_n, \alpha_n) \in  \mathrm{L}_2(\varepsilon_1,\varepsilon_2)\ $ for $\ n \geq N.$
 }
\end{defn}

Let $(X, \pi)$ be micro-macro phase space.  The partition $\ \pi \ $ induces another partition $\ Z \ $ of $ \ X \ $ into zones.
For $\ k\in \mathbb{N}_{\geq 1} \ $ let $\ \pi_k \subseteq \pi\ $ be the subset of $\ \pi \ $ consisting of
all blocks of cardinality $\ k, \ $ and let the $k$-zone $\ \widehat{\pi}_k \subseteq  X \ $ be given by
$$\widehat{\pi}_k \  =  \ \bigsqcup_{a\in \pi_k}a.$$
We set $\ O_{\pi}=\{k \in \mathbb{N}_+ \ | \ \widehat{\pi}_k \neq \emptyset  \}. \ $
A zone of $\ X \ $ is a subset of the form  $\ \widehat{\pi}_k \ $ for $\ k\in O_{\pi}. \ $
Typically we write $\ O_{\pi}=\{k_1 < \cdots < k_o \} \ $ and set $\ \widehat{\pi}_j = \ \widehat{\pi}_{k_j} \ $
for $\ j \in [o].$

\begin{thm}{\em  Let $\ (X, \pi) \ $ be a micro-macro phase-space with $\ O_{\pi}=\{k_1 < \cdots < k_o \} \ $ and
 $\ \delta_1,  \delta_2,   \delta_3, \delta_4, \delta_5 \in \mathbb{R}_{\geq 0}\ $ be such that
 \begin{enumerate}
   \item $ \displaystyle  \delta_1 |X^{\mathrm{eq}}|  \leq  |X^{\mathrm{neq}}|
     \leq \delta_2 |X^{\mathrm{eq}}| \ \ \  \mbox{with} \ \ \ \delta_2 \leq \varepsilon_1(\delta_1+1),$

   \item $ \displaystyle   \delta_3 |\widehat{\pi}_{o-1}|\leq  \sum_{i=1}^{o-2}|\widehat{\pi}_{i}|
   \leq \delta_4 |\widehat{\pi}_{o-1}|\ \ \  \mbox{with} \ \ \ \frac{\delta_4}{\delta_3+1} +
       \delta_5 \leq \varepsilon_2,$

 \end{enumerate}

then  $\ (X, \pi, \alpha) \in \mathrm{L}_2(\varepsilon_1,\varepsilon_2 )\ $ for any permutation
 $\ \alpha \in \mathrm{S}_X \ $  with $\ |CX^{\mathrm{neq}}| \leq \delta_5|X^{\mathrm{neq}}|.$
}
\end{thm}

\begin{proof} Theorem \ref{i1} implies that
$\ \displaystyle \frac{|D|}{|X|}  \leq  \varepsilon_1. \ $
The desired result follows from
$$\frac{|DX^{\mathrm{neq}} \sqcup CX^{\mathrm{neq}}|}{|X^{\mathrm{neq}}|} \ \leq \
\frac{\sum_{i=1}^{o-2}|\widehat{\pi}_{i}|}{\sum_{i=1}^{o-2}|\widehat{\pi}_{i}|  +
|\widehat{\pi}_{o-1}|} + \frac{|CX^{\mathrm{neq}}|}{|X^{\mathrm{neq}}|} \ \leq \
 \frac{\delta_4|\widehat{\pi}_{o-1}|}{(\delta_3+1)|\widehat{\pi}_{o-1}|} + \delta_5 \  \leq \ \varepsilon_2.$$

\end{proof}

Next we show that if an invertible system has as many strict increases as strict decreases
 in entropy, and satisfies  $\ \mathrm{L}_2(\varepsilon_1,\varepsilon_2), \ $ then most microstates
  are equilibrium microstates with constant entropy. Note that if $\ |D|=|I|,\ $ then
  $\ \mathrm{GAT}(\varepsilon)\ $ implies
  $\ \mathrm{L}_1(\varepsilon)\ $ since
$\ \displaystyle \frac{|D|}{|X|}  \leq  \frac{|I|}{|X^{\mathrm{neq}}|}  \leq  \varepsilon.$
\begin{thm}\label{jj}
{\em Let $\ (X,A,f,\alpha) \ $ be an invertible micro-macro dynamical system such that $\ |D|= |I| \ $ and $\ (X,A,f,\alpha) \in \mathrm{L}_2(\varepsilon_1,\varepsilon_2). \ $ We have that:
\begin{enumerate}
  \item $ \displaystyle \frac{|X^{\mathrm{eq}}|}{|X|}  \geq  1  -  \frac{\varepsilon_1}{1-\varepsilon_2}.$
  \item $\displaystyle \frac{|CX^{\mathrm{eq}}|}{|X|}  \geq  1  -  \frac{\varepsilon_1(2 - \varepsilon_2)}{1-\varepsilon_2}.$
  \item $\displaystyle (1-2\varepsilon_2)|X^{\mathrm{neq}}|  \leq   |DX^{\mathrm{eq}}| \leq |X^{\mathrm{neq}}|.$
  \item $\displaystyle (1-2\varepsilon_2)|X^{\mathrm{neq}}|  \leq   |\{i \in X^{\mathrm{neq}} \ | \ \alpha(i) \in X^{\mathrm{eq}} \}|  \leq   |X^{\mathrm{neq}}|.$
\end{enumerate}

}
\end{thm}

\begin{proof} From hypothesis we have that $ \displaystyle (1-\varepsilon_2)|X^{\mathrm{neq}}| \leq   |D|
 \leq  \varepsilon_1 |X| $  which implies item 1. From item 1 and the identity
  $$\frac{|X^{\mathrm{eq}}|}{|X|} \ = \ \frac{|DX^{\mathrm{eq}}|}{|X|}  +  \frac{|CX^{\mathrm{eq}}|}{|X|}, \ \ \  \mbox{we get that} \ \ \ 1  -  \frac{\varepsilon_1}{1-\varepsilon_2} \ \leq \ \frac{|CX^{\mathrm{eq}}|}{|X|} \ +\ \varepsilon_1 ,$$ which is equivalent to the inequality from item 2. Item 3 follows from $$(1 -  \varepsilon_2)|X^{\mathrm{neq}}| \ \leq \ | IX^{\mathrm{neq}} | \ = \  | I |  \ = \  | D |  \ = \ |DX^{\mathrm{eq}} |  +  | DX^{\mathrm{neq}} | \ \leq \   |DX^{\mathrm{eq}}|  +   \varepsilon_2 |X^{\mathrm{neq}}|.$$ Item 4 follows from the identity $ \ |DX^{\mathrm{eq}}|  =  |\{i \in X^{\mathrm{neq}} \ | \ \alpha(i) \in X^{\mathrm{eq}} \}|. \  \ $Indeed we have a bijection  $ \  DX^{\mathrm{eq}} \longrightarrow  \{i \in X^{\mathrm{neq}} \ | \ \alpha(i) \in X^{\mathrm{eq}} \} \ $ sending $\ i \in DX^{\mathrm{eq}}\ $ to $\ \alpha^{e-1}(i)\ $ where $\ e \ $ is the least positive integer with $\ \alpha^{e}(i) \in X^{\mathrm{eq}}. \ $ The inverse map sends $\ i \ $ to $\ \alpha^{l}(i),\ $ where $\ l \ $ is the largest positive integer such that $\ \alpha^l(i) \in X^{\mathrm{eq}}.$
\end{proof}

\

\begin{thm}\label{kk}
{\em Let $ \ (X,A,f, \alpha)\ $ be an invertible micro-macro dynamical system.
\begin{enumerate}
  \item If $\ |X^{\mathrm{eq}}| \geq  (1 - \varepsilon_1)|X|,  \ $ then $ \ (X,A,f, \alpha) \in   \mathrm{L}_1(\varepsilon_1).\ $
  \item If  $\ | DX^{\mathrm{eq}}|  \geq   (1 - \varepsilon_2)|X^{\mathrm{neq}}|, \ $ then $ \ (X,A,f, \alpha) \in \mathrm{GAT}(\varepsilon_2). $
  \item If $\  |X^{\mathrm{eq}}| \geq  (1 - \varepsilon_1)|X|  \ $ and  $ \ | DX^{\mathrm{eq}}|  \geq   (1 - \varepsilon_2)|X^{\mathrm{neq}}|, \ $  then $\ (X,A,f, \alpha) \in \mathrm{L}_2(\varepsilon_1,\varepsilon_2). $
\end{enumerate}
}
\end{thm}

\begin{proof} If $\ \displaystyle |X^{\mathrm{eq}}| \geq (1 - \varepsilon_1)|X|, \ $ then
$\  \displaystyle \frac{|D|}{|X|} \ =  \ \frac{| \alpha D|}{|X|} \ \leq \ \frac{|X^{\mathrm{neq}}|}{|X|} \ \leq \
\varepsilon_1. \ $
Suppose now that $\ | DX^{\mathrm{eq}}| \geq  (1 - \varepsilon_2)|X^{\mathrm{neq}}|,\ $ then
$$\ |IX^{\mathrm{neq}}| \ \geq \ |\{i\in X^{\mathrm{neq}} \ | \ \alpha(i) \in X^{\mathrm{eq}} \}| \ =
 \ | DX^{\mathrm{eq}}|  \ \geq \ (1 - \varepsilon_2)|X^{\mathrm{neq}}|.$$
\end{proof}

\begin{defn}\label{rs}
{\em \ A reversible micro-macro dynamical system is a tuple $ \ (X,A,f,\alpha,r) \ $  such that
 $\ r:X \longrightarrow X \ $ is an involution $ \ (r^2 =  1),\ $ and
$ \ r \ $ conjugates $\ \alpha\ $ and $\ \alpha^{-1} \ $  $\ ( r\alpha r=\alpha^{-1}).\ $
We say that  $ \ (X,A,f,\alpha,r) \ $ is invariant if $\ fr=f;\ $
equivariant if $\ r \ $ induces an involution  $ r:A \longrightarrow A \ $ such that $\ fr=rf;\ $ entropy
preserving if $\ S(ri)=S(i) \ $ for $\ i \in X.$
}
\end{defn}

\

\begin{lem}
{\em Let $ (X,A,f,\alpha,r)  $ be a reversible micro-macro dynamical system,
$\ a\in A,\ k \in \mathbb{N}_{\geq 1}. \ $
\begin{itemize}
 \item The system is invariant iff  $\ r \ $ induces bijections $\ r: a \longrightarrow a;\ $ the system is
 equivariant iff $\ r \ $
 induces an involution  $ r:A \longrightarrow A \ $ together with bijections $\ r: a \longrightarrow ra;\ $ the system is entropy preserving iff $\ r \ $ induces bijections
$\ r: \widehat{\pi}_k \longrightarrow \widehat{\pi}_k.\ $
\item The reversion map $\ r: X \longrightarrow X\ $ induces stochastic maps $\ r_A: A \longrightarrow A \ $
and $\ r_Z: Z \longrightarrow Z \ $ on macrostates and zones, respectively. We have that
$\ r \ $ is invariant iff $\ r_A \ $ is the identity map;
$\ r \ $ is equivariant iff $ \ H(r_A)= 0; \ $ and   $\ r \ $ is entropy preserving
iff $ \ H(r_Z)= 0. \ $
\end{itemize}
}
\end{lem}

\

\begin{prop}\label{cc}
{\em Let  $\ (X,A,f,\alpha,r) \ $ be an entropy preserving reversible micro-macro dynamical system, then $ \ |I|=|D|.$ }
\end{prop}

\begin{proof} We show that $\ r\alpha I  =  D,\ $ using that $\ r \ $ preserves entropy a couple of times.
We have that $\ \alpha I  =   D_{\alpha^{-1}} =  D_{r\alpha r}=  D_{\alpha r}, \ $ so
the desired result follows from the identity $\ rD_{\alpha r} =  D.\ $
\end{proof}

Let $\ \mathrm{rmmds} \ $ be the category of reversible micro-macro dynamical systems. A morphism
$\ (X_1,A_1,f_1,\alpha_1, r_1)  \longrightarrow (X_2,A_2,f_2,\alpha_2, r_2) \ \  \mbox{in}\ \ \mathrm{rmmds} \ $
is a morphism $\ k\ $ of the subjacent micro-macro dynamical systems such that the following diagram commutes
\[\xymatrix @R=.4in  @C=.8in
{X_1 \ar[r]^{k} \ar[d]_{r_1} &  X_2 \ar[d]^{r_2}
\\ X_1 \ar[r]^{k} & X_2  } \]

Our next constructions show that there are plenty of reversible systems, indeed we associate an invariant reversible micro-macro
dynamical system, and an equivariant reversible micro-macro
dynamical system to each invertible micro-macro dynamical system. Let $\ \mathrm{irmmds} \ $ and $\ \mathrm{ermmds} \ $ be the full
subcategories of $\ \mathrm{rmmds} \ $ whose objects are, respectively, invariant and equivariant reversible micro-macro dynamical system.
We have inclusion functors $\ \mathrm{irmmds} \longrightarrow \mathrm{ermmds} \longrightarrow \mathrm{rmmds}, \ $ and the forgetful functor
$\ u:\mathrm{rmmds} \longrightarrow \mathrm{immds}\ $
given  by
$\ u(X,A,f,\alpha,r)  =  (X,A,f,\alpha) . \ $ Set  $\ \mathbb{Z}_2 = \{1,-1 \}.$

\begin{thm}\label{cf}
{\em  The forgetful functor $\  u:\mathrm{irmmds} \longrightarrow \mathrm{immds} \ $ has a left adjoint functor \\
$IR: \mathrm{immds} \longrightarrow \mathrm{irmmds}\ $ given by
$ \  IR(X,A,f,\alpha)  =  (X\times \mathbb{Z}_2, A, f\pi_X,\widehat{\alpha}, r) \  $  where:
\begin{itemize}
  \item $\ \pi_X:X\times \mathbb{Z}_2 \longrightarrow X \ $ is the projection to $\ X$;
  \item  $\ r:X\times \mathbb{Z}_2 \longrightarrow X\times \mathbb{Z}_2\ $ is given by $\ r(i,s) =  (i,-s);$
  \item  $\ \widehat{\alpha}:X\times \mathbb{Z}_2 \longrightarrow X\times \mathbb{Z}_2\ $ is given by  $\ \widehat{\alpha}(i,s)  =  (\alpha^s(i),s). $
\end{itemize}
We have that:
\begin{enumerate}
\item $|(f\pi_X)^{-1}(a)|=2|f^{-1}(a)| \ $  for $\ a \in A, \ \ $ and  $\ \ S(X\times \mathbb{Z}_2, A,f\pi_X) =   S(X,A,f)  +  \mathrm{ln}(2) .$
\item $  |D_{\widehat{\alpha}}| = |I_{\widehat{\alpha}}|   =  |D_{\alpha}| + |I_{\alpha}|\  $ and
$  \ |C_{\widehat{\alpha}}| =  2|C_{\alpha}|.$
\item $2T_{ab}(\widehat{\alpha})  =  T_{ab}(\alpha)  +  T_{ab}(\alpha^{-1}).$
  \item $\widehat{\alpha} \in \mathrm{L}_1(\varepsilon_1) \  \ $ iff $\ \  |D_{\alpha}| + |I_{\alpha}|  \leq
     2\varepsilon_1 |X| \ \ $
iff $\  \  |C_{\alpha}|   \geq   (1-  2\varepsilon_1) |X|.$
  \item $\widehat{\alpha} \in \mathrm{GAT}(\varepsilon_2) \ \ $ iff $\ \  |D_{\alpha}| + |I_{\alpha}|
    \geq 2(1-\varepsilon_2) |X^{\mathrm{neq}}| \ \ $  iff $\ \ |C_{\alpha}|  \leq
|X^{\mathrm{eq}}|  -(1- 2\varepsilon_2) |X^{\mathrm{neq}}|.$
  \item $ \widehat{\alpha} \in \mathrm{L}_2(\varepsilon_1, \varepsilon_2)
\  \ \ \ \mbox{iff} \ \ \ \ 2(1-\varepsilon_2) |X^{\mathrm{neq}}|  \leq  |D_{\alpha}| +  |I_{\alpha}|\leq
  2\varepsilon_1 |X|\ \ \ \ $  iff \\
$ (1-2\varepsilon_1) |X| \leq    |C_{\alpha}|  \leq
|X^{\mathrm{eq}}|  -(1- 2\varepsilon_2) |X^{\mathrm{neq}}|.$
\end{enumerate}
}
\end{thm}
\noindent Figure \ref{rds} shows an invertible micro-macro dynamical system and its associated invariant
 reversible micro-macro dynamical system.

\begin{figure}[t]
    \centering
    \includegraphics[width=14cm, height=5cm]{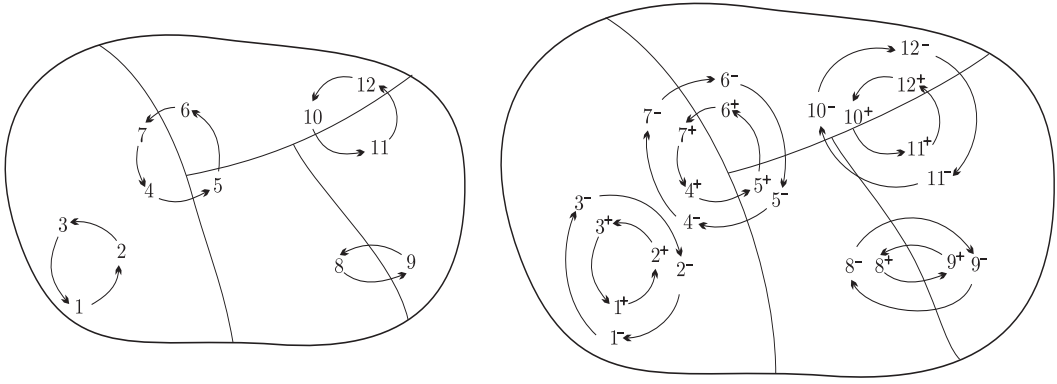}
    \caption{Invertible system and its associated invariant reversible system.}
    \label{rds}
\end{figure}

\begin{thm}\label{df}
{\em  The forgetful functor $\  u:\mathrm{ermmds} \longrightarrow \mathrm{immds} \ $ has a left adjoint functor \\
$ER: \mathrm{immds} \longrightarrow \mathrm{ermmds}\ $ given by
$ \  ER(X,A,f,\alpha)  =  (X\times \mathbb{Z}_2, A\times \mathbb{Z}_2, f\times 1,\widehat{\alpha}, r) \  $  where:
\begin{itemize}
\item  $\ f\times 1:X\times \mathbb{Z}_2 \longrightarrow X\times \mathbb{Z}_2\ $ is given by $\ f\times1(i,s) =  (fi,s);$
  \item  $\ r:X\times \mathbb{Z}_2 \longrightarrow X\times \mathbb{Z}_2\ $ is given by $\ r(i,s) =  (i,-s);$
  \item  $\ \widehat{\alpha}:X\times \mathbb{Z}_2 \longrightarrow X\times \mathbb{Z}_2\ $ is given by  $\ \widehat{\alpha}(i,s)  =  (\alpha^s(i),s). $
\end{itemize}
We have that:
\begin{enumerate}
\item $|(f\times 1)^{-1}(a,s)|=|f^{-1}(a)| \ $  for $\ a \in A, \ \ $ and  $\ \  S(X\times \mathbb{Z}_2, A,f\times 1) =   S(X,A,f).$
  \item Properties 2,4,5,6  \ from Theorem \ref{cf} hold.
  \item $T_{(a,1)(b,1)}(\widehat{\alpha})  =  T_{ab}(\alpha),\ $
  $\ T_{(a,-1)(b,-1)}(\widehat{\alpha})  =  T_{ab}(\alpha^{-1}),\ \ $ and $\ \ T_{(a,s)(b,-s)}(\widehat{\alpha}) =0.\ $
  \item $IR(X,A,f,\alpha) \in \mathrm{L}_2(\varepsilon_1, \varepsilon_2)\ $ if and only if
  $\ ER(X,A,f,\alpha) \in \mathrm{L}_2(\varepsilon_1, \varepsilon_2).$
\item $IR(X,A,f,\alpha) \in \mathrm{L}_3(\varepsilon_1, \varepsilon_2)\ $ if and only if
  $\ ER(X,A,f,\alpha) \in \mathrm{L}_3(\varepsilon_1, \varepsilon_2).$
\end{enumerate}
}
\end{thm}

Next result follows from Proposition \ref{cc} and Theorems \ref{jj}, \ref{kk}, \ref{cf}, \ref{df}.

\begin{prop}\label{ter}
{\em Let $\ (X, \pi) \ $ be a micro-macro phase space with
$\ |X^{\mathrm{eq}}|  \geq   (1-\varepsilon_1)  |X^{\mathrm{neq}}| \ $, and consider the (invariant, equivariant)
associated reversible system $\ R(X,\pi,\alpha).\ $ If $ \ R(X,\pi,\alpha) \in
\mathrm{L}_2(\varepsilon_1, \varepsilon_2), \ $ then
$\ | D_{\alpha}X^{\mathrm{eq}}| + | D_{\alpha^{-1}}X^{\mathrm{eq}}| \geq  2 (1 - 2\varepsilon_2)|X^{\mathrm{neq}}|.\ $
The system    $ \ R(X,\pi,\alpha) \in \mathrm{L}_2(\varepsilon_1, \varepsilon_2) \ $ for any permutation
$\ \alpha \in \mathrm{S}_X\ $ such that
$\  | D_{\alpha}X^{\mathrm{eq}}| + | D_{\alpha^{-1}}X^{\mathrm{eq}}| \geq  2 (1 - \varepsilon_2)|X^{\mathrm{neq}}|.$
}
\end{prop}

Below we introduce three methods for constructing invariant  reversible micro-macro dynamical
systems out of an $A$-colored disjoint union of linearly order sets:\\

\noindent 1) Let $ \ (L,A,f) \ $ be such that $ \ L \ $ is a non-empty poset obtained
 as a finite disjoint union $\ \coprod L_c \ $  of  linearly ordered sets $\ L_c \ $
  with $ \ c \in C, \ $ and $ \ f: L \longrightarrow A\ $ is a  map.
Let $\ n(L,A,f,g) \ $  be the invariant reversible micro-macro dynamical system given by
$ \ (L\times \{1,-1\},  A, \widetilde{f}, \alpha, r)   \  \mbox{where:}$

\begin{itemize}
  \item $\widetilde{f} = f\pi_L \ $ on $\ L\times \{1,-1\}, \ $ $\ r(l,1)=(l,-1),\ $  and $\ r(l,-1)=(l,1).$
  \item $ \alpha$-orbits with cyclic order
  $\ L_c\times \{1\} \  \sqcup \ L_c^{\mathrm{op}}\times \{-1\}  .$
\end{itemize}

\noindent 2) Consider  $ \ (L,A,f,g) \ $ with $ \ (L,A,f) \ $  as in item 1 with $\ L \ $ allowed to be empty,
 and $\ g: C \longrightarrow A\ $ another  map.
Let $\ o(L,A,f,g) \ $ be the invariant reversible micro-macro dynamical system given by
 $\  (  L\times \{1,-1\}\sqcup \{o_c \}_{c \in C},  A, \widetilde{f}, \alpha, r)\ \ \  \mbox{where:}$

\begin{itemize}
  \item $\widetilde{f} = f\pi_L \ $ on $\ L\times \{1,-1\},  \ \ \widetilde{f}(o_c) = g(c), $ $\
  r(o_c)=o_c,\ $  $ r(l,1)=(l,-1), \ $  and $\ r(l,-1)=(l,1).$
  \item $ \alpha$-orbits with cyclic order
  $\ o_c \  \sqcup   \ L_c\times \{1\}  \ \sqcup \ L_c^{\mathrm{op}}\times \{-1\}  .$
\end{itemize}

\noindent 3) Consider $ \ (L,A,f,g,h) \ $ with $\ (L,A,f,g) \  $ as in item 2 and $\ h: C \longrightarrow A\ $ another map.
Let $\ t(L,A,f,g,h)=(  L\times \{1,-1\}\sqcup \{o_c, t_c \}_{ c \in C },  A, \widetilde{f}, \alpha, r) \ $ be
the invariant reversible micro-macro dynamical system  given by
\begin{itemize}
  \item $\widetilde{f} = f\pi_L \ $ on $\ L\times \{1,-1\}, \ $ $\ \widetilde{f}(o_c) = g(c), \ $ and
   $\ \widetilde{f}(t_c) = h(c). \ $
  \item $r(o_c)=o_c,\ \ r(t_c)=t_c, \ $  $\ r(l,1)=(l,-1), \ $  and $\ r(l,-1)=(l,1).$
  \item $ \alpha$-orbits with cyclic order
  $\ o_c \ \ \sqcup \  \ L_c\times \{1\}\ \ \sqcup \ \  t_c \ \ \sqcup \ \ L_c^{\mathrm{op}}\times \{-1\}  .$
\end{itemize}

The  constructions above can be modified to yield equivariant reversible micro-macro dynamical
systems out of an $A$-colored disjoint union of linearly order sets, where now $\ A \ $ is a set provided with
an involution map $\ r:A \longrightarrow A .\ $ One proceeds as in the previous case demanding
that macrostates in the image of the maps $\ g \ $  and $\ h \ $ be fixed by $ \ r.  \ $ We denote
by $\ \tilde{n}, \  \tilde{o}, \ \tilde{t}\ $ the resulting equivariant systems.

\begin{thm}{\em \
\begin{enumerate}
  \item An invariant reversible  system $\ (X,A,f,\alpha,r) \ $ is isomorphic to a system of the form
$$  IR(X_1,A,f,\alpha)\  \sqcup \ n(X_2,A,f)  \ \sqcup \ o(X_3,A,f,f)\  \sqcup \ t(X_4,A,f,f,f).   $$
  \item An equivariant reversible  system $\ (X,A,f,\alpha,r) \ $ is isomorphic to a system of the form
$$  ER(X_1,A,f,\alpha)\  \sqcup \ \tilde{n}(X_2,A,f)  \ \sqcup \ \tilde{o}(X_3,A,f,f)\  \sqcup \ \tilde{t}(X_4,A,f,f,f).   $$
\end{enumerate}
  }
\end{thm}

\begin{proof}The result follows by an orbitwise analysis. Once the required properties are checked for microstates, the corresponding
properties for  macrostates follow as well, in both cases.  Since $\ r \ $ is an involution its cycles have length one or
 two, inducing a partition
of $\ X \ $ in two blocks
$\ \widehat{r}_1 \ $ and $\ \widehat{r}_2, \ $ defined as the union of blocks of the respective cardinality.
For a cycle $\ c \ $ of $\ \alpha \ $ the following possibilities arise:
\begin{enumerate}
\item $c \subseteq   \widehat{r}_2 \ $ and $\ c \cap r(c) = \emptyset. \ $ The map $\ r \ $ induces an involution without fixed points on the set
 such cycles. Choose a cycle  for each match pair and let $\ X_1 \ $ be the reunion
of such microstates.
\item $c \subseteq   \widehat{r}_2 \ $ and $\ c \cap r(c) \neq \emptyset. \ $ In this case necessarily $\ c = r(c) \ $ and
$\ r \ $ induces a matching on $\ c. \ $
Choose for each such cycle a maximal segment of $\alpha$-orbit with unmatched points, and let $\ X_2 \ $ be the reunion
of such microstates.
\item $|c \cap \widehat{r}_1|=1.$ Away from the fixed point
 (to be identified with $o_c$) $\ r \ $ defines a matching on $c$. Choose for each such cycle a maximal segment of $\alpha$-orbit with unmatched not fixed points,  and let $\ X_3 \ $ be the union  of microstates.
\item $|c \cap \widehat{r}_1|=2.$ Away from the pair of fixed points
 (to be identified with $o_c$ and $t_c$), $\ r \ $ defines a matching on $c$. Choose for each such cycle a maximal segment of $\alpha$-orbit with unmatched not fixed points,  and let $\ X_4 \ $ be the union  of microstates.
\end{enumerate}
Suppose that an $\ \alpha$-cycle $\ c \ $  has a $r$-fixed point $\ i, \ $  then
$\ r\alpha^s(i)= \alpha^{-s}(i) \ $ for $\ s>0. \ $
It follows that there should be a  minimum  $\ s >0 \ $ for which either $\ \alpha^s(i)= r\alpha^{s-1}(i) \ $ or
$\ r\alpha^s(i)= \alpha^{s}(i) \ $ (exclusively). If the former condition holds we are in case 3,
and  if the latter condition holds we are in case .
\end{proof}

Statements 1 and 2 of our next result provide the substrate of the Loschmidt's paradox within  the combinatorial
framework, note the subtle asymmetry in statements 3 and 4.

\begin{thm}\label{loc}
{\em Let  $\ (X,A,f,\alpha,r) \ $ be a entropy preserving reversible micro-macro dynamical system  and fix $ N, M \geq 0 . $
\begin{enumerate}
  \item Entropy defines an $\ (\varepsilon,N,M)$-arrow
of time around $\ i \in X \ $ if and only if  negative entropy defines  an $(\varepsilon,M,N)$-arrow
of time around $\ r\alpha(i) \in X. \ $
  \item $ \big| \{i \in X \ | \  (\varepsilon,N,M)\mbox{-arrow
of time around}\ \ i \  \} \big|  \ = \\
\big| \{i \in X \ | \ \mbox{reversed} \ (\varepsilon,M,N)\mbox{-arrow
of time around} \ \ i \big|.$
\item  $\big| \{i \in X^{\mathrm{neq}} \ | \   (\varepsilon,N,M)\mbox{-arrow
of time around} \ \ i \ \mbox{ and } \ \alpha^{M+1}(i) \in  X^{\mathrm{neq}}\} \big| \ = $ \\
$\big| \{i \in X^{\mathrm{neq}} \ | \ \mbox{reversed} \ (\varepsilon,M,N)\mbox{-arrow
of time around}\ \  i \}\big|.$
  \item $ |DX^{\mathrm{neq}}| =  |\{i \in IX^{\mathrm{neq}} \ | \ \alpha(i) \in  X^{\mathrm{neq}}\}| .$

\end{enumerate}
}
\end{thm}

\begin{proof}
Since entropy is preserved by reversion, it is strictly increasing along the $\alpha$-sequence
$$\alpha^{-N}(i) \rightarrow \cdots  \rightarrow \alpha^{-1}(i) \rightarrow i \rightarrow
 \alpha^{1}(i) \rightarrow \cdots \rightarrow \alpha^{M}(i) \rightarrow  \alpha^{M+1}(i),$$
 if and only if it is strictly decreasing along the $\alpha$-sequence
$$r\alpha^{M+1}(i) \rightarrow \cdots  \rightarrow r\alpha^{1}(i) \rightarrow ri \rightarrow r\alpha^{-1}(i)
\rightarrow \cdots \rightarrow r\alpha^{-N+1}(i) \rightarrow r\alpha^{-N}(i) .$$
\end{proof}

Our combinatorial settings  provide a  straightforward approach to the next results
known collectively as fluctuation theorems \cite{sb, c, de1, de3, de2, es, s}.

\begin{thm}\label{ft}
{\em \

\begin{enumerate}
  \item  Let $\ (X,A,f, \alpha) \ $ be an invertible micro-macro dynamical system, then
$$ T_{a,b}\  = \ e^{S(a)-S(b)}T_{b,a}(\alpha^{-1}) .$$
 \item Let  $\ (X,A,f,\alpha,r) \ $ be a reversible micro-macro dynamical system, then
$$T_{a,b} \  =  \ e^{S(a)-S(b)}T_{r(b),r(a)} .$$

\end{enumerate}

}
\end{thm}

\begin{proof} For item 1 set $\ l =  |\{i \in b \ | \ \alpha(i) \in a \}|  =   |\{i \in a \ | \ \alpha^{-1}(i) \in b \}|.\ $ We have that
$$T_{a,b}(\alpha) \ = \ \frac{l}{|b|} \ = \ \frac{|a|}{|b|}\frac{l}{|a|}  \ = \ e^{S(a)-S(b)}T_{b,a}(\alpha^{-1}).$$
Item 2 follows from item 1 and  the identity  $  \ T_{b,a}(\alpha^{-1})  =  T_{r(b),r(a)}, \ $ indeed we have that
$$ \ \big|\{i \in a \ | \ \alpha^{-1}(i) \in b \}\big|\  =  \ \big|\{i \in a \ | \ \alpha r(i) \in r(b) \}\big|\  =
\ \big|\{i \in r(a) \ | \ \alpha(i) \in rb \}\big|.\ $$ Note that $\ r(a) \ $ need not be a macrostate.
\end{proof}

\begin{defn}{\em Let $(X,A,f, \alpha)  $ be a micro-macro dynamical system.  The $n$-steps entropy production rate is the map $\ \sigma_n:X \longrightarrow \mathbb{R}\ $ given by
$$\sigma_n(i) \ = \ \frac{S(\alpha^ni)-S(i)}{n} \ =  \ \frac{1}{n}\mathrm{ln}(\frac{|\alpha^ni|}{|i|}). $$
}
\end{defn}

Let  $\ (X,A,f, \alpha, r)\  $ be an entropy preserving reversible system. Below we consider three probability distributions
on $\ X: \ $ uniform probability $\ u,  \ $ uniform probability on non-equilibrium microstates, also denoted by $ \  u, \ $  and the  probability $\ q \ $ from Definition \ref{bs}
given by $\ q_i=\frac{1}{|A||i|}.\ $ Note that $\ q_i > \frac{1}{|X|}\ $ if and only if $\ S(i) < \mathrm{ln}(\frac{|X|}{|A|}), \ $
i.e. $\ q\ $ assigns to low entropy microstates a probability higher than the uniform probability, moreover,
the lower the entropy the higher the $\ q$-probability of a microstate.
For $\ i \in X  \ $ we have that
$$\sigma_n(r\alpha^ni)= - \sigma_n(i) \ \ \ \ \ \mbox{and} \ \ \ \ \ |r\alpha^ni|= e^{n\sigma(i)}|i|. $$
The density functions $\ \mathrm{W}_n^u, \ \mathrm{W}_{n, \mathrm{neq}}^u, \ \mathrm{W}_n^q: \mathbb{R}\longrightarrow [0,1]\ $ associated to
 the uniform probability, the uniform probability on non-equilibrium microstates, and $\ q\ $ via
the map $\ \sigma_n\ $ are given  by
$$ \mathrm{W}_n^u(x)= \frac{|\{i \ | \ \sigma_n(i)=x \}|}{|X|}, \ \
 \mathrm{W}_{n, \mathrm{neq}}^u =  \frac{|\{i \in X^{\mathrm{neq}}\ | \ \sigma_n(i)=x \}|}{|X^{\mathrm{neq}}|},
 \ \  \mbox{and}  \ \
\mathrm{W}_n^q(x)=  \sum_{\sigma_n(i)=x}\frac{1}{|A||i|}.  $$
The statement and proofs of Theorem \ref{cr1}-3 and Theorem \ref{ie}-3  below are combinatorial
renderings of the arguments given by Deward and Maritan \cite{de2}.
\begin{thm}\label{cr1}
{\em  Let $\ (X,A,f, \alpha,r) \ $ be an entropy preserving reversible system.  For $\ x\in \mathbb{R}_{\geq 0} \ $
 we have that: $\ \displaystyle 1) \ \mathrm{W}_n^u(x) = \mathrm{W}_n^u(-x). \  $
  $\ \displaystyle 2) \mathrm{W}_{n, \mathrm{neq}}^u(x) \geq \mathrm{W}_{n, \mathrm{neq}}^u(-x). \ $
  $\ \displaystyle  3) \mathrm{W}_n^q(x) = e^{nx}\mathrm{W}_n^q(-x) .$
}
\end{thm}

\begin{proof} For item 1 we have that
$$\mathrm{W}_n^u(x)\ = \ \frac{|\{i \ | \ \sigma_n(i)=x \}|}{|X|}\  = \ \frac{|\{i \ | \ \sigma_n(r\alpha^ni)=x \}|}{|X|}
\ = \ \frac{|\{i \ | \ \sigma_n(i)=-x \}|}{|X|}\ = \ \mathrm{W}_n^u(-x).  $$
The item 2 statement is trivial for $\ x=0,\ $ so we assume $ \ x>0. \ $ We have that
$$\mathrm{W}_{n, \mathrm{neq}}^u(x)\ = \ \frac{|\{i \in X^{\mathrm{neq}}\ | \ \sigma_n(i)=x \}|}{|X|}\  = \ $$
$$ \frac{|\{i \in X^{\mathrm{neq}} \ | \ \sigma_n(i)=x, \ \ r\alpha^ni \in X^{\mathrm{neq}} \}|}{|X^{\mathrm{neq}}|}
\ + \ \frac{|\{i \in X^{\mathrm{neq}}\ | \ \sigma_n(i)=x, \ \ r\alpha^ni \in X^{\mathrm{eq}} \}|}{|X^{\mathrm{neq}}|}\ =  $$
$$ \frac{|\{i \in X^{\mathrm{neq}} \ | \ \sigma_n(i)=-x \}|}{|X^{\mathrm{neq}}|}
\ + \ \frac{|\{i \in X^{\mathrm{eq}}\ | \ \sigma_n(i)=-x \}|}{|X^{\mathrm{neq}}|}\ =  $$
$$ \mathrm{W}_{n, \mathrm{neq}}^u(-x)\ + \ \frac{|\{i \in X^{\mathrm{eq}}\ | \ \sigma_n(i)=-x \}|}{|X^{\mathrm{neq}}|}\ \geq \
 \mathrm{W}_{n, \mathrm{neq}}^u(-x) .  $$
For item 3 we have that
$$\mathrm{W}_n^q(x)\ = \ \sum_{\sigma_n(i)=x}\frac{1}{|A||i|} \  = \
\sum_{\sigma_n(r\alpha^ni)=x}\frac{1}{|A||r\alpha^ni|}\ = $$
$$\sum_{\sigma_n(i)=-x}\frac{e^{-n\sigma_n(i)}}{|A||i|} \ = \
e^{nx}\sum_{\sigma_n(i)=-x}\frac{1}{|A||i|}= e^{nx}\mathrm{W}_n(-x) .$$
\end{proof}

Below we consider the $n$-steps entropy production rate mean value
$\ \overline{\sigma}_{n}^u, \ \overline{\sigma}_{n, \mathrm{neq}}^{u}, \ \overline{\sigma}_{n}^q  \ $
with respect to the uniform probability,
the uniform probability over non-equilibrium microstates, and the $\ q \ $ probability on $\ X.$

\begin{thm}\label{ie}
{\em Let $ \ (X,A,f, \alpha) \ $ be a micro-macro dynamical system, then  $$ \displaystyle n\overline{\sigma}_{n}^u   =
 S_{T(\alpha^n)p}(A)  -  S_p(A) \ $$ where $\ T(\alpha^n) \ $ is the stochastic map on $\ A \ $ induced by
$\ \alpha^n.$
\begin{enumerate}

  \item  If $\ \alpha\ $ is invertible, then $\  \ \overline{\sigma}_{n}^u   =  0. $
\item If $\ \alpha\ $ is invertible, then $\ \ \displaystyle \overline{\sigma}_{n, \mathrm{neq}}^{u}
 \ = \ \frac{1}{n|X^\mathrm{neq}|}\sum_{i\in D_{\alpha^n}(X^{\mathrm{eq}})}[S(i) - S(\alpha^ni)] \ \geq \ 0,\ \ $
 and $\ \ \overline{\sigma}_{n, \mathrm{neq}}^{u} =0 \ \ $ if and only if $\  X^\mathrm{eq} \ $ is $\ \alpha^n$-invariant.
\item For an entropy preserving reversible  system $\ (X,A,f, \alpha, r) \ $, we have that
$\ \ \overline{\sigma}_{n}^q \geq 0 \ $ and $\  \overline{\sigma}_{n}^q = 0 \ $ if and only if  $\ \alpha^n\ $ preserves entropy.
\end{enumerate}
}
\end{thm}

\begin{proof} Recall that $\ p \ $ and $\ T(\alpha^n)p\ $ are the probability measures on $\ A \ $
given by $\ \displaystyle p_a  =   \frac{|a|}{|X|} \ $ and $ \ \displaystyle   (T(\alpha^n)p)_a  =
 \sum_{b \in A}T_{ab}(\alpha^n)p_b. \ $
We have that
$$n\overline{\sigma}_{n}^u  \  =  \ \sum_{i\in X}\frac{S(\alpha^n(i))  -  S(i)}{|X|} \  =
 \  \sum_{i\in b, \ \alpha^n(i) \in a}\frac{S(a)-S(b)}{|X|}\  = $$
$$ \sum_{a,b}[S(a)-S(b)]\frac{|\{i\in b \ | \ \alpha^n(i) \in a \}|}{|b|}\frac{|b|}{|X|} \  =
\ \sum_{a,b}[S(a)-S(b)]T_{ab}(\alpha^n)p_b \ = $$
$$\sum_{a,b}S(a)T_{ab}(\alpha^n)p_b   -   \sum_{a,b}S(b)T_{ab}(\alpha^n)p_b \  = \ S_{T(\alpha^n)p}(A) - S_{p}(A).$$
If $\ \alpha \ $ is invertible, then
$$\overline{\sigma}_{n}^u\  = \ \frac{1}{n|X|}\sum_{i\in X}[S(\alpha^ni)-S(i)]\ = \
 \frac{1}{n|X|}\bigg[\sum_{i\in X}S(\alpha^ni)  -
\sum_{i\in X} S(i)\bigg] \   = \ 0 .$$
Furthermore we  have that
$$\overline{\sigma}_{n,\mathrm{neq}}^u\  = \
\frac{1}{n|X^\mathrm{neq}|}\sum_{i\in X^\mathrm{neq}}[S(\alpha^ni)-S(i)]\ \ = $$
$$\frac{|X|}{n|X^\mathrm{neq}|}\bigg[ \frac{1}{|X|}\sum_{i\in X}[S(\alpha^ni)-S(i)] \ -\
\frac{1}{|X|}\sum_{i\in X^{\mathrm{eq}}}[S(\alpha^ni)-S(i)] \bigg]\ \ = $$
$$\frac{|X|}{n|X^\mathrm{neq}|}\overline{\sigma}_{n}^u \ - \
\frac{1}{n|X^\mathrm{neq}|}\sum_{i\in X^{\mathrm{eq}}}[S(\alpha^ni)-S(i)] \ = \
\frac{1}{n|X^\mathrm{neq}|}\sum_{i\in D_{\alpha^n}(X^{\mathrm{eq}})}[S(i) - S(\alpha^ni)]\ \geq 0.$$
By the Gibb's inequality \cite{ct} we have that
$$n\overline{\sigma}_{n}^q  \ = \ \sum_{i\in X}q_i\mathrm{ln}(\frac{|\alpha^ni|}{|i|})
\ = \ \sum_{i\in X}q_i\mathrm{ln}(\frac{q_i}{q_{\alpha^ni}})\ \geq \ 0. $$

\end{proof}

Items 2 and 3  of Theorem \ref{ie} guarantee that with uniform probability on non-equilibrium microstates, and
with probability $\ q \ $ the average $n$-steps entropy production
rate is non-negative. Suppose now that we are given before hand the mean value $\ \alpha_n>0 \ $ of the $n$-steps
entropy production rate, then following
the Jaynes' maximum entropy method is natural to consider the probability $\ t \ $ on $\ X \ $ given on
$\ i \in a \in \pi \ $ by
$$t(i)=\frac{p_{a}}{|a|} \ \ \   \mbox{where} \ \ \ p_a= \frac{e^{\lambda \overline{\sigma}_{n}(a)}}{Z(\lambda)},
\ \ \  \overline{\sigma}_{n}(a) = \frac{1}{|a|} \sum_{i\in a}S(\alpha^ni)-S(i), \ \ \  \mbox{and}
\ \ \ Z(\lambda)= \sum_{a\in A} e^{\lambda \overline{\sigma}_{n}(a)}.$$
Assuming that $\ \alpha_n \in \big[\underset{a\in A}{\mathrm{min}}\ \overline{\sigma}_{n}(a)\ , \
\underset{a\in A}{\mathrm{max}}\ \overline{\sigma}_{n}(a)\big], \ $ the parameter $\ \lambda \ $ is chosen so that
$$\sum_{a\in A}\overline{\sigma}_{n}(a) \frac{e^{\lambda \overline{\sigma}_{n}(a)}}{Z(\lambda)} = \alpha_n.$$

\section{Structural Properties of Micro-Macro  Systems}\label{spmms}

In this section we review  some of the structural properties of the category of micro-macro dynamical systems:
we introduce the product, disjoint union, restriction, coarse-graining, meet, and joint of micro-macro dynamical systems.
We also provide five general construction yielding interesting examples of micro-macro phase spaces.\\

If $\ (X,A,f,\alpha)\ $  satisfies the axioms for a micro-macro dynamical system except that $\ f \ $ may not be
surjective, then we  have  the micro-macro dynamical system $\ (X,f(X), f, \alpha). \ $ We
use this construction without change of notation, and even without mention.
The inversion functor $\ \mathrm{inv}: \mathrm{immds} \longrightarrow \mathrm{immds}\ $ is
given on objects by  $ \ \mathrm{inv}(X,A,f,\alpha) =  (X,A,f,\alpha^{-1}).\ $ The functor $\ \mathrm{inv}\ $
is defined for reversible systems as  $\ \mathrm{inv}(X,A,f,\alpha,r )=  (X,A,f,\alpha^{-1}, r).$

\begin{prop}{\em  Let $\ (X,A,f,\alpha,r) \ $ be an entropy preserving reversible micro-macro dynamical system. Then
 $\ (X,A,f,\alpha,r) \in \mathrm{L}_2(\varepsilon_1, \varepsilon_2) \ \ $ if and only if
$\ \ (X,A,f,\alpha^{-1},r) \in\mathrm{L}_2(\varepsilon_1, \varepsilon_2). \ $
}
\end{prop}

\begin{proof}Follows from $ \ |D_{\alpha^{-1}}|  =  |I|  =  |D| \ $ and
$ \ |I_{\alpha^{-1}}X^{\mathrm{neq}}|  =   |I_{\alpha^{-1}}X| = |D| =|I| =     |IX^{\mathrm{neq}}|  .$
\end{proof}

\begin{defn}{\em
Let $\ (X_1,A_1,f_1,\alpha_1) \ $ and $\ (X_2,A_2,f_2,\alpha_2)\ $ be micro-macro dynamical systems.
The product micro-macro dynamical system
is given by
$ (X_1\times X_2,A_1\times A_2,f_1\times f_2,\alpha_1\times \alpha_2). $
The map $\ \times: \mathrm{mmds} \times \mathrm{mmds} \ \longrightarrow \  \mathrm{mmds}\ $ is functorial. }
\end{defn}

The product functor  induces a product functor  on $\ \mathrm{immds}, \ $  and can be compatibly defined
 on $\ \mathrm{rmmds}\ $ so that $\ i: \mathrm{immds} \longrightarrow \mathrm{mmds},  $
$\ i_{\ast}: \mathrm{mmds} \longrightarrow \mathrm{immds}, $  $\ u: \mathrm{rmmds} \longrightarrow \mathrm{immds}, \ $
and $\ \mathrm{inv} :\mathrm{immds} \longrightarrow \mathrm{immds} \ $ are product preserving.
The following result justifies  the presence of the logarithmic function
in the Boltzmann entropy from the structural viewpoint.

\begin{prop}{\em Consider the system $\ (X_1\times X_2,A_1\times A_2,f_1\times f_2,\alpha_1\times \alpha_2).\ $ We have that
$\ S(A_1\times A_2)  =   S(A_1)  +  S(A_2), \ $ $ \ H(p_{A_1\times A_2})  =  H(p_{A_1})  +  H(p_{A_2}), \ $
$\ H(T_{A_1\times A_2})  =   H(T_{A_1})  +  H(T_{A_2}).\ $
}
\end{prop}

\begin{proof}
For $\ (a_1,a_2) \in A_1\times A_2\ $ we have that  $ \  p(a_1,a_2) =   \frac{|(a_1,a_2)|}{|A_1\times A_2|}
  =   \frac{|a_1||a_2|}{|A_1||A_2|}  =  p(a_1)p(a_2) \ $ and
$\ S(a_1,a_2)  = \mathrm{ln}|(a_1,a_2)|   =   \mathrm{ln}|a_1|  +  \mathrm{ln}|a_2|
 =  S(a_1)  +  S(a_2).\ $ Therefore
{\footnotesize $$S(A_1\times A_2) =   \sum_{(a_1,a_2)\in A_1\times A_2}S(a_1,a_2)p(a_1,a_2)
  = \sum_{a_1 \in A_1}S(a_1)p(a_1)  +  \sum_{a_2\in A_2}S(a_2)p(a_2)  =   S(A_1) +  S(A_2).\ $$}
Thus $\ H(p_{A_1\times A_2})  =   \ln\big(|A_1||A_2|\big)   -   S(A_1)  -  S(A_2)  =   H(p_{A_1})  +  H(p_{A_2}).\ $
Considering transition maps we have that $\ T_{(b_1,b_2),(a_1,a_2)} = T_{b_1 a_1}T_{b_2 a_2} \ $ since
{\footnotesize $$  \frac{\big|\{(i,j) \in a_1\times a_2\ | \ (\alpha_1(i),\alpha_2(j)) \in b_1 \times b_2 \}\big|}
{|a_1||a_2|} \ \ = \ \
\frac{\big|\{i \in a_1 \ | \ \alpha_1(i) \in b_1 \}\big|}{|a_1|}
\frac{\big|\{j \in a_2\ | \ \alpha_2(j) \in b_2 \}\big|}{|a_2|} . $$}
Therefore
$$H(T_{A_1 \times A_2}) \  =  \ -\sum_{(a_1,a_2),(b_1,b_2) \in A_1\times A_2}\mathrm{ln}|T_{(b_1,b_2),(a_1,a_2)}|\ T_{(b_1,b_2),(a_1,a_2)}p(a_1,a_2) \  = $$
$$\sum_{a_1,b_1 \in A_1, \ a_2,b_2 \in A_2}-\Big[\mathrm{ln}|T_{(b_1,a_1)}| + \mathrm{ln}|T_{(b_2,a_2)}|\Big]\ T_{(b_1,a_1)}T_{(b_2,a_2)}p(a_1)p(a_2)  \ = \ H(T_{A_1})  + H(T_{A_2}) .$$
\end{proof}

Next we phrase the asymptotic equipartition theorem \cite{ct} in terms of  Boltzmann entropy.

\begin{thm}
{\em Let $ (X,A,f) $ be a micro-macro phase space and $  \varepsilon > 0.   $ For $ n \in \mathbb{N}_{\geq 1} $ consider the $n$-power
micro-macro phase space $ \ (X^n,A^n,f^{\times n}).  \ $ Let the set of typical microstates $\ X^{n}_{\epsilon} \subseteq X^n \ $ be given by
$\ \displaystyle X^{n}_{\varepsilon}   =   \big\{ (i_1,...,i_n) \in X^n \  \big|  \ e^{n(S(A)-\varepsilon)}
 \leq  |f(i_1)|\cdots |f(i_n)|  \leq   e^{n(S(A)+\varepsilon)} \big\}, \ $ and let the set of typical macrostates
  be given by $\ A^n_{\varepsilon}  =  f^{\times n}X^n_{\varepsilon}.\ $ For $\ n \ $ large enough we have that:
\begin{enumerate}
 \item $ |X^{n}_{\varepsilon}|  \geq  (1-\varepsilon)|X^n|. $

  \item $  e^{n(S(A)-\varepsilon)} \leq  |(a_1,... , a_n)|  \leq   e^{n(S(A)+\varepsilon)} \ \ \ $ for $\ \ \ (a_1,...,a_n) \in A^n_{\varepsilon}.$

  \item $(1-\varepsilon)|A|^ne^{-n(S(A) + \varepsilon)}   \leq   |A^{n}_{\varepsilon}|  \leq   |A|^ne^{-n(S(A)-\varepsilon)}.$
\end{enumerate}
}
\end{thm}

\

\begin{defn}{\em
Let $\ (X_1,A_1,f_1,\alpha_1) \ $ and $\ (X_2,A_2,f_2,\alpha_2)\ $ be micro-macro dynamical systems. The disjoint union micro-macro dynamical system is given by
$\ (X_1\sqcup X_2,A_1\sqcup A_2,f_1\sqcup f_2,\alpha_1\sqcup  \alpha_2). $ The map $\ \sqcup: \mathrm{mmds} \times \mathrm{mmds}  \longrightarrow  \mathrm{mmds} \ $ is functorial. }
\end{defn}

Disjoint union functor  induces a disjoint union functor on $\ \mathrm{immds}, \ $  which can be naturally extended
to $\ \mathrm{rmmds}\ $ so that the functors $\ i: \mathrm{immds} \longrightarrow \mathrm{mmds}, \ $
$ i_{\ast}: \mathrm{mmds} \longrightarrow \mathrm{immds},  $  $\ u: \mathrm{rmmds} \longrightarrow \mathrm{immds}, \ $
$\ IR: \mathrm{immds} \longrightarrow \mathrm{irmmds}, \ $ $\ ER: \mathrm{immds} \longrightarrow \mathrm{ermmds}, \ $
and $\ \mathrm{inv}: \mathrm{immds} \longrightarrow \mathrm{immds} \ $ preserve disjoint unions.

\begin{prop}
{\em The following identities hold for  $ (X_1\sqcup X_2,A_1\sqcup A_2,f_1\sqcup f_2,\alpha_1\sqcup  \alpha_2) :$
\begin{enumerate}
  \item $S_{A_1\sqcup A_2}(a)  =  S_{A_1}(a) \ $   if $ \ a \in A_1; \ \ $
$S_{A_1\sqcup A_2}(a)  =  S_{A_2}(a) \ $   if $ \ a \in A_2.\ $
  \item $ p_{A_1\sqcup A_2}(a)  = \frac{|X_1|}{|X_1|+|X_2|}p_{A_1}(a) \ $ \ if $ \ a \in A_1; \ \ $
$ p_{A_1\sqcup A_2}(a)  = \frac{|X_2|}{|X_1|+|X_2|}p_{A_2}(a) \ $ \ if $ \ a \in A_2. \ \ $
  \item $S(A_1 \sqcup A_2)   =   \frac{|X_1|}{|X_1|+|X_2|}S(A_1)  +  \frac{|X_2|}{|X_1|+|X_2|}S(A_2) .$
  \item $T^{A_1\sqcup A_2}_{ab} = T^{A_1}_{ab} \ $ if $ \ a,b \in A_1; \ \ $ $T^{A_1\sqcup A_2}_{ab} = T^{A_2}_{ab} \ $ if $ \ a,b \in A_2; \ \ $ $T^{A_1\sqcup A_2}_{ab} = 0 \ $ otherwise.
\item $H(T_{A_1 \sqcup A_2})  =   \frac{|X_1|}{|X_1|+|X_2|}H(T_{A_1})  +   \frac{|X_2|}{|X_1|+|X_2|}H(T_{A_2}) .$
\end{enumerate}
}
\end{prop}

\begin{defn}\label{sub}
{\em
Let $ \ (X,A,f,\alpha)\ $ be an invertible micro-macro dynamical system and $\ Z\ $ be a subset of $\ X. \ $
The restriction of $ \ (X,A,f,\alpha)\ $ to
$\ Z \ $ is the invertible micro-macro dynamical system
$\ (Z,f(Z),f_Z,\alpha_Z) \ $ such that
 $\ f_Z: Z \longrightarrow f(Z)\ $ is the restriction to $\ Z \ $ of $\ f , \ $ and
  the bijective map $\ \alpha_Z: Z \longrightarrow Z \ $ is constructed as follows: for $\ i \in Z \ $ find the smallest
  $\ l\in \mathbb{N}_{>0} \ $ such that $\ \alpha^{l}(i) \in Z \ $ and set $\ \alpha_Z(i) =  \alpha^{l}(i).$
}
\end{defn}

The restriction construction can be applied to reversible systems  $ \ (X,A,f,\alpha, r)\ $ as follows. Let $ Z $ be a subset of $X$ closed under $ r, $ then $\ (Z,f(Z),f_Z,\alpha_Z, r_Z) \ $ is a reversible micro-macro dynamical system. Suppose that $\ \alpha_Z(i) =   \alpha^l(i) =  j, \ $ with $\ l > 0 \  $ as small as possible. Note that the identities $\alpha^l(i)=j$ and $\alpha^l (rj)=ri$ are equivalent, the former identity implies the latter since $\alpha^l(rj)=r(r\alpha r)^l(j)=r\alpha^{-l}(j)=ri,$  the other implication is similar.
Therefore we conclude that  $r\alpha_Z(rj) = r\alpha^l(rj) =  i =  \alpha_Z^{-1}(j). \ $

\begin{thm}\label{j2l}
{\em Let $\ (Z, \pi_Z, \alpha_Z) \ $ be the restriction to $\ Z \subseteq X \ $ of the invertible micro-macro dynamical system
$ \ (X,\pi,\alpha).\ $ The maximum entropy of \ $\ (Z, \pi_Z, \alpha_Z) \ $  is less than or equal to
the maximum entropy of $ \ (X,\pi,\alpha). \ $
 }
\end{thm}

\begin{proof}
$\ \ \underset{i \in Z}{\mathrm{max}}\ S(i) \ = \  \underset{a \in \pi_Z}{\mathrm{max}}\ \mathrm{ln}|a| \ = \
\underset{a \in \pi, a \cap Z \neq \emptyset}{\mathrm{max}}\ \mathrm{ln}|a \cap Z| \ \leq \
 \underset{a \in \pi}{\mathrm{max}}\ \mathrm{ln}|a| \ = \  \underset{i \in X}{\mathrm{max}}\ S(i).$
\end{proof}

\

\begin{defn}{\em Let $\ (X,A,f,\alpha)\ $ be a micro-macro dynamical systems and  $\ g:A \longrightarrow B \ $ be  a surjective map.
The associated coarse-grained micro-macro dynamical system is given by  $\ (X,B,gf,\alpha). \ $
Coarse-graining is also naturally  defined for reversible systems, preserving
invariant reversible systems since $\ fr=f\ $ implies  $\ (gf)r=gf, \ $ and preserving equivariant
reversible systems if $\ B \ $ comes with an involution $\ r:B \longrightarrow B \ $ such that $ \ gr=rg,\ $ since
in this case  $\ (gf)r=g(rf)=r(gf). \ $

}
\end{defn}

\begin{prop}\label{cogra}
{\em Let $\ (X,B,gf,\alpha) \ $ be the coarse-grained micro-macro dynamical system obtain
from $(X,A,f,\alpha) $ and  $ g:A \longrightarrow B.  $ We have that  $\ S(B)  \geq  S(A) , \   $
$\ H(p_B)  \leq  H(p_A),  $ $ \  H(T_B)- S(B) \leq H(T_A) - S(A), \ $ and
$ \    H(p_B) + H(T_B)  \leq       H(p_A) + H(T_A).$
}
\end{prop}
\begin{proof}
$$S(B)\  =  \ \sum_{b\in B}\mathrm{ln}|b|\ p_b \  =  \ \sum_{b\in B}\sum_{g(a)=b}\mathrm{ln}(\sum_{g(a)=b}|a|)\frac{|a|}{|X|}
\  \geq $$
$$\sum_{b\in B}\sum_{g(a)=b}\mathrm{ln}|a|\ \frac{|a|}{|X|} \ =  \ \sum_{a \in A}\mathrm{ln}|a| \ \frac{|a|}{|X|} \  =
\  \sum_{a\in A}\mathrm{ln}|a|\ p_a \  = \ S(A),$$
thus $\ \ H(p_B) =  \mathrm{ln}|X| - S(B)  \leq  \mathrm{ln}|X| - S(A)  =  H(p_A) . \ $
We have that
$$T_{b_2b_1}  \ =  \ \frac{|\{i \in b_1 \ | \ \alpha(i) \in b_2 \}| }{|b_1|} \  =  \ \sum_{f(a_1)=b_1,\ g(a_2)=b_2}\frac{|\{i \in a_1 \ | \ \alpha(i) \in a_2 \}|}{|b_1|}.$$ Therefore
$$H(T_B) \ \leq  \  -\sum_{b_1,b_2 \in B}\sum_{f(a_1)=b_1,\ g(a_2)=b_2}\mathrm{ln}(\frac{|\{i \in a_1 \ | \ \alpha(i) \in a_2 \}|}{|b_1|})\frac{|\{i \in a_1 \ | \ \alpha(i) \in a_2 \}|}{|b_1|}\frac{|b_1|}{|X|}  \ =$$
$$-\sum_{a_1,a_2 \in A}\mathrm{ln}(\frac{|\{i \in a_1 \ | \ \alpha(i) \in a_2 \}|}{|a_1|} \frac{|a_1|}{|b_1|})\frac{|\{i \in a_1 \ | \ \alpha(i) \in a_2 \}|}{|b_1|}\frac{|b_1|}{|X|} \ = $$
$$-\sum_{a_1,a_2 \in A}\mathrm{ln}(\frac{|\{i \in a_1 \ | \ \alpha(i) \in a_2 \}|}{|a_1|}) \frac{|\{i \in a_1 \ | \ \alpha(i) \in a_2 \}|}{|a_1|}\frac{|a_1|}{|X|} \ \ + $$
$$-\sum_{a_1,a_2 \in A}\mathrm{ln}(\frac{|a_1|}{|b_1|})\frac{|\{i \in a_1 \ | \ \alpha(i) \in a_2 \}|}{|b_1|}\frac{|b_1|}{|X|} \  = $$
$$H(T_A) -   \sum_{a_1\in A}\mathrm{ln}|a_1|\frac{|a_1|}{|X|} \ + \
\sum_{a_1\in A}\mathrm{ln}|b_1|\frac{|a_1|}{|X|} \ =  \ H(T_A)+ S(B)- S(A).$$
Thus we get that
$$H(T_B) +  H(p_B) =    H(T_B)  +  \mathrm{ln}|X|  - S(B)    \leq  H(T_A)  +  \mathrm{ln}|X|   -  S(A)   =   H(T_A)  +  H(p_A). $$
\end{proof}

 Given a micro-macro dynamical system $\ (X,A,f,\alpha)\ $
we let $\ (X,A_{e},f_{e},\alpha) \ $ be the micro-macro dynamical system   with a unique equilibrium macrostate where
 $ \  A_{e}  =  A\setminus A^{\mathrm{eq}} \sqcup \{e\},  \ $ and $\ f_e \ $ is given by
$\ f_{e}(i)=f(i) \ $   if $\ i\notin X^{\mathrm{eq}}, \ $ and   $ \  f_{e}(i)=e \ $ if $ \  i \in X^{\mathrm{eq}}. \  $
Let $\ (X,\{n,e\}, p,\alpha) \ $ be the micro-macro dynamical system  with a unique equilibrium  and a unique non-equilibrium macrostates,
 with  $\ p \ $ given by $ \  p(i)=e \ $ if $ \ i\in X^{\mathrm{eq}}, \ $  and $ \  p(i)=n \ $ if $\ i \notin X^{\mathrm{eq}}. \ $
 The systems $\ (X,A,f,\alpha)\ $ and  $ \ (X,A_{e},f_{e},\alpha) \ $ have, respectively, the same number
 of (non) equilibrium microstates, microstates with strict increase, strict decrease, and constant entropy.
 Thus $\ (X,A,f,\alpha) \in \mathrm{L}_1(\varepsilon_1)\  $ if and only if
$\ (X,A_{e},f_{e},\alpha) \in \mathrm{L}_1(\varepsilon_1). \ $
We have that $\ S(\{n,e\}) \geq  S(A_e) \geq   S(A), \ $ and
 $\ |DX^{\mathrm{eq}}| \geq  (1 - \varepsilon_2)|X^{\mathrm{neq}}| \ $  if and only if
$\ |De| \geq  (1 - \varepsilon_2)|n|  .$

\begin{defn}{\em
Let $\ (X,A,f) \ $ and $\ (X,B,g)\ $ be micro-macro phase spaces. The meet micro-macro phase space is given by
$\ (X,(f,g)X,(f,g))  \ $
where $\ (f,g)X \ $ is the image of the map $\ (f,g):X \longrightarrow A\times B. \ $
The joint micro-macro phase space is given by
$ \ (X,A \sqcup_X B,i_A f), \ $ where
the amalgamated sum $\ A \sqcup_X B \ $ is the quotient of $\ A \sqcup B \ $ by the relation generated by
$\ f(x) \sim g(x) \ $ for $ \ x \in X.  $
}
\end{defn}

\begin{prop}{\em
$\ S(A)  \geq  S((f, g)X), \ $ $ \  S(B)  \geq  S((f, g)X), \  $
$\ S(A \sqcup_X B) \geq  S(A), \ $ and $ \  S(A \sqcup_X B)  \geq  S(B).$}
\end{prop}

\begin{proof}The result follows since entropy grows under coarse-graining (Proposition \ref{cogra}), the identities
$\ f  =  \pi_A(f,g), \ g = \pi_B(f,g), \ i_Af  =  i_Bf \ $ and  the fact that the maps
$\ i_A:A \longrightarrow A \sqcup_X B \ $ and  $\ i_B:B \longrightarrow A \sqcup_X B\ $ are surjective.

\end{proof}

Next we introduce five general constructions of micro-macro phase spaces. \\

\noindent $\mathrm{I}.  \ $ Let $  (X, A, f, \alpha) $ be a micro-macro dynamical system. For $  n \in \mathbb{N}_{\geq 1}  $ we construct micro-macro phase spaces  $\ (X, A^n, \gamma_n) \ $ useful for  understanding macrostates transitions in $ \ (X, A, f, \alpha). \ $
The map $\ \gamma_n: X \longrightarrow A^n \ $ is given by $ \ \gamma_n(i)  =   ( fi,  f\alpha(i), f\alpha^2(i),  ...  ,  f\alpha^{n-1}(i) ) .\ $
The entropy $  \ S(a_1,...,a_n) \ $ of  macro-state $\ (a_1,...,a_n) \in A^n \ $ is the logarithm of the number of micro-realizations
of the transitions $\  a_1  \rightarrow  \cdots  \rightarrow a_n \ $
through the $\alpha$-dynamics.\\

\noindent $\mathrm{II}. \  $ Given a finite set $ B $ (boxes) and $ k \in \mathbb{N} $ we let $\ \mathbb{N}^{B}_k \ $ be the set of $\mathbb{N}$-valued measures on $ B $ of total measure $k, $ that is
$\ \displaystyle \mathbb{N}^{B}_k    =   \{  w:B \longrightarrow \mathbb{N}  \  | \ \sum_{b \in B}w(b) =  k  \}. \  $
Given another finite set $ P $ (particles) we let $\ [P,B] \ $ be the set of maps from $\ P \ $ to $\ B \  $ (particles to boxes).
We obtain the micro-macro phase space $ \  ([P,B], \ \mathbb{N}^{B}_{|P|}, \ c  )   \ $
where the surjective map $\ c:[P,B] \longrightarrow  \mathbb{N}^{B}_{|P|}\ $ sends  $ f  $ to the measure $\ c_f  \ $ given by $ \ c_f(b) =  |f^{-1}(b)|.\ $ We study the two-boxes case  in Example \ref{ss}.\\

\noindent $\mathrm{III}. \  $  $ S_P \ $ acts on $\ [P,B] \ $ by $\ (\alpha f)(p)  =  f(\alpha^{-1}(p)). \ $  $\ S_B\ $ acts respectively on $\ [P,B] \ $ and on   $\ \mathbb{N}^{B}_k \ $  by $ \  (\beta f)(b) =  \beta f(b)\ \ \mbox{and} \ \ (\beta c) (b)  =  c(\beta^{-1}(b)). \  $
We have that $$ c_{\beta f \alpha^{-1}}( b)  \  =  \  |(\beta f \alpha^{-1})^{-1}( b)|  \  =  \
 |\alpha^{-1}(f^{-1}(\beta^{-1} b))|  \  =  \ |f^{-1}(\beta^{-1} b)|  \  =  \ (\beta c_f)(b) .$$
Let $\ G_P \subseteq  S_P\ $ and  $\ G_B  \subseteq  S_B\ $ be subgroups, and $\ M(P,B)  \subseteq  [P,B]\ $ be
 invariant under the action of $\ G_P \times G_B\ $ on $\ [P,B]. \ $ From the identities above we get the micro-macro phase space
$ \  (M(P,B)/G_P\times G_B,\ \mathbb{N}^{B}_k/G_B,\  c ). $

\

Several instances of this construction, attached to illustrious names, have been study in the literature.
Niven  \cite{ni, ni2} considers the following cases:
\begin{itemize}
  \item Maxwell-Boltzmann statistics: $\ M(P,B)=[P,B], \ \ G_P=1,\ \ G_B=1.$
  \item  Lynden-Bell statistics: $\ M(P,B)=\mathrm{Inj}(P,B), \ \ G_P=1,\ \ G_B=1.$
  \item  Bose-Einstein statistics: $\ M(P,B)=[P,B], \ \ G_P=S_P,\ \ G_B=1.$
 \item $m$-gentile statistics:  $\ M(P,B)=\{f\in [P,B] \ | \ |f^{-1}b|\leq m \ \}, \ \ G_P=S_P,\ \ G_B=1.$
  \item Fermi-Dirac statistics: $\ M(P,B)=\mathrm{Inj}(P,B), \ \ G_P=S_P,\ \ G_B=1.$
  \item DI statistics: $\ M(P,B)=[P,B], \ \ G_P=1,\ \ G_B=S_B.$
  \item II statistics: $\ M(P,B)=[P,B], \ \ G_P=S_P,\ \ G_B=S_B.$
\end{itemize}
An unifying aim of these studies has been  finding the macrostate of greatest entropy. Other instances of this fairly general construction are yet to be explored. \\

\noindent $\mathrm{IV}.\  $ With the notation of example III \ let $\ \mathrm{prob}_B \ $ be space of
probability distributions on
 $\ B, \ $ $\ \mathbb{N}_k^{B} \longrightarrow \mathrm{prob}_B\ $ be the normalization map,  and
$\ \mathbb{N}^{B}_k/G_B \longrightarrow \mathrm{prob}_B/G_B\ $ be the induced map.
Let  $\ \mathrm{prob}_B \longrightarrow A \ $ be a $\ G_B$-invariant map with $\ A \ $
a finite set,  and  $\ \mathrm{prob}_B/G_B \longrightarrow A \ $ be the induced map.
Consider the composition map $\ f \ $ obtained from the chain of maps
$$\ M(P,B)/G_P\times G_B \longrightarrow \mathbb{N}^{B}_k/G_B \longrightarrow \mathrm{prob}_B/G_B  \longrightarrow A.\ $$
 We have constructed a micro-macro phase space $\ (M(P,B)/G_P\times G_B, A, f). \ $\\

\noindent $\mathrm{V}.\   $ Our last construction relies on a  generalized version of
the theory of combinatorial species \cite{berg,  di2, dc, di} where $\ \mathrm{surj}, \ $ the category
 of finite sets and surjective maps, plays the role usually reserved  for the category
of finite sets and bijections. Given functor $\ F:\mathrm{surj} \longrightarrow \mathrm{surj}, \ $
 we obtain the map
$ \  \widehat{F}:\mathrm{mmds} \longrightarrow \mathrm{mmds}  \ $ given by
$ \ \widehat{ F}(X,A,f, \alpha)  =  (FX, FA, Ff, F\alpha), \ $
acting functorially on surjective morphisms in $\ \mathrm{mmds}.\ $
For example,  we have functors $\ \mathrm{P},  G, L, \mathrm{Par }: \mathrm{surj}
 \longrightarrow \mathrm{surj} \ $ sending a set $\ X \ $ to the set $\ \mathrm{P}X \ $ of subsets
 of $\ X, \ $ the set $\ GX\ $ of simple graphs
on $ \ X,\ $ the set $\ LX\ $ of linear orderings on $\ X,\ $ and the set $\ \mathrm{Par}X\ $ of partitions on $\ X,\ $
respectively. Moreover, given functors $\ F,G: \mathrm{surj} \longrightarrow \mathrm{surj}\ $ we build new such
functors using the following natural operations:
$$\ (F + G)(x) = F(x) \sqcup G(x), \ \ \ \ \ \ (F \times G)(x) = F(x) \times G(x), \ $$
$$FG(x)= \bigsqcup_{a \cup b=x}F(a) \times G(b) , \ \ \ \ \ \  F \circ G(x)= F(G(x)),$$
$$\ F(G)(x)=\bigsqcup_{\pi \in \mathrm{Par}(x)}F(\pi)\times \prod_{a \in \pi}G(a) ,\ $$
in the latter case we set $\ F(G)(\emptyset)= F(\emptyset)=G(\emptyset)=\emptyset \ $ and assume
that $\ G \ $ is monoidal, i.e. it comes with  functorial (under bijections) maps $\ G(a)\times G(b) \longrightarrow G(a \sqcup b)\ $
satisfying natural associativity constraints. Note that the functor $\ \widehat{F} \ $ can be extended to reversible systems
yielding the map $ \  \widehat{F}:\mathrm{rmmds} \longrightarrow \mathrm{rmmds}  \ $ given by
$ \ \widehat{ F}(X,A,f, \alpha,r)  =  (FX, FA, Ff, F\alpha, Fr)  \ $ acting functorially on surjective morphisms
in $\ \mathrm{rmmds} .$

\section{Always Increasing Entropy on Invertible Systems}\label{fsl}

In this section we consider invertible micro-macro dynamical systems for which entropy is always increasing,
i.e. those systems for which property $\ \mathrm{L}_1(0)\ $ holds. Although we are going to show that
this case occurs with low probability, Theorem \ref{parfi}, we develop it in details to illustrate  the duality principle described in the introduction, see Theorem \ref{un}.
We first show a combinatorial analogue of Zermelo's observation of the tension between
recurrence and the second law.

\begin{prop}\label{csod}
{\em Let $\ (X,A,f,\alpha)\ $ be an invertible micro-macro dynamical system. Entropy is always increasing if and only if entropy
 is constant on $\ \alpha$-orbits.
}
\end{prop}\label{ce}
\begin{proof}
If $\ S \ $ is constant on the orbits of $\ \alpha, \ $ then $\ S(i) =  S(\alpha(i))\ $ for all $\ i \in X\ $ and
thus  entropy is always increasing. Conversely, if entropy is always increasing then  $$S(i) \leq  S(\alpha(i))  \leq \cdots \leq  S(\alpha^l(i))  =  S(i),$$
where $l$ is the cardinality of the $\alpha$-orbit of $\ i. \ $ Thus the inequalities above are  identities.
\end{proof}

\begin{cor}{\em Let $(X,A,f,\alpha) $ be an invertible micro-macro dynamical system.
 The induced stochastic map $\ T: A \longrightarrow A \ $ is experimentally reproducible, in Jaynes' sense, if and only
 if there is a permutation $ t:A \longrightarrow A  $ such that $ \ T_{ab} =  \delta_{at(b)} \ $
 and $ \ S(t(a)) =  S(a) \ $ for $\ a \in A.$
}
\end{cor}

\begin{thm}\label{parfi}
{\em Let $\ (X,\pi) \ $ be a micro-macro phase space with $\ O_{\pi}=\{k_1 < \cdots < k_o \} \ $.
 A random  permutation $ \alpha \in \mathrm{S}_X $ determines a micro-macro dynamical system
 $\ (X, \pi, \alpha)\ $ with always increasing entropy with probability
$$ \binom{|X|}{|\widehat{\pi}_1|,\hspace*{.1cm}\dots{}\hspace*{.1cm},|\widehat{\pi}_o|}^{-1}.$$

 }
\end{thm}

\begin{proof}
It follows from Proposition \ref{csod} that such a permutation $\ \alpha \ $ induces and is determined by
  permutations on the sets
$\ \widehat{\pi}_j. \ $ The induced permutations are arbitrary, so  the result follows because
 a set with $\ n \ $ elements has $\ n! \ $ permutations.  The probabilistic statement is then clear assuming
  uniform probability on $\ \mathrm{S}_X.\ $ Figure \ref{tp} shows a couple of permutations for
   which entropy is always increasing given the partitioned set.
\end{proof}

\begin{figure}[t]
    \centering
    \includegraphics[scale=.7]{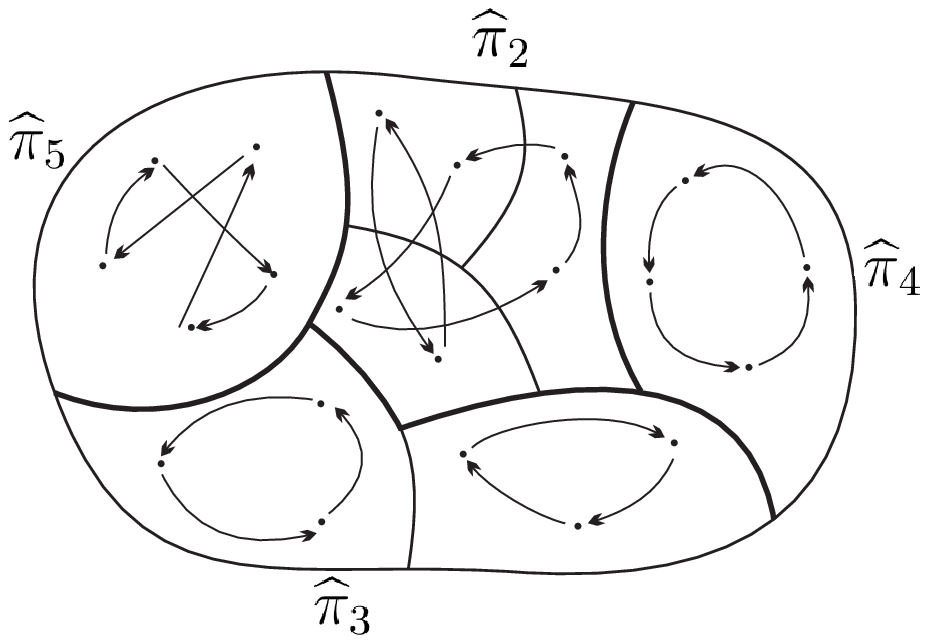}\hspace{.8cm}
    \includegraphics[scale=.7]{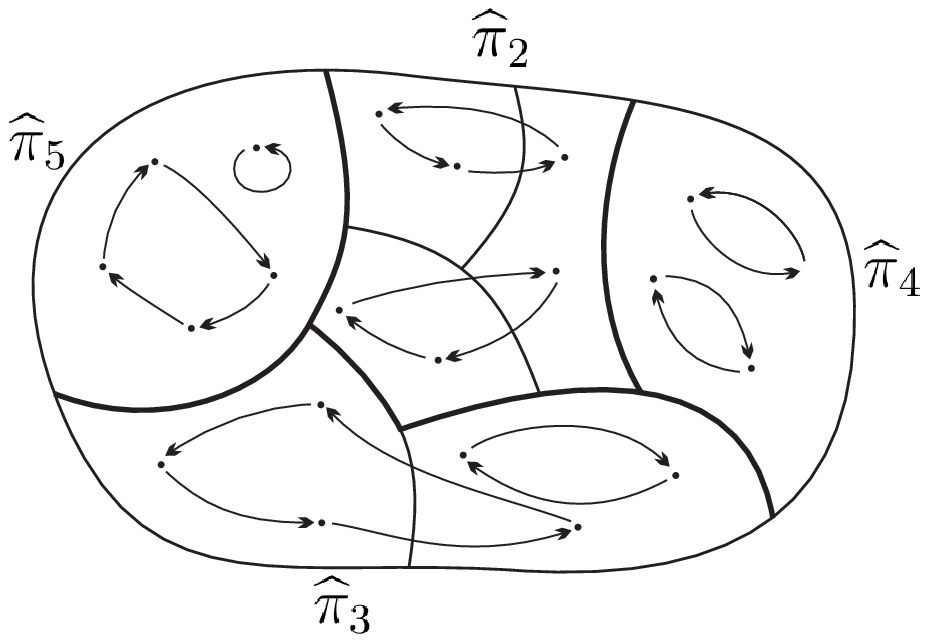}
    \caption{Permutations on a partitioned set with always increasing entropy.}
    \label{tp}
\end{figure}

Suppose now that we are given a set $\ X \ $ together with a permutation $ \ \alpha \ $ on it. We want to know how many partitions $\ \pi  \ $ are there such that
entropy is always increasing on the system $\ (X, \pi, \alpha). \ $  Recall that the number
 of partitions on $\ \{1,2,\dots{},nk \}\ $ into $\ n \ $ blocks each of cardinality $\ k \ $ is
 given by $\ \displaystyle   \frac{(nk)!}{n!k!^n}. \ $ Given $\ S \subseteq \mathrm{P}X, \ $ a family of subsets $X, $  we let $\ \underline{S} \ $  be the set of maps
$\ l:S \longrightarrow \mathbb{N}_{\geq 1}\ $ such that for all $ \ k \in  \mathbb{N}_{\geq 1}\ $ we have that
$$k  \ \  \mbox{divides}  \ \ \overline{l}(k)=  \sum_{l(A) = k}|A|.$$
Given a permutation $\ \alpha \ $ on the finite set $\ X, \ $ we let $\ \mathrm{Cyc}(\alpha) \ $
 be the partition of $\ X \ $ into $\alpha$-cycles.
We obtain the micro-macro phase space $\ (X, \mathrm{Cyc}(\alpha), c), \ $
 where $\ c \ $ is the map sending $\ i \in X\ $ to the $ \alpha$-cycle $\ c(i) \ $
generated by $\ i. \ $

\begin{thm}\label{perfi}
{\em
Let $\ X \ $ be a finite set and $\ \alpha \in \mathrm{S}_X. \ $
The number of partitions $\ \pi \in \mathrm{Par}(X)\ $ such that entropy is always increasing in  $\ (X, \pi, \alpha) \ $  is given by
$$\sum_{l \in \underline{\mathrm{Cyc}(\alpha)}}\ \prod_{k \in \mathrm{Im}(l)}\ \frac{\overline{l}(k)!}{k!^{\frac{\overline{l}(k)}{k}} \frac{\overline{l}(k)}{k}!}.$$
}
\end{thm}

\begin{proof}
Let $\ \pi \ $ be a partition on $\ X  \ $ such that entropy is always increasing in $\ (X, \pi,  \alpha) \ $ and let $ \ c \in \mathrm{Cyc}(\alpha)$.
According to Proposition \ref{csod} if $\ c \cap \widehat{\pi}_k  \neq \emptyset, \ $ then $\ c \subseteq \ \widehat{\pi}_k. \ $
Thus we can associate to $\ \pi \ $ the map $\ l_{\pi} \in \underline{\mathrm{Cyc}(\alpha)} \ $ given by
$\ l_{\pi}(c)=k \ \ \ \  \mbox{if \ and \ only \ if}   \ \ \ \ c \subseteq  \widehat{\pi}_k. \ $ Conversely, given $\ l \in \underline{\mathrm{Cyc}(\alpha)} \ $ the partitions $\ \pi \ $ with $\ l_{\pi}=l \ $ can be constructed by choosing for each $\ k \in \mathrm{Im}(l) \ $ a uniform partition with blocks of cardinality $k$ on the set
$ \ \displaystyle \bigcup_{l(c)=k}c \ \subseteq \  X.\ $ Entropy is always increasing for such partitions, and there are
$\ \displaystyle \frac{\overline{l}(k)!}{k!^{\frac{\overline{l}(k)}{k}}\frac{\overline{l}(k)}{k}!} \ $ of them.
Figure \ref{fig:Ex4} displays a couple of examples of this construction.
\end{proof}

We have shown the following  instance of the micro/macro duality principle.
\begin{thm}\label{un}
{\em
Let $ X  $ be a finite set. The number of invertible micro-macro dynamical systems
$\ (X,\pi, \alpha{}) \ $ such that entropy is always increasing is given by
$$ \sum_{\pi \in \mathrm{Par}(X) } |\widehat{\pi}_1|!\dots{}|\widehat{\pi}_{o}|!\ =  \
\sum_{\alpha\in \mathrm{S}_X}\ \sum_{l \in \underline{\mathrm{Cyc}(\alpha)}}\ \prod_{k \in \mathrm{Im}(l)}\frac{\overline{l}(k)!}{k!^{\frac{\overline{l}(k)}{k}}\frac{\overline{l}(k)}{k}!}.$$
}
\end{thm}

\begin{proof}Consider the set of pairs $\ (\pi, \alpha) \in \mathrm{Par}(X)\times S_X\ $ such that
 entropy is always increasing on the invertible micro-macro dynamical system  $\ (X, \pi, \alpha).\ $
  Cardinality of this set can be found by  fixing $\ \pi\ $ and then counting permutations
  $\ \alpha, \ $ leading, by Theorem \ref{parfi},  to the left-hand side of the proposed formula.
  Alternatively, it can be counted by fixing $\ \alpha \ $ and then counting  partitions
   $\ \pi, \ $ leading, by Theorem \ref{perfi}, to the right-hand side of the proposed formula.
\end{proof}

\begin{figure}[t]
    \centering
    \includegraphics[scale=.8]{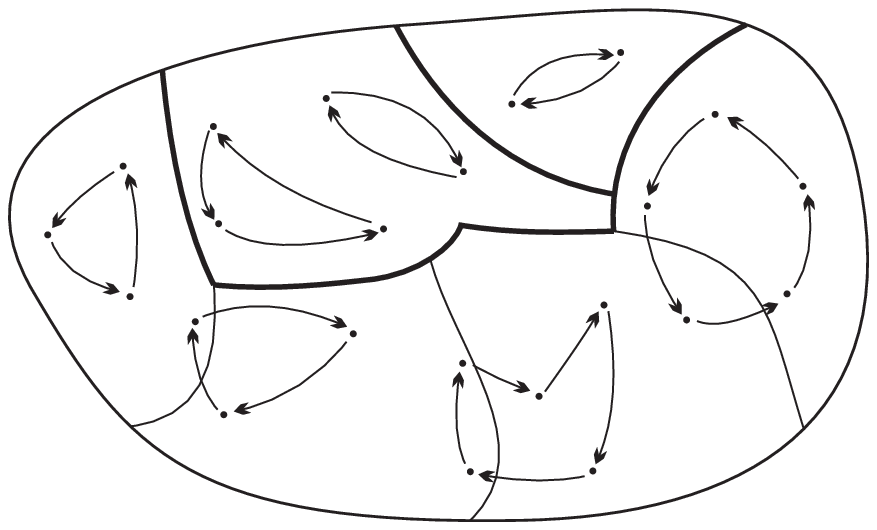}\hspace{.7cm}
    \includegraphics[scale=.8]{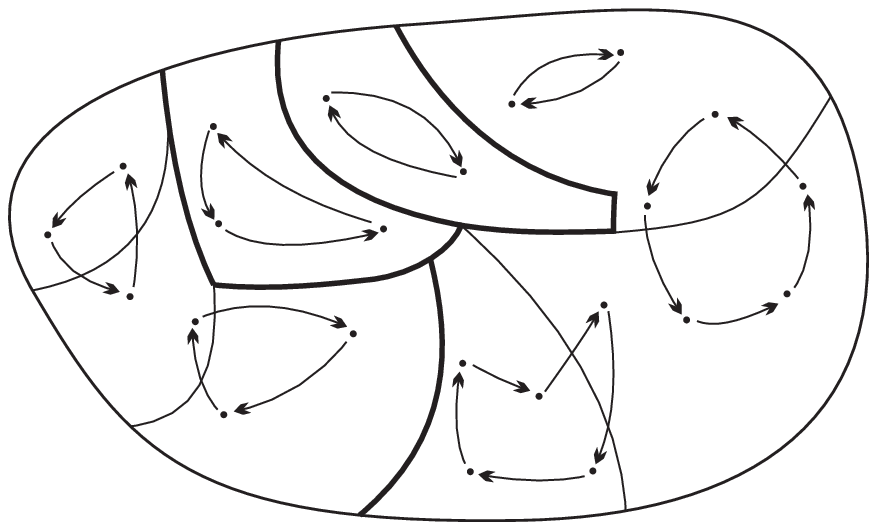}
    \caption{Partitions on a set with permutation for which entropy is always increasing.}
    \label{fig:Ex4}
\end{figure}

\section{Bounding Entropy Strict Decreases on Invertible Systems}\label{mfsl}

In this section we fix a micro-macro phase space $\  (X,\pi ) \ $ and find  an upper bound
for the  number $\ |D(X,\pi, \alpha )| \ $ of strict decreases in entropy for an arbitrary permutation
$\ \alpha \in \mathrm{S}_X. $

\begin{thm}\label{n}
{\em
Let $\ (X,\pi)\ $ be an invertible micro-macro dynamical system.  We have that
$$ \underset{\alpha \in \mathrm{S}_X}{\mathrm{max}}\ |D(X,\pi, \alpha )| \ =  \  |X| -  |\widehat{\pi}_{r}|,$$
where $\ r \ $ is such that $\  |\widehat{\pi}_{r}|  \geq  |\widehat{\pi}_k| \ $ for $ \ k \in \mathbb{N}. $}
\end{thm}

\begin{proof} First we show that $\  \underset{\alpha \in \mathrm{S}_X}{\mathrm{max}}\ |D(X,\pi, \alpha )|  \leq     |X| -  |\widehat{\pi}_{r}|,$ i.e. we show that
$\ |D(X,\pi, \alpha)|  \leq   |X| -  |\widehat{\pi}_{r}| \ $ for  $\ \alpha \in \mathrm{S}_X . \ $ Consider an $ \alpha$-orbit
containing  $ \ l \geq 1\ $ microstates in $\ |\widehat{\pi}_{r}|; \ $ such an orbit can be written as
$\  a_1 \ast a_2 \ast \cdots a_l \ast a_1 \ $ where the symbol $\ \ast \ $ stands for the orbit elements (if any)  not in $ \ \widehat{\pi}_{r}, \ $ and $\  a_1, a_2, \cdots, a_l \ $ are the orbit elements in $ \ \widehat{\pi}_{r}. \ $ Note that in each subsegment of orbit
$\  a_i \ast  \ $ there must be at least one  microstate with strictly increasing entropy, and thus a total of $\ l \ $ microstates with strictly increasing entropy.
Taking all orbits that intersect $ \ \widehat{\pi}_{r} \ $ into account  we obtain the desired inequality.
It remains to show that $\  \underset{\alpha \in S_X}{\mathrm{max}}\ |D(X,\pi, \alpha )|
 \leq    |X| -  |\widehat{\pi}_{r}|,\ $ i.e. one has to check that there exists a permutation $\ \alpha \ $
such that $\  |D(X,\pi,\alpha )| =    |X| -  |\widehat{\pi}_{r}|. \ $ Write the set $\ X \ $ as in Figure \ref{fig:Ex12} with the blocks
$ \ \widehat{\pi}_{k} \ $ contained in left justified line and $ \ \widehat{\pi}_{k} \ $ above $ \ \widehat{\pi}_{l} \ $ if $ \  k > l. \  $
Define the permutation $\ \alpha \  $ by flowing downwards on each vertical column, and sending the bottom element of a column to the highest
element in the column.  The permutation $\ \alpha \  $ obtained has exactly $\  |\widehat{\pi}_{r}|, \ $  microstates with increasing entropy.
\end{proof}

Figure \ref{fig:Ex12} displays an example of a micro-macro dynamical systems for which the bound from Theorem \ref{n} on the number of
strict decreases in entropy is achieved. For a finite set $\ X\ $  set
$$  \displaystyle d_X  =   \underset{\pi \in \mathrm{Par}X, \ \alpha \in S_X}{\mathrm{max}}\  |D(X,\pi, \alpha )|. \  $$

\begin{figure}[t]
    \centering
    \includegraphics[width=12cm, height=5cm]{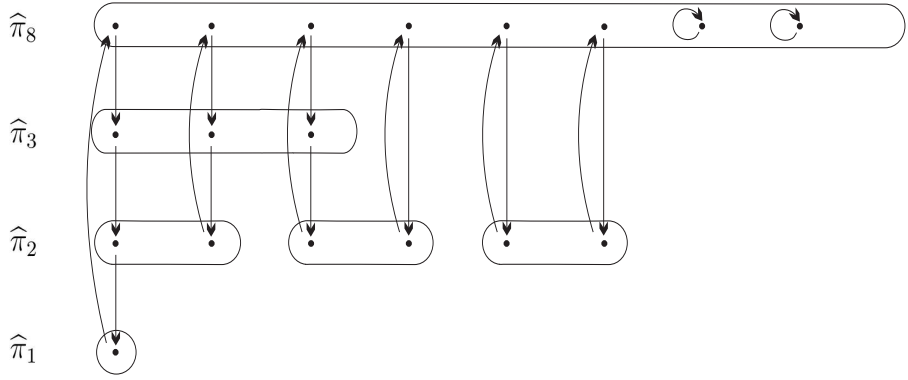}
    \caption{Micro-macro phase space with maximal decreasing set.}
    \label{fig:Ex12}
\end{figure}

\begin{thm}{\em  For any $\ X,\ $ we have that $\ \displaystyle d_X   =
  |X|  -  \underset{l \vdash |X|}{\mathrm{min}}\
\underset{1 \leq k \leq |X|}{\mathrm{max}}\ kl_k, \ \ $
where $\ l \ $ runs over the partitions of $\ |X|, \ $ i.e. $\ \displaystyle l=(l_1,...,l_{|X|})
\ \ \mbox{and}  \ \ \sum_{i=1}^{|X|}kl_k  =  |X|. $
}
\end{thm}

\begin{proof} Follows from the identities
$\ \ \ \displaystyle d_X \ = \   \underset{\pi \in \mathrm{Par}X, \ \alpha \in S_X }{\mathrm{max}}\
 |D(X,\pi, \alpha )| \  =
 \   \underset{\pi \in \mathrm{Par}X}{\mathrm{max}}\ |X| -  |\widehat{\pi}_{r}| \  =  $
$\ \ \ \ |X|  -  \underset{\pi \in \mathrm{Par}X}{\mathrm{min}}\  |\widehat{\pi}_{r}|\  =
    \ \ \ |X|  -  \underset{l \vdash |X|}{\mathrm{min}}\ \underset{1 \leq k \leq |X|}{\mathrm{max}}\ kl_k.$
\end{proof}

Next we describe various scenarios guaranteeing or not the validity of property
$\ \mathrm{L}_1 \ $ for arbitrary permutations on combinatorial micro-macro phase spaces
under suitable hypothesis on the growth of  $\ |\widehat{\pi}_{k}|.\ $ Note that both  $\ X \ $ and $\ A \ $
 grow to infinity, in subsequence sections
we will let $\ |X| \ $ go to infinity but keep the cardinality of $\ A \ $ fixed.

\begin{thm}\label{ycle}
{\em Let $\ \{c_1,...,c_n,... \} \ $ and $\ \{k_1  <  ... < k_n  < ... \} \ $  be a couple of sequences of natural numbers such that $\ c_nk_n   \leq   c_{n+1}k_{n+1}.\ $
For $\ n \in \mathbb{N}_{\geq 1}\ $ consider the micro-macro phase space $\ (X_n,\pi_n) \ $
such that $$X_{n} \  =  \ \bigsqcup_{s=1}^n [k_s]^{\sqcup c_s},$$
and $\ \pi_n \ $ is the displayed partition of $\ X_{n} \ $ with $\ c_s \ $ blocks of cardinality $\ k_s \ $ for $\ s \in [n]$.
\begin{enumerate}
  \item If $\ c_nk_n \simeq  an^r\ $  with $\ r > 1, \ $ then there are permutations $\ \alpha_n \in \mathrm{S}_{X_{n}} \ $ such
  that $$ \underset{n\rightarrow \infty}{\lim}\frac{|D(X_n,\pi_n, \alpha_n)|}{|X_n|}\ = \ 1.  $$

  \item If $\ c_nk_n \simeq  ar^n \ $ with $\ r > 1, \ $   then for arbitrary permutations $\ \alpha_n \in \mathrm{S}_{X_{n}} \ $ we have:
  $$ \underset{n\rightarrow \infty}{\lim}\ \frac{|D(X_n,\pi_n,  \alpha_n)|}{|X_n|}\ \leq \ \frac{1}{r}.$$

  \item If $\ c_nk_n \simeq  an^n, \ $ then for arbitrary permutations $\ \alpha_n \in \mathrm{S}_{X_{n}} \ $ we have that:
  $$ \underset{n\rightarrow \infty}{\lim}\ \frac{|D(X_n,\pi_n, \alpha_n)|}{|X_n|}\ = \  0. $$

\end{enumerate}
}
\end{thm}

\begin{proof}
We  study the asymptotic behavior of
$ \  \displaystyle \frac{1}{|X_n|}  \underset{\alpha_n \in \mathrm{S}_{X_n}}{\mathrm{max}}\ |D(X_n,\pi_n, \alpha_n )| \  $  as $ \ n \rightarrow  \infty. \ $ Note that $\ |\widehat{\pi}_{n,k_s}|  =  c_sk_s \ $ assume its largest value for $ s=n. \ $ By Theorem \ref{n} we have:
$$\lim_{n\rightarrow \infty}\frac{1}{|X_n|}  \underset{\alpha_n \in \mathrm{S}_{X_n}}{\mathrm{max}}\ |D(X_n,\pi_n, \alpha_n )|\  =  \ \lim_{n\rightarrow \infty}\frac{|X_n|-|\widehat{\pi}_{n,k_n}|}{|X_n|}\  = \ \lim_{n\rightarrow \infty}\ \frac{\sum\limits_{s=1}^{n-1}|\widehat{\pi}_{n,k_s}|}{\sum\limits_{s=1}^{n}|\widehat{\pi}_{n,k_s}|}.$$
Under the three alternative hypothesis for  $\ c_nk_n\ $  stated in the Theorem  the limit $\  \displaystyle \underset{n\rightarrow \infty}{\lim}\ \frac{c_nk_n}{c_{n+1}k_{n+1}}\ $ exits. Therefore by the Stolz-Cesaro theorem we have that
$$\lim_{n\rightarrow \infty}\frac{1}{|X_n|}  \underset{\alpha_n \in \mathrm{S}_{X_n}}{\mathrm{max}}\ |D(X_n,\pi_n, \alpha_n )|\ =  \ \lim_{n\rightarrow \infty}\ \frac{\sum\limits_{s=1}^{n-1}|\widehat{\pi}_{n,k_s}|}{\sum\limits_{s=1}^{n}|\widehat{\pi}_{n,k_s}|}  \ = \ \lim_{n\rightarrow \infty}\ \frac{\sum\limits_{s=1}^{n-1}c_sk_s}{\sum\limits_{s=1}^{n}c_sk_s} \  = \    \underset{n\rightarrow \infty}{\lim}\ \frac{c_nk_n}{c_{n+1}k_{n+1}}.$$

The desired result follows since:
\begin{enumerate}
  \item If $\  c_nk_n \simeq  an^r\ $  with $\ r > 1, \ $ then $\  \displaystyle  \underset{n\rightarrow \infty}{\lim}\ \frac{c_nk_n}{c_{n+1}k_{n+1}}= 1. \ $
  \item If $\ c_nk_n \simeq  ar^n \ $ with $\ r > 1, \  $ then $\  \displaystyle  \underset{n\rightarrow \infty}{\lim}\ \frac{c_nk_n}{c_{n+1}k_{n+1}} =  \frac{1}{r}. \ $
  \item If $\ c_nk_n \simeq  an^n, \  $ then $\  \displaystyle  \underset{n\rightarrow \infty}{\lim}\ \frac{c_nk_n}{c_{n+1}k_{n+1}}  = 0. \ $
\end{enumerate}

\end{proof}

So polynomial growth for $\ c_nk_n \ $ yields no control on the number of strict decreases in entropy
for arbitrary invertible micro-macro  dynamical systems on $\ (X_n,\pi_n); \ $
exponential growth  gives or not a good control on the number of strict decreases in entropy depending on
the value of $r; \ $  if $\ r \ $ is close to $1  $  exponential growth gives only a minor improvement
over polynomial growth in terms of imposing property $\mathrm{L}_1$; \ if $\ r \ $ is quite large,
then any invertible micro-macro dynamical systems on $\ (X_n,\pi_n) \ $ have a relatively negligible
set of strict decreases in entropy, for  large $\ n;\ $ the faster than exponential growth $\ n^n\ $
guarantees a vanishing numbers of decreases in entropy for an invertible micro-macro dynamical
 systems on $\ (X_n,\pi_n),\ $  for large $\ n. \ $

\begin{exmp}\label{ss}
{\em For $ n \in \mathbb{N},\ $ let $\ (\mathrm{P}[n], [0,n], |\ |) =   (\mathrm{P}[n], \pi_n)\ $ be the micro-macro
 phase space with $ |\ | $ sending $ A \subseteq [n]$ to its cardinality $ |A|.$ We show that for arbitrary permutations one
 has no  control on the number of entropy decreasing microstates.  In the odd case $\ (\mathrm{P}[2n+1], \pi_{2n+1})\ $ we have that
$$O_{\pi_{2n+1}}  =  \Big\{ {2n+1 \choose k} \ | \ 0 \leq k \leq n \Big\} \ \ \ \ \ \mbox{and} \ \ \ \ \
\Big|\widehat{\pi}_{2n+1,{2n+1 \choose k}}\Big|  =  2{2n+1 \choose k}.$$
Since $\ |\widehat{\pi}_{2n+1,{2n+1 \choose n}}|  \geq   |\widehat{\pi}_{2n+1,{2n+1 \choose k}}|, \ $
by Theorem \ref{n} we have
{\small $$\lim_{n\rightarrow \infty}\frac{\underset{\alpha \in S_{\mathrm{P}[2n+1]}}{\mathrm{max}}
 |D(\mathrm{P}[2n+1],\widehat{\pi}_{2n+1}, \alpha)|}{|\mathrm{P}[2n+1]|} =
 \lim_{n\rightarrow \infty}\frac{|\mathrm{P}[2n+1]|-
|\widehat{\pi}_{2n+1,{2n+1 \choose n}}|}{|\mathrm{P}[2n+1]|}  =
1  -  \lim_{n\rightarrow \infty}\frac{2{2n+1 \choose n}}{2^{2n+1}}   =  1. $$} The even case $\ (\mathrm{P}[2n], \pi_{2n})\ $ case is interesting since the block of larger cardinality does not lie in
the zone of largest cardinality for  $\ n \geq 3. \ $ Indeed we have that
$$O_{\pi_{2n}} =  \Big\{ {2n \choose k} \ | \ 0 \leq k \leq n \Big\}, \ \ \ \
\Big|\widehat{\pi}_{2n,{2n \choose n}}\Big|  =  {2n \choose n}, \ \ \ \
 \Big|\widehat{\pi}_{2n,{2n \choose k}}\Big|  =  2{2n \choose k} \ \  \mbox{for}
\ \  0 \leq k \leq n-1.$$
Thus by Theorem \ref{n} we have that
$$ \lim_{n\rightarrow \infty}\frac{\underset{\alpha \in S_{\mathrm{P}[2n]}}{\mathrm{max}}\
|D(\mathrm{P}[2n],\widehat{\pi}_{2n}, \alpha)|}{|\mathrm{P}[2n]|}\ =\
 \lim_{n\rightarrow \infty}\frac{|\mathrm{P}[2n]|-  |\widehat{\pi}_{2n,{2n \choose n-1}}|}{|\mathrm{P}[2n]|}
\  =  \
1  -  \lim_{n\rightarrow \infty}\frac{2{2n \choose n-1}}{2^{2n}} \  = \  1. $$
}
\end{exmp}
\section{Local Arrow of Time and  Zero Jump Permutations}\label{sllip2}

In this section we introduce a couple of further  formalizations of the arrow of time,
one  dealing with zones and the other one dealing with  blocks. We  introduce  jump of a permutation
on a micro-macro phase space, and study zero jump micro-macro dynamical systems. We introduce a further
"continuity" restriction on permutations by given microstates the structure of a simple graph.
Fix a micro-macro phase-space $ \ (X, \pi) \ $  with $\ O_{\pi}=\{k_1 < \cdots < k_o \}. \ $

\begin{defn}{\em
Entropy defines a zonal $\ \varepsilon$-arrow of time on $\ (X, \pi, \alpha), \ $ written $\ (X, \pi, \alpha)
\in \mathrm{ZAT}(\varepsilon), \ $
if $\ \ \displaystyle |I\widehat{\pi}_{i}|  \geq  (1-\varepsilon ) |\widehat{\pi}_{i}|\ \ $  for $\ i \in [o-1]. \ $
 We say that $\ (X, \pi, \alpha) \ $ satisfy the $\ \mathrm{L}_3(\varepsilon_1 , \varepsilon_2)\ $ property if
  it satisfies properties  $\ \mathrm{L}_1(\varepsilon_1)\ $ and $\ \mathrm{ZAT}(\varepsilon_2)$.
}
\end{defn}

Property $\ \mathrm{ZAT}(\varepsilon) \ $ implies property $\ \mathrm{GAT}(\varepsilon) \ $ since
$$|IX| \ = \ \sum_{i=1}^{o-1}|I\widehat{\pi}_{i}|
\ \geq \   \sum_{i=1}^{o-1}(1-\varepsilon) |\widehat{\pi}_{i}| \ = \ (1-\varepsilon ) |X^{\mathrm{neq}}|.$$

\begin{thm}\label{alli}
{\em  Let $\ (X, \pi) \ $ be a micro-macro phase-space and
$\ \delta_1, \delta_2,  \gamma_i \in \mathbb{R}_{\geq 0} \ $ for $\ i\in [o-1] \ $ ($\gamma_1=0$) \ be such that
 \begin{enumerate}

   \item $ \displaystyle  \delta_1 |X^{\mathrm{eq}}|  \leq  |X^{\mathrm{neq}}|  \leq \delta_2 |X^{\mathrm{eq}}| \ \ \  \mbox{with} \ \ \ \delta_2 \leq \varepsilon_1(\delta_1+1),$

   \item $ \displaystyle   \sum_{j=1}^{i-1}|\widehat{\pi}_{j}|  \leq \gamma_i |\widehat{\pi}_{i}| \ $
   with $\ \gamma_i \leq \varepsilon_2 \ $ for $\ 2 \leq i \leq o-1,$
 \end{enumerate}
then  $\ (X, \pi, \alpha) \in \mathrm{L}_3(\varepsilon_1,\varepsilon_2 )  \ $ for $ \ \alpha \in \mathrm{S}_X \ $
such that $\ |C\widehat{\pi}_{i}| \leq (\varepsilon_2 - \gamma_i)|\widehat{\pi}_{i}|\ $ for $\ i \in [o-1].$
}
\end{thm}

\begin{proof} By Theorem \ref{i1} we have that
$\ \frac{|D|}{|X|}  \leq  \varepsilon_1. \ $
The desired result holds since
$$ |D\widehat{\pi}_{i} \sqcup C\widehat{\pi}_{i}| \ = \
|D\widehat{\pi}_{i}| + |C\widehat{\pi}_{i}| \ \leq \
\sum_{j=1}^{i-1}|\widehat{\pi}_{j}| + |C\widehat{\pi}_{i}|  \ \leq \
(\gamma_i + \varepsilon_2 - \gamma_i )|\widehat{\pi}_{i}|  \ \leq \ \varepsilon_2 |\widehat{\pi}_{i}|.$$
\end{proof}

Under the hypothesis of Theorem \ref{alli} letting $\  \lambda = \underset{i}{\mathrm{min}}\gamma_i^{-1}\ $
we have that $ \ |\widehat{\pi}_{i}|   \geq   \lambda(1+ \lambda)^{i-2}|\widehat{\pi}_{1}| \ $
for  $ \ 2 \leq i \leq o-1 .\ $  The following results are direct consequence of Theorem \ref{alli}. Our next results illustrate quite well
the principle of large differences:  zone cardinalities even if  relative negligible may actually be approaching infinity.

\begin{cor}\label{ct1}
{\em Let $\ (X_n, \pi(n), \alpha_n) \ $ be a sequence of micro-macro dynamical systems
with $\ O_{\pi(n)}=\{k_1(n) < \cdots < k_o(n) \} \ $ and such
that:
\begin{itemize}

  \item $ \displaystyle   \frac{\sum_{j=1}^{i-1}|\widehat{\pi}_{j}(n)|}{ |\widehat{\pi}_{i}(n)| }  \rightarrow 0  \ \ $
  as $\ \ n \rightarrow \infty, \ \ $ for $\ \ 2 \leq i \leq o, \ $
  \item $\displaystyle \frac{|C_{\alpha_n}\widehat{\pi}_{i}(n)| }{|\widehat{\pi}_{i}(n)|} \rightarrow 0 \ \ $
  as $\ \ n \rightarrow \infty, \ \ $ for $\ \ 2 \leq i \leq o-1,\ $
\end{itemize}
under these conditions $\ (X_n, \pi(n), \alpha_n) \in \mathrm{L}_3(\varepsilon_1,\varepsilon_2 ).  \ $

}
\end{cor}

\

\begin{cor}\label{ct2}
{\em Let $\ (X_n, \pi(n), \alpha_n) \ $ be a sequence of micro-macro dynamical systems with
$\ O_{\pi(n)}=\{k_1(n) < \cdots < k_o(n) \} \ $ satisfying a large deviation principle in the sense that there
number $\ 0=z_o < \cdots  < z_1 \ $ and $ \ c_i > z_i \ $ for $\ i \in [o-1]\ $
such that $ \ |\widehat{\pi}_i(n)| \simeq e^{-nz_i}|X_n|\ $ and $ \  |C_{\alpha_n}\widehat{\pi}_i(n)| \simeq e^{-nc_i}|X_n| .\  $
Under these conditions  $\ (X_n, \pi(n), \alpha_n) \in \mathrm{L}_3(\varepsilon_1,\varepsilon_2 ).  \ $

}
\end{cor}

\

Let $\ (X,f,A)\ $ be a micro-macro phase space, and let $\ C \subseteq \mathrm{prob}_A  \subseteq \mathbb{R}^{|A|} \ $
be a non-empty convex subset
of the space of probabilities on $ \ A .\ $ Typically $\ C \ $ is given as the subspace of $\ \mathrm{prob}_A \ $
satisfying linear constrains, i.e. one is given  maps $ \ h_l:A \longrightarrow \mathbb{R} \ $ and constants
$\ u_l \in  \mathbb{R} \ $  for $\ l \in [k]\ $ such that
$\ q \in C \ $ if and only if
$$\sum_{a\in A}q(a)h_l(a) = u_l. $$
Let $\ B=\{b_1,...,b_o \} \ $ be a partition of
$\ C \ $ such that $\ \overline{\mathrm{int}(b_i)}=\overline{b_i}, \ $ and  consider a sequence of micro-macro dynamical
systems $\ (X_C^n, L_n, B, \alpha_n ) \ $ constructed as follows:
\begin{itemize}
  \item  Consider the map $ \ \widehat{L}_n \ $ obtained as the composition of maps
$\ X^n  \longrightarrow  A^n  \longrightarrow  \mathrm{prob}_A  ,\ $ where
the map $ \ A^n  \longrightarrow   \mathrm{prob}_A\ $ sends a tuple $\ s \in A^n \ $
to the empirical probability distribution $\ \widehat{s} \ $ on $\ A\ $ given by
$\ \displaystyle \widehat{s}(a)=\frac{1}{n}\big| \{i \in [n]\ | \ s_i=a \} \big| . \  $
  \item Set $\ X_C^n = \widehat{L}_n^{-1}C, \ $ and let $\ L_n \ $ be the composition of the
  of maps   $ \  X_C^n  \longrightarrow C  \longrightarrow  B , \ $ where the first map is
  the restriction to $ \  X_C^n \ $ of $ \ \widehat{L}_n, \ $ and the second map is the coarse
  graining map induced by the partition $\ B.$
  \item $\alpha_n \ $ is a permutation on $\ X_C^n.$
\end{itemize}

Relative entropy (Kullback-Leibler divergence)  is the map $\ D(\ | \ ): \mathrm{prob}_A \times \mathrm{prob}_A  \longrightarrow [0,\infty]\ $
given on   $\ r,q \in \mathrm{prob}_A\ $  by
$$D(r|q) = \sum_{a\in A}r(a)\mathrm{ln}\frac{r(a)}{q(a)}. $$
Let $\  q_{\ast} \in C \ $ be the probability in $\ C \ $ with minimum relative entropy $\ D(q|p) \ $ with respect to  the probability
$\ p \ $  on $\ A \ $ given by $\ p(a)=\frac{|a|}{|X|}.$

\begin{thm}\label{ct3}
{\em
Assume that the  systems $\ (X_C^n, L_n, B, \alpha_n ) \ $ are such that:
\begin{itemize}
  \item $\underset{q \in b_o}{\mathrm{inf}}D(q|p) < \underset{q \in b_{o-1}}{\mathrm{inf}}D(q|p)< \cdots
  < \underset{q \in b_1}{\mathrm{inf}}D(q|p),$
  \item $\displaystyle \frac{|C_{\alpha_n}b_i|}{|b_i|} \rightarrow 0\ $ as $\ n \rightarrow \infty,$
\end{itemize}
then $\ \ (X^n, B, L_n, \alpha_n) \in  \mathrm{L}_3. \ $  }
\end{thm}

\begin{proof}
Follows from  Corollary \ref{ct1} using Sanov's theorem \cite{dz, e}, which implies for
$\ E \subseteq C\ $ with $\ \overline{\mathrm{int}(E)}= \overline{E}\ $ that:
$$p(\widehat{s}\in E|s\in  X_C^n  ) \ = \ \frac{p(\widehat{s}\in E, s\in X_C^n) }{p( X_C^n  ) } \ \simeq \
e^{ -n\big( \underset{q \in E}{\mathrm{inf}} D(q|p) - \underset{q \in C}{\mathrm{inf}} D(q|p)\big) }=
e^{ -n\big( \underset{q \in E}{\mathrm{inf}} D(q|p) - D(q_{\ast}|p)\big) }.$$
\end{proof}

Given a set with a partition on it we have a jump degree on maps from the set to itself. In the
 applications it is expected that the dynamics is given by  a low jump map.
\begin{defn}
{\em Let $\ (X, \pi, \alpha) \ $ be a micro-macro dynamical system. The jump of $\ \alpha \ $
is the cardinality of its set of jumps  $ \ J_{\alpha} \ $ given by
$$J_{\alpha} \ = \ \coprod_{i \in X}J_{\alpha}(i)  \ = \ \coprod_{i \in X}\Big\{|a| \ \Big| \ a \in \pi \ \ \mbox{and} \ \ |a| \in\  <|\overline{i}|, |\overline{\alpha(i)}|> \Big\},$$ where for $\ n,m \in \mathbb{N}\ $ we set
$ \  <n,m>  =  (n,m) \ \ \mbox{if} \ \  n\leq m, \ \  \mbox{and}  \ \  <n,m> =  (m,n) \ \  \mbox{if} \ $ $  n > m. \ $
We call $\  |J_{\alpha}(i)| \ $ the jump of $\ \alpha\ $ at $\ i \in X,\ $ and
let $\ S_{X,\pi}^0\ $ be the set of zero jump permutations of $\ (X, \pi) $.
}
\end{defn}

\begin{prop}\label{pp}
{\em Let $ \ (X, \pi)\ $ be a micro-macro phase space and $\ \alpha \in S_{X,\pi}^0 \ $ then
$\  |I| =  |D|.\ $
}
\end{prop}

\begin{proof} We show that the number of strict increases and the number
 of strict decreases on each $ \alpha $-orbit are equal.
Let $\ i \  $ be a microstate with lowest entropy among the microstates
in an $\alpha$-orbit. Assume that there are more strict increases than
strict decreases in the $\alpha$-orbit of $\ i, \ $  a contradiction arises because the orbit can't return to
the microstate $\ i \ $ as it will necessarily  end up in a microstate of higher entropy since there are no jumps.
 \end{proof}

\begin{cor}\label{calli}
{\em  Let $\ (X, \pi) \ $ be a micro-macro phase-space  and $\  \delta_1, \delta_2, \gamma_i \in
 \mathbb{R}_{\geq 0} \ $ for $\ i\in [o-1] \ $ ($\gamma_1=0$) \ be such that
 \begin{enumerate}

   \item $ \displaystyle  \delta_1 |X^{\mathrm{eq}}|  \leq  |X^{\mathrm{neq}}|
   \leq \delta_2 |X^{\mathrm{eq}}| \ \ \  \mbox{with} \ \ \ \delta_2 \leq \varepsilon_1(\delta_1+1),$

   \item $ \displaystyle   |\widehat{\pi}_{i-1}|  \leq \gamma_i |\widehat{\pi}_{i}| \ $ with $\ \gamma_i \leq
   \varepsilon_2 \ $ for $\ 2 \leq i \leq o-1,$
 \end{enumerate}
then  $\ (X, \pi, \alpha) \in \mathrm{L}_3(\varepsilon_1,\varepsilon_2 )  \ $ for
$ \ \alpha \in \mathrm{S}_{X,\pi}^0\  $ such that $\ |C\widehat{\pi}_{i}|
\leq (\varepsilon_2 - \gamma_i)|\widehat{\pi}_{i}|\ $ for $\ i \in [o-1].$
}
\end{cor}

\begin{proof} Follows from  Theorem \ref{i1} we the inequalities
$$|D\widehat{\pi}_{i} \sqcup C\widehat{\pi}_{i}| \ = \
|D\widehat{\pi}_{i}| + |C\widehat{\pi}_{i}| \ \leq \
|\widehat{\pi}_{i-1}| + |C\widehat{\pi}_{i}|  \ \leq \
(\gamma_i + \varepsilon_2 - \gamma_i )|\widehat{\pi}_{i}|  \  \leq \ \varepsilon_2 |\widehat{\pi}_{i}|.$$

\end{proof}
Under the hypothesis of Corollary \ref{calli} letting
$\  \lambda = \underset{i}{\mathrm{min}}\gamma_i^{-1}\ $ we have that
$\ |\widehat{\pi}_{i}|   \geq   \lambda^{i-1}|\widehat{\pi}_{1}| \ $ for $ \  i \in  [o-1] .\ $
Next we  introduce the arrow of time in block form.

\begin{defn}{\em
Entropy defines a block $\varepsilon$-arrow of time on  $\ (X, \pi, \alpha), \ $  written
 $\ (X, \pi, \alpha) \in \mathrm{BAT}(\varepsilon), \ $
if $\ \ \displaystyle |Ia|  \geq  (1-\varepsilon ) |a|\ \ $  for $\ a \in \pi^{\mathrm{neq}}. \ $  We say that $\ (X, \pi, \alpha) \ $ satisfy property $\ \mathrm{L}_4(\varepsilon_1 , \varepsilon_2)\ $ if
it satisfies both $\ \mathrm{L}_1(\varepsilon_1)\ $ and $\ \mathrm{BAT}(\varepsilon_2)$.
}
\end{defn}

Property $\ \mathrm{BAT}(\varepsilon) \ $ implies property $\ \mathrm{ZAT}(\varepsilon) \ $ since
$$|I\widehat{\pi}_{i}| \ = \ \sum_{a \in \pi, |a|=k_i}|Ia|
\ \geq \   \sum_{|a|=k_i}(1-\varepsilon) |a| \ = \ (1-\varepsilon ) |\widehat{\pi}_{i}|.$$

\

For our next results we assume that the microstates $\ X \ $ are the vertices  of a simple graph
 $\ (X,E),\ $ i.e. $E  $ is a family of subsets of $X$  of cardinality two.
The macrostates $\ \pi\ $ acquire a simple graph structure $\ (\pi, \mathcal{E}), \ $ where $\ \{a,b\}\in \mathcal{E}\ $ if and only if
there are microstates $\ i\in a \ $  and $\ j\in b\ $ with $\ \{i,j\}\in E.\ $ Given $\ a \in \pi\ $ we set
$$Ba\ \ = \underset{\{b\in \pi^{\mathrm{neq}} \ | \ |b|< |a|, \
\{a,b \} \in \mathcal{E} \}}\bigsqcup b  \ \ \ \subseteq \  \ X.$$
A permutation $\ \alpha \in \mathrm{S}_X\ $ is called $\ E$-1-Lipschitz continuous  if for $\ \{i,j \}\in E \ $ we have that either $\ \alpha(i)=\alpha(j)\ $ or $\ \{\alpha(i),\alpha(j) \} \in E. \ $
Let $\ \mathrm{S}_{X,E}  \subseteq \mathrm{S}_X $ be the set of $E$--1-Lipschitz continuous permutations.

\begin{thm}\label{allivoy}
{\em  Let $ (X, \pi)  $ be a micro-macro phase-space  and  $\ (X,E) \ $  a simple graph,
and $\   \delta_1, \delta_2, \gamma_a \in \mathbb{R}_{\geq 0} \ $ for $\ a\in \pi^{\mathrm{neq}}\ $ ($\gamma_a=0\ $ if $\ Ba=\emptyset$) \ be such that
 \begin{enumerate}

   \item $ \displaystyle  \delta_1 |X^{\mathrm{eq}}|  \leq  |X^{\mathrm{neq}}|  \leq \delta_2 |X^{\mathrm{eq}}| \ \ \  \mbox{with} \ \ \ \delta_2 \leq \varepsilon_1(\delta_1+1),$

   \item $ \displaystyle   |Ba|  \leq \gamma_a |a| \ $ with $\ \gamma_a \leq \varepsilon_2 \ $ for $\ a \in \pi^{\mathrm{neq}},$
 \end{enumerate}
then  $\ (X, f, A, \alpha) \in \mathrm{L}_4(\varepsilon_1,\varepsilon_2 ) \ $ for any
$E$-1-Lipschitz continuous permutation $\  \alpha \in \mathrm{S}_{X,E}\ $  such that $\ |Ca| \leq (\varepsilon_2 - \gamma_a)|a|\ $ for $\ a \in \pi^{\mathrm{neq}}.$
}
\end{thm}

\begin{proof} Follows from Theorem \ref{i1} and the inequalities
$$|Da \sqcup Ca| \ = \
|Da| + |Ca| \ \leq \
|Ba| + |Ca|  \ \leq \
(\gamma_a + \varepsilon_2 - \gamma_a )|a| \  \leq \ \varepsilon_2 |a|.$$
\end{proof}

\begin{cor}\label{callivoy}
{\em  Let $ \ (X, \pi, E) \ $ be a micro-macro phase-space with  $\ (X,E) \ $ a simple graph,
and   $\   \delta_1, \delta_2, \gamma_a \in \mathbb{R}_{\geq 0} \ $ for $\ a\in \pi^{\mathrm{neq}} \ $ ($\gamma_a=0\ $ if $\ Ba=\emptyset$) \ be such that
 \begin{enumerate}

   \item $ \displaystyle  \delta_1 |X^{\mathrm{eq}}|  \leq  |X^{\mathrm{neq}}|  \leq \delta_2 |X^{\mathrm{eq}}|
 \ \ \  \mbox{with} \ \ \ \delta_2 \leq \varepsilon_1(\delta_1+1),$

   \item $ \displaystyle   |Fa|  \leq \gamma_a |a| \ $ with $\ \gamma_a \leq \varepsilon_2 \ $
 for $\ a\in A^{\mathrm{neq}},\ $ where $$Fa \ \ = \  \underset{\{b\in \pi \ | \ |b|< |a|, \
\{a,b \} \in \mathcal{E}, \ (|b|,|a|)= \emptyset \}}\bigsqcup b \ \ \ \subseteq \ \ X,$$
 \end{enumerate}
then  $\ (X, \pi, \alpha) \in \mathrm{L}_4(\varepsilon_1,\varepsilon_2 ) \ $ for any
zero jump $\ E$-1-Lipschitz  continuous permutation $ \ \alpha \in \mathrm{S}_{X,\pi}^0 \ $ such that $\ |Ca| \leq (\varepsilon_2 - \gamma_a)|a|\ $ for $\ a \in \pi^{\mathrm{neq}}.$
}
\end{cor}

\section{Orbit Properties and the Equilibrium Reaching Time}\label{op}

By design the equilibrium  plays a priori no distinguished role in properties
$\ \mathrm{L}_i.\ $ Localizing to orbits suggest further interesting properties inspired by the Gibbs description of the
second law for which the equilibrium plays a main role.
All definitions and constructions in this section can be weakened by
allowing a set of badly behaved orbits not having the required properties, with a small parameter bounding the probability that
 a microstate be in such orbits. Within this more general framework
all arguments given in this section should be though as applying generically, i.e. to the complement of the bad orbits.
We begin by defining, for equilibrium bound systems, a strictly increasing
function on non-equilibrium macrostates.

\begin{defn}\label{e}
{\em A micro-macro dynamical system $\ (X,A,f,\alpha)\ $ is equilibrium bound if each $\ \alpha$-orbit
intersects $\ X^{\mathrm{eq}}. \ $ For such systems the equilibrium reaching time map $\ e: X \longrightarrow \mathbb{N} \ $ is given by
$ \  \displaystyle  e(i) =   \mbox{smallest} \ \ k\in \mathrm{\mathbb{N}}\  \ \mbox{such\ that \ } \alpha^k(i) \in X^{\mathrm{eq}}.
\ \ \ \mbox{Set} \ \ E = \underset{i\in X}{\mathrm{max}}\ e(i).$
}
\end{defn}

\begin{thm}\label{g0}
{\em Let $\ (X,A,f,\alpha)\ $ be an equilibrium bound  micro-macro dynamical system.
\begin{enumerate}
  \item The map $\ e: X \longrightarrow [0,E] \ $ is strictly decreasing  on $\ X^{\mathrm{neq}}\ $ and  has value  $ \ 0 \ $ on  $\ X^{\mathrm{eq}}.\ $
  \item If $\ (X,A,f,\alpha) \in \mathrm{L}_1(\varepsilon),\ $
then the probability that $\ e \ $ be strictly decreasing is less than  $\ \varepsilon, \ $ and
the  average jump of $\ e \ $ is  less than  $\ (|A|-2)\varepsilon.$
\item Consider the micro-macro dynamical  system $ \ X_e=(X,[0,E],e,\alpha) \ $  where we assume that $\ \alpha \ $ is invertible,
 $\ (X,A,f,\alpha) \in \mathrm{L}_1(\varepsilon),\ $   and
$ \alpha X^{\mathrm{eq}} \cap  e^{-1}(k) \neq \emptyset \ $ for $\ k  \in [0,E]. \ $  Then $ \ (X,[0,E],e,\alpha) \in \mathrm{L}_2(\varepsilon,0).$
\end{enumerate}

}
\end{thm}

\begin{proof}For $\ i \in X^{\mathrm{neq}}\ $ we have that $\ e(i)>0 \ $ and
$\ \alpha^{e(i)-1}(\alpha(i)) =   \alpha^{e(i)}(i) \in X^{\mathrm{eq}}, \ $ thus
$\ \ e(\alpha(i)) \leq  e(i) -1 <  e(i).\ $ Under the hypothesis of item 2 we have that
$$\frac{|\{i \in X \ |\  e(\alpha(i)) > e(i) \}|}{|X|} \ = \
\frac{|DX^{\mathrm{eq}}|}{|X|} \ \leq \ \frac{|D|}{|X|} \ \leq \ \varepsilon.$$
Regarding the average jump of $\ e \ $ we have that:
$$\frac{1}{|X|}\sum_{i\in X}|J(i)| \ = \ \frac{1}{|X|}\sum_{i\in DX^{\mathrm{eq}}}|J(i)| \ \leq \
\frac{(|A|-2)|DX^{\mathrm{eq}}|}{|X|} \ \leq \ \frac{(|A|-2)|DX|}{|X|} \ \leq \ (|A|-2)\varepsilon.$$
We show item 3. Note first that $ \ e^{-1}(k) \neq \emptyset, \ $  since $ \ e^{-1}(E) \neq \emptyset \ $ by definition,
and choosing $\ i\in e^{-1}(E) \ $  we have that $\ \alpha^{E-k}(i) \in e^{-1}(k) .\ $
Note also that $\ X_e^{\mathrm{eq}} =X^{\mathrm{eq}} \ $  and  $ \ X_e^{\mathrm{neq}} =X^{\mathrm{neq}}. \  $
The restriction map
$\ \alpha: e^{-1}(k) \longrightarrow e^{-1}(k-1)\ $ is injective, thus  $\ |e^{-1}(k)| \leq |e^{-1}(k-1)|. \ $
Moreover  $  \ e^{-1}(k-1)=  \alpha(e^{-1}(k))  \ \sqcup \ \alpha(X^{\mathrm{eq}}) \cap  e^{-1}(k-1), \ $
and thus for $\ k \in [1,E]\ $ and $\ i \in e^{-1}(k) \ $  we have
$ \ S(i)=  \mathrm{ln}|e^{-1}(k)| <  \mathrm{ln}|e^{-1}(k-1)| = S(\alpha(i)). \ $ Finally, by item 2 we  have that $$\displaystyle
\frac{|DX_e|}{|X|}\  = \ \frac{|\{i \in X \ |\  e(\alpha(i)) > e(i) \}|}{|X|} \ \leq \ \varepsilon. \ $$
\end{proof}

\begin{rem}{\em The condition $ \ \alpha X^{\mathrm{eq}}\cap  e^{-1}(k) \neq \emptyset \ $  is quite natural for
reversible systems since, in this case, the image of entropy decreasing equilibrium microstates
nearly covers all non-equilibrium macrostates (Lemma \ref{cc}, Theorem  \ref{jj}), and thus it is reasonable
to  expect that $ \ |e^{-1}(k)|\ $  and  $ \ |\alpha X^{\mathrm{eq}} \cap  e^{-1}(k)|  \ $
be nearly equal.
}
\end{rem}

\begin{rem}{\em It is worthwhile to analyze Theorem \ref{g0} in the light of Zermelo's  critique of the Boltzmann
$H$-theorem \cite{sew, v}. Zermelo pointed out that a recurring system does not admit a non-constant always decreasing
function along orbits (Proposition \ref{csod}), and thus regardless of further details the main claim of
the $H$-theorem can not be correct.  Boltzmann accepts the argument but claims  that the $H$-theorem
remains valid if understood as a probabilistic statement, i.e.  allowing the $H$-function to  be strictly increasing
with low probability. The  equilibrium reaching time function $\ e \ $
satisfies, under the conditions of Theorem \ref{g0},  probabilistic properties similar  to those expected
for the $H$-function, according to Boltzmann, indeed it satisfies stronger properties as it is strictly increasing
on non-equilibrium microstates.

}
\end{rem}

\begin{rem}{\em The Loschmidt's  critique of the  $H$-theorem \cite{cer, ceri, chi, sew} has its combinatorial counterpart
in Proposition \ref{loc}: if the $H$-function is defined on a reversible system (it is not if dynamics is 
defined via the Boltzmann equation, but it should be if dynamics is defined mechanically), then since
it is reversion invariant it must have an equal number of increasing and decreasing microstates, contrary to the claim
that it is predominantly decreasing.  
Boltzmann acknowledges the argument, but points out that
the  $H$-decreasing orbit segments are the ones that actually show up in nature, i.e. the probabilistic
symmetry of microstates is broken. This observation is the origin of the low entropy past hypothesis.
The function $\ e \ $ can be constructed for equilibrium bound reversible systems as well; as a rule it will not
be reversion invariant, indeed if  $\ r \ $ preserves equilibria,
i.e. $r$ restricts to a map
$r: X^{\mathrm{eq}} \longrightarrow X^{\mathrm{eq}}, \ $ then $\ e(ri)=e(i) \ $ if and only if
$$\mbox{smallest} \ \ k\in \mathrm{\mathbb{N}}\  \ \mbox{such\ that \ } \alpha^k(i) \in X^{\mathrm{eq}} \ \  =  \ \
\mbox{smallest} \ \ k\in \mathrm{\mathbb{N}}\  \ \mbox{such\ that \ } \alpha^{-k}(i) \in X^{\mathrm{eq}},$$
a trivial condition for $\ i \in X^{\mathrm{eq}} \ $ but fairly restrictive for $\ i \in X^{\mathrm{neq}}.\ $ In fact it holds
only for the middle microstate on each maximal $\ \alpha$-orbit segment of odd cardinality in  $\ X^{\mathrm{neq}}. \ $
}
\end{rem}

We proceed to localize to $\ \alpha$-orbits the various properties formalizing the second law
previously introduced. Let $ \ \mathrm{Orb}(\alpha) \ $ be the set of $\ \alpha$-orbits.

\begin{defn}{\em Let $\ (X, \pi, \alpha) \ $ be a micro-macro dynamical system with
$\ O_{\pi}=\{k_1 < \cdots < k_o \}. \ $
\begin{enumerate}
\item  $\ (X, \pi, \alpha)\in  \mathrm{G}_0(\varepsilon) \ $ if and only if
 $ \ |c^{\mathrm{eq}}| \geq (1-\varepsilon)|c| \ $ for $ \ c \in \mathrm{Orb}(\alpha). $
\item  $\ (X, \pi, \alpha)\in  \mathrm{G}_1(\varepsilon) \ $ if and only if
 $\ |Dc| \leq \varepsilon |c| \ $ for
$\ c \in \mathrm{Orb}(\alpha). \  $
\item  $\ (X, \pi, \alpha)\in  \mathrm{G}_2(\varepsilon_1, \varepsilon_2)  \ $ if and only if
$ \ |c^{\mathrm{eq}}| \geq (1-\varepsilon_1)|c| \ $ and  $\  |Ic|  \geq (1-\varepsilon_2)|c^{\mathrm{neq}} | \ $
 for $\  c \in \mathrm{Orb}(\alpha). $
 \item $\ (X, \pi, \alpha)\in  \mathrm{G}_3(\varepsilon_1, \varepsilon_2)  \ $ if and only if $ \ |c^{\mathrm{eq}}| \geq (1-\varepsilon_1)|c| \ $
and  $\ |I(\hat{\pi}_{i} \cap c)|  \geq  (1-\varepsilon ) |\hat{\pi}_{i}\cap c|\ $
 for $\ i \in [o-1]\ $ and $\ c \in \mathrm{Orb}(\alpha). $
\item $\ (X, \pi, \alpha)\in  \mathrm{G}_4(\varepsilon_1, \varepsilon_2)  \ $ if
$ \ |c^{\mathrm{eq}}| \geq (1-\varepsilon_1)|c| \ $
and  $\ |I(a\cap c)|  \geq  (1-\varepsilon_2 ) |a\cap c|\ $
 for $\ a \in \pi^{\mathrm{neq}}\ $ and $ \ c \in \mathrm{Orb}(\alpha). $
\end{enumerate}
}
\end{defn}

\begin{thm}\label{lo}
{\em Let $\ (X, \pi, \alpha) \  $ be an invertible equilibrium bound micro-macro dynamical system
satisfying property $\ \mathrm{G}_0(\varepsilon), \ $  and let
$\ r,e: X \longrightarrow \mathbb{N} \ $ be  the first return time and  the equilibrium reaching time maps,
respectively. We have that $\ \ \displaystyle \overline{\frac{e}{r}}\leq  \varepsilon. $
}
\end{thm}

\begin{proof} By definition the maps $\ r,e: X \longrightarrow  \mathbb{N} \ $ are such that
$\ r(i)=|c| \ $ if $\ i \in c \in \mathrm{Cyc}(\alpha), \ $ and $ \  e(i) \  $ is the smallest $\ k \in \mathbb{N} \ $
with $\ \alpha^k(i) \in X^{\mathrm{eq}}. \ $ Since $\ e(i) \leq  |c^{\mathrm{neq}}| \leq \varepsilon |c|\ $ for
$ \ i \in c, \ $ we have
$$ \overline{\frac{e}{r}}\ = \ \frac{1}{|X|}\sum_{i\in X}\frac{e(i)}{r(i)}
\ = \ \frac{1}{|X|}\sum_{c\in \mathrm{Cyc}}\sum_{i\in c}\frac{e(i)}{r(i)} \ \leq \
\frac{1}{|X|}\sum_{c\in \mathrm{Cyc}}\frac{\varepsilon |c|}{|c|}|c| \ = \
 \frac{\varepsilon}{|X|} \sum_{c\in \mathrm{Cyc}}|c| \ = \ \varepsilon.$$
\end{proof}

\begin{rem}{\em Theorem \ref{lo} is consistent with Boltzmann's response to Zermelo's critique of his $H$-theorem:
recurrence, even if it holds for all microstates, occurs  long after a microstate
have evolved to the equilibrium where it remains for a long period of time, making recurrence of little practical
importance. We leave open the problem of determining if an analogue of  Theorem \ref{lo} holds when
$\ r\ $ is replaced  by the first block return map, or by the first zone return map. }
\end{rem}

The proofs of the following results are similar to those of Theorems \ref{jj}, \ref{kk} and \ref{allivoy}.

\begin{thm}{\em Let $\ (X, \pi, \alpha) \ $ be an invertible micro-macro dynamical system.

\begin{enumerate}
  \item If  $\ (X, \pi, \alpha)  \in  \mathrm{G}_0(\varepsilon), \ $  then
$\ |X^{\mathrm{eq}}| \geq (1-\varepsilon_1)|X| \ $ and  $\ (X, \pi, \alpha)  \in \mathrm{G}_1(\varepsilon). \ $
\item If  $\ (X, \pi, \alpha)  \in  \mathrm{G}_1(\varepsilon), \ $  then
$ (X, \pi, \alpha)  \in \mathrm{L}_1(\varepsilon). \ $
  \item If $\ (X, \pi, \alpha) \in \mathrm{G}_i(\varepsilon_1, \varepsilon_2), \ $ then
$\ (X, \pi, \alpha) \in \mathrm{L}_i(\varepsilon_1, \varepsilon_2), \ $ for $\ i=2,3,4.$
  \item If $\ \alpha\ $ is a zero-jump permutation, then $\  |Dc^{\mathrm{eq}}| \geq (1-2\varepsilon_2)|c^{\mathrm{neq}}|. $
 \item If $\ (X, \pi, \alpha)  \in  \mathrm{G}_0(\varepsilon_1), \ $
 and  $\ | Dc^{\mathrm{eq}}|  \geq   (1 - \varepsilon_2)|c^{\mathrm{neq}}| \ $
for all cycles $\ c \in \mathrm{Cyc}(\alpha), \ $ then $\ (X,\pi, \alpha) \in \mathrm{G}_2(\varepsilon_1,\varepsilon_2). $
\item  Let $(X,E) $ be a simple graph and $\ \sigma \ $ be another partition on $\ X. \ $
 Let $\ \mathrm{P}\ $ be the set of pairs
$\ (a,s)\in \pi\times \sigma \ $ such that $a\cap s\neq \emptyset, \ $ and let
$\ \mathfrak{E} \ $ be the simple graph  on $\ P \ $ such that
there is an edge between
$\ (a,s)\ $ and $\ (b,t)\ $ in $\ \mathfrak{E}\ $ if and only if  $\ s=t\ $ and there are microstates
$\  i \in a\cap s \ $
and $\ j\in b\cap s \ $ such that $\ \{i,j\} \in E. \ $
For $\ (a,s)\in \mathrm{P} \ $ set
$$M_{as} \ \ = \  \underset{\{(b,s)\in \mathrm{P }\ | \ |b|< |a|, \
\{(a,s),(b,s) \} \in \mathfrak{E} \}}{\bigsqcup} b\cap s \ \ \ \subseteq \ \ X.$$
Assume that $\ (X, \pi, \alpha)  \in  \mathrm{G}_0(\varepsilon_0), \ $  $\ \mathrm{Cyc}(\alpha)=\sigma,\ $
$ \ \alpha \ $ is $E$-1-Lipschitz continuous,  there are constants
 $\ \delta_1, \delta_2, \gamma_{a,s} \in \mathbb{R}_{\geq 0} \ $
for $\ a\in \pi^{\mathrm{neq}} , \ (a,s)\in \mathrm{P} $  (with $\gamma_{as}=0\ $ if $\ M_{as}=\emptyset$) \ such that
$ \ |M_{as}|  \leq \gamma_{as} |a\cap s| \ $ with $\ \gamma_{as} \leq \varepsilon_2, \ $
 and $\ |C(a\cap s)| \leq (\varepsilon_2 - \gamma_{ac})|a\cap s|,\ $
then  we have that $\ (X, \pi, \alpha) \in \mathrm{G}_4(\varepsilon_1, \varepsilon_2). $
 \end{enumerate}

}
\end{thm}

\section{ Second Law and Convex Geometry}\label{sllip}

In this section we show that several problems arising from the combinatorial formalizations of the second law
can be equivalently reformulated as problems in convex geometry and integer programming  \cite{m, y}, namely the problem of
computing integer sums over lattice points in convex polytopes. This equivalence allows us to analyze a few simple but
interesting examples, and provides a pathway towards numerical computations.
Given a convex polytope $\ P \subseteq \mathbb{R}^{d} \ $ we set $\ P^{\mathbb{Z}} =
P \cap \mathbb{Z}^{d}. \ $  \\

Let $\ (X,\pi) \ $ be a micro-macro dynamical system with
$\  O_{\pi}  = \{k_1 <   ...  < k_o\} \ $ and consider integers $ \ 0 \leq  d, e  \leq  |X| -|\widehat{\pi}_r|, \ $
where $\ r \ $ is such that $\  |\widehat{\pi}_{r}|  \geq  |\widehat{\pi}_i| \ $ for  $\ i\in [o]. \ $ Let
 $\ \Lambda_{d}  \subseteq \mathbb{R}_{\geq 0}^{o^2} \ $ be the convex polytope given by
$$\sum_{j=1}^o x_{ji}   =  |\widehat{\pi}_{i}|, \ \ \ \ \ \ \  \ \sum_{j=1}^o x_{ij}  =  |\widehat{\pi}_{i}|,
\ \ \ \ \ \ \ \ \sum_{i<j}x_{ij} =  d.$$
Let $\ \Lambda_{d}^e  \subseteq  \mathbb{R}_{\geq 0}^{o^2} \ $ be the convex polytope given by
$$\sum_{j=1}^o x_{ji} =  |\widehat{\pi}_{i}|, \ \ \ \ \ \ \  \ \sum_{j=1}^o x_{ij} =  |\widehat{\pi}_{i}|,
 \ \ \ \ \ \ \  \ \sum_{i<j}x_{ij}  =  d, \ \ \ \ \ \ \  \ \sum_{i>j}x_{ij}  =   e.$$

\noindent Let $\ \Upsilon_{d}  \subseteq  \mathbb{R}_{\geq 0}^{o-1} \ $ be the convex polytope given by
$$\sum_{i=1}^{o-1}x_i = d, \ \ \  \ x_{1} \leq  |\widehat{\pi}_{1}|, \ \ \ \  x_{o-1} \leq  |\widehat{\pi}_{o}|,
 \ \ \ \ x_{i-1} +  x_{i}  \leq   |\widehat{\pi}_{i}| \ \ \ \mbox{for} \ \ \ 2 \leq i \leq o-1.$$

\begin{thm}\label{cg}
{\em  Let $\ (X,\pi) \ $ be a micro-macro dynamical system with
$\  O_{\pi}  = \{k_1 <   ...  < k_o\}. \ $

\begin{enumerate}

  \item A random permutation in $\ \mathrm{S}_X \ $ has  $\ d\ $ strict decreases in entropy with probability
  $$  \binom{|X|}{|\widehat{\pi}_1|,\hspace*{.1cm}\dots{}\hspace*{.1cm},|\widehat{\pi}_o|}^{-1}
\sum_{a \in \Lambda_{d}^{\mathbb{Z}}} \ \prod_{i=1}^o {|\widehat{\pi}_{i}| \choose a_{1i},...,a_{oi}} .$$

\item A random permutation in $\ \mathrm{S}_X \ $ has $\ d\ $ strict decreases and
$\ e \ $ strict increases in entropy with probability
  $$ \binom{|X|}{|\widehat{\pi}_1|,\hspace*{.1cm}\dots{}\hspace*{.1cm},|\widehat{\pi}_o|}^{-1}
  \sum_{a \in \Lambda_{d}^{e,\mathbb{Z}}} \ \prod_{i=1}^o{|\widehat{\pi}_{i}| \choose a_{1i},...,a_{oi}} .$$

\item A random permutation in $\ \mathrm{S}_{X,\pi}^0 \ $  has $\ d\ $ strict decreases in entropy (and thus
$ d $ strict increases) with probability
  $$  \frac{|\widehat{\pi}_1|!^2\hspace*{.1cm}\dots{}\hspace*{.1cm}|\widehat{\pi}_o|!^2}{|\mathrm{S}_{X,\pi}^0|}
   \sum_{a \in \Upsilon_{d}^{\mathbb{Z}}} \big( \prod_{i=1}^o a_i!^2  (\pi_i - a_{i} - a_{i-1})! \big)^{-1},$$
    where  we  set $\ a_0=0\ $ and $\ a_o=0.$
\end{enumerate}

}
\end{thm}

\begin{proof}Items 1 and 2 are similar.  For item 1,  it is enough to show that the number of permutations
$ \alpha \in \mathrm{S}_X $ such that $\ (X, \pi, \alpha) \ $
has $\ d \ $ strict decreases in entropy is given by
$$\prod_{i=1}^o |\widehat{\pi}_{i}|!\sum_{a \in \Lambda_{d}^{\mathbb{Z}}}  \
 \prod_{i=1}^o{|\widehat{\pi}_{i}| \choose a_{1i},...,a_{oi}}. $$
A permutation $\ \alpha \in \mathrm{S}_X \ $ determines the matrix $\ \big(a_{ij}\big) \in \mathbb{N}^{o^2} \ $
 given by
$$a_{ij} =   \big| \{ s \in \widehat{\pi}_{j} \  |  \ \alpha(s) \in \widehat{\pi}_{i}  \} \big| .$$
The desired result follows from the identities
$$\sum_{j=1}^o a_{ji} = |\widehat{\pi}_{i}|, \ \ \ \ \ \ \  \ \sum_{j=1}^o a_{ij}  =    |\widehat{\pi}_{i}|,
\ \ \ \ \ \ \ \ |D(X,\pi,\alpha)|   =   \sum_{i<j}a_{ij}.$$
Moreover each matrix $\  \big(a_{ij}\big) \in \mathbb{N}^{o^2} \ $ satisfying the left and center identities
 above comes from a permutation $\ \alpha \in \mathrm{S}_X. \ $
Indeed, there are
$$\prod_{i,j=1}^oa_{ij}!\prod_{i=1}^o{|\widehat{\pi}_{i}| \choose a_{i1}...a_{io}}{|\widehat{\pi}_{i}|
 \choose a_{1i}...a_{oi}}
\  = \  \prod_{i=1}^o |\widehat{\pi}_{i}|!{|\widehat{\pi}_{i}| \choose a_{1i}...a_{oi}}$$
permutations with $\ \big(a_{ij}\big) \ $ as their associated matrix. Item 3 follows from item 2 after setting
$\ x_i = x_{i, i+1} \ $ for $\ i \in [o-1],\ $ taking into account that  $\ x_{i, i+1}=x_{i+1, i}, \ $ and
$\ x_{ij}=0 \ $ unless $\ i=j,\ i=j+1, \ $ or $\ i=j-1.$
\end{proof}

Examples \ref{a} and \ref{b} below show that  a random permutation is more likely to have exactly one decrease
 in entropy than being always increasing in entropy.

\begin{exmp}\label{a}
{\em Under the hypothesis of Theorem \ref{cg} \ set $\ d=0.\ $  If $\ (a_{ij})  \in \Lambda_{\pi,0}^{\mathbb{Z}},
\ $ then by definition $\ a_{ij}=0 \ $ for $\ i < j \ $, which implies that $\ a_{ij}  =  |\widehat{\pi}_{i}|\delta_{ij}. \ $ Thus
there are
$$\prod_{i=1}^o |\widehat{\pi}_{i}|!\sum_{a \in \Lambda_{0}^{\mathbb{Z}}}  \ {|\widehat{\pi}_{i}| \choose a_{1i}...a_{oi}}
\ =  \ \prod_{i=1}^o |\widehat{\pi}_{i}|!{|\widehat{\pi}_{i}| \choose 0...|\widehat{\pi}_{i}| ...0}
\ = \ \prod_{i=1}^o |\widehat{\pi}_{i}|!
$$ entropy preserving permutations. We have recovered Theorem \ref{parfi}.
}
\end{exmp}

\begin{exmp}\label{b}
{\em Under the hypothesis of Theorem \ref{cg} \ set $\ d=1. \ $  A
matrix $\ (a_{ij}) \in \Lambda_{1}^{\mathbb{Z}}\ $ is uniquely determined by a set of indices
$\ \{i_1<...< i_l\}  \subseteq [o]\ $, with $\ 2 \leq l \leq o, \ $ such that $\ a_{ij} =  0  \ $ for $\ i \neq j, \ $ except for $ \ a_{i_1i_l}  =  a_{i_2i_1}  = \cdots  =  a_{i_{l}i_{l-1}}  =  1. \ $
Thus there are
$$ \prod_{i=1}^o |\widehat{\pi}_{i}|!
\sum_{\{i_1<...< i_l\}  \subseteq  [o]}\
{|\widehat{\pi}_{i_l}| \choose 0.... \underset{i_{1}\uparrow}{1} .... |\widehat{\pi}_{i_l}|-1  .... 0}
\prod_{s=1}^{l-1}
 {|\widehat{\pi}_{i_s}| \choose 0...|\widehat{\pi}_{i_s}|-1 .... \underset{i_{s+1}\uparrow}{1} .... 0}$$
permutations with exactly one  strict decrease in entropy.
}
\end{exmp}

\begin{thm}{\em Let $\ (X, \pi) \ $ be a micro-macro phase space   with
$\ O_{\pi} =  \{k_1 < k_2 \} \ $ and $\ |\widehat{\pi}_{1}|  <  |\widehat{\pi}_{2}|. \ $
\begin{enumerate}
\item A random permutation on $\ X \ $ has $\ d \ $ strict decreases in entropy with probability
$$ {|X| \choose \widehat{\pi}_{1}}^{-1}{|\widehat{\pi}_{1}| \choose d}{|\widehat{\pi}_{2}| \choose d}.$$
\item A random permutation is most likely to have
$\ \displaystyle
\bigg\lfloor\frac{|\widehat{\pi}_{1}||\widehat{\pi}_{2}|}{|\widehat{\pi}_{1}| \ + \ |\widehat{\pi}_{2}| \ +  \ 2 }\bigg\rfloor \ $
strict decreases in entropy.
\item  If $\  |\widehat{\pi}_{1}| \ $ is fixed and $\ |\widehat{\pi}_{2}| \ $ grows to infinity, then a random permutation is most likely to have a relatively vanishing number of strict decreases in entropy.
\item If $\  |\widehat{\pi}_{2}|  =   c|\widehat{\pi}_{1}| \ $  and $\ |\widehat{\pi}_{1}|\ $ grows to infinity, then a random permutation is most likely to have a relative number of strict decreases entropy of
$\   \frac{c}{(1+c)^2}. \ $
\item If $\ |\widehat{\pi}_{2}|= c|\widehat{\pi}_{1}|^s \ $ with $\ s >1 \ $  and $\ |\widehat{\pi}_{1}|\ $ grows to infinity, then a random permutation is most likely to have  a relatively vanishing number of strict decreases in entropy.
\item If $\ |\widehat{\pi}_{2}| = c e^{|\widehat{\pi}_{1}|} \ $   and $\ |\widehat{\pi}_{1}|\ $ grows to infinity, then a random permutation is most likely
to have  a relatively vanishing number of strict decreases in entropy.
\end{enumerate}
}
\end{thm}

\begin{proof}
 A permutation on $ X  $ with $ d  $ decreases in entropy gives
rise to a matrix $ (a_{ij}) $ with
$$ \ a_{12}   =  d, \ \ \  a_{11} + a_{21}  =  |\widehat{\pi}_{1}|, \ \   a_{11} + a_{12}  =  |\widehat{\pi}_{1}|, \ \
a_{12} + a_{22}  =  |\widehat{\pi}_{2}|, \ \  \mbox{and} \ \   a_{21} + a_{22} =  |\widehat{\pi}_{2}|. \ $$
Thus $\ \ \ 0 \leq d \leq |\widehat{\pi}_{1}|,  \ \ \ \ a_{12}=a_{21}=d , \ \ $ $\ \ a_{11} = |\widehat{\pi}_{1}| -d,
\ \ \ $ and $ \ \ \ a_{22} = |\widehat{\pi}_{2}| - d, \ \ $ and there are
$$|\widehat{\pi}_{1}|!|\widehat{\pi}_{2}|! {|\widehat{\pi}_{k_1}| \choose d}{|\widehat{\pi}_{2}| \choose d} $$
permutations of $\ X \ $ with $\ d \ $ strict decreases in entropy.\\

To find out the most likely number of strict decreases in entropy, we should find the integer $\ 0 \leq d \leq a\ $ for which the product
$\ {a \choose d}{a+k \choose d} \ $ achieve its maximum. It is not hard to check  that
$$\frac{ {a \choose e}{a+k \choose e} }{ {a \choose e+1}{a+k \choose e+1} }   \leq   1
\ \ \ \ \ \mbox{is equivalent to} \ \ \ \ \ d   \leq   \frac{a^2 + ak}{2a+k+2} . $$
Therefore a random permutation is most likely to have
$$d \ \simeq \ \bigg\lfloor\frac{|\widehat{\pi}_{1}|^2 \ + \ |\widehat{\pi}_{1}|(|\widehat{\pi}_{2}| \ - \  |\widehat{\pi}_{1}|)}
{2|\widehat{\pi}_{1}| \ + \ |\widehat{\pi}_{2}| \ - \ |\widehat{\pi}_{1}| \ + 2 }\bigg\rfloor\  =  \
\bigg\lfloor\frac{|\widehat{\pi}_{1}||\widehat{\pi}_{2}|}{|\widehat{\pi}_{1}| \ + \ |\widehat{\pi}_{2}| \ +  \ 2 }\bigg\rfloor .$$
If $\  |\widehat{\pi}_{1}| \ $ is fixed and $\ |\widehat{\pi}_{2}| \ $ grows to infinity
then a permutation is most likely to have $\ d   \simeq  |\widehat{\pi}_{1}| \ $
strict increases in entropy, thus
$\ \frac{|\widehat{\pi}_{1}|}{|\widehat{\pi}_{1}|\ + \ |\widehat{\pi}_{2}|} \  \simeq  \ \frac{|\widehat{\pi}_{1}|}{|\widehat{\pi}_{2}|}
\  \simeq  \ 0. \ \ $
If $\ |\widehat{\pi}_{2}| =  c|\widehat{\pi}_{1}|\ $ grows to infinity, we have that
$$\ d \ \simeq \ \bigg\lfloor \frac{c|\widehat{\pi}_{1}| }{1+c} \bigg\rfloor \ \ \ \ \ \ \ \ \ \mbox{and}\
 \ \ \ \ \ \ \ \ \frac{d}{|\widehat{\pi}_{1}|\ + \ |\widehat{\pi}_{2}|} \
\simeq \   \frac{c}{(1+c)^2}.$$
If $\ |\widehat{\pi}_{2}| =  |\widehat{\pi}_{1}|^s \ $ with $\ s >1 \ $ and $\ |\widehat{\pi}_{1}| \ $ growing to infinity, we get that
the most likely number of strict decreases in entropy is given by
$$d\  \simeq   \  \bigg\lfloor\frac{c|\widehat{\pi}_{1}|^{s+1}}{|\widehat{\pi}_{1}| \ + \ c|\widehat{\pi}_{1}|^s \ + \ 2}
\bigg\rfloor \  \simeq  \
|\widehat{\pi}_{1}|, \ \ \ \ \mbox{thus} \ \ \ \  \frac{d}{|\widehat{\pi}_{1}| \ + \ c|\widehat{\pi}_{1}|^s } \  \simeq \
\frac{1}{c|\widehat{\pi}_{1}|^{s-1}} \simeq \ 0. $$
If $\ |\widehat{\pi}_{2}|=  e^{|\widehat{\pi}_{1}|} \ $ and $\ |\widehat{\pi}_{1}| \ $ growing to infinity, we get that
the most likely number of strict decreases in entropy is given by
$$d\  \simeq   \ \bigg\lfloor\frac{c|\widehat{\pi}_{1}|e^{|\widehat{\pi}_{1}|}}{|\widehat{\pi}_{1}| \ + \ ce^{|\widehat{\pi}_{1}|} \ + \ 2} \bigg\rfloor \  \simeq  \
|\widehat{\pi}_{1}|, \ \ \ \ \mbox{thus} \ \ \ \  \frac{|\widehat{\pi}_{1}|}{|\widehat{\pi}_{1}| \ + \ ce^{|\widehat{\pi}_{1}|}} \  \simeq \
\frac{|\widehat{\pi}_{1}|}{c}e^{-|\widehat{\pi}_{1}|}\ \simeq \ 0. $$

\end{proof}

Consider again a micro-macro phase space  $\ (X,\pi) \ $  with
$\  O_{\pi} =  \{k_1 <   ...  < k_o\}, \ $ and let
$\ \Psi_{\varepsilon_1,\varepsilon_2} \subseteq  \mathbb{R}_{\geq 0}^{o^2} \ $ be the convex polytope given by
$$\sum_{j=1}^o x_{ji}   = |\widehat{\pi}_{i}|, \ \ \ \ \sum_{j=1}^o x_{ij}  =  |\widehat{\pi}_{i}|,
 \ \ \   \sum_{i<j}x_{ij}  \leq  \varepsilon_1|X|, \ \ \ \ \sum_{i>j}x_{ij}  \geq   (1-\varepsilon_2)|X^{\mathrm{neq}}|.$$
Let $\ \Theta_{\varepsilon_1,\varepsilon_2} \subseteq  \mathbb{R}_{\geq 0}^{o^2} \ $ be the convex polytope given by
$$\sum_{j=1}^o x_{ji}   = |\widehat{\pi}_{i}|, \ \ \ \ \sum_{j=1}^o x_{ij}  =  |\widehat{\pi}_{i}|,
 \ \ \   \sum_{i<j}x_{ij}  \leq  \varepsilon_1|X|, \ \ \ \ \sum_{i>j}x_{ij}  \geq   (1-\varepsilon_2)|\widehat{\pi}_{j}|.$$
Let $\ \Omega_{\varepsilon_1,\varepsilon_2} \subseteq  \mathbb{R}_{\geq 0}^{o-1} \ $ be the convex polytope given by
$$\sum_{i=1}^{o-1}x_i \leq \varepsilon_1|X|, \ \  x_{1} \leq  |\widehat{\pi}_{1}|, \ \   x_{o-1} \leq  |\widehat{\pi}_{o}|,
 \ \  x_{i-1} +  x_{i}  \leq   |\widehat{\pi}_{i}| \ \ \mbox{and} \ \  x_{i}  \geq  (1-\varepsilon_2) |\widehat{\pi}_{i}|.$$

\begin{thm}\label{cg2}
{\em   Let $\ (X,\pi) \ $ be a micro-macro dynamical system with
$\  O_{\pi}  = \{k_1 <   ...  < k_o\}. \ $
\begin{itemize}
  \item  A random permutation $\ \alpha \in \mathrm{S}_X \ $ determines a system  $\   (X,\pi, \alpha)
   \in \mathrm{L}_2(\varepsilon_1,\varepsilon_2) \ $ with probability
  $$  \binom{|X|}{|\widehat{\pi}_1|,\hspace*{.1cm}\dots{}\hspace*{.1cm},|\widehat{\pi}_o|}^{-1}\sum_{a \in
  \Psi_{\varepsilon_1, \varepsilon_2}^{\mathbb{Z}}} \ \prod_{i=1}^o {|\widehat{\pi}_{i}| \choose a_{1i},...,a_{oi}}. $$
  \item  A random permutation $\ \alpha \in \mathrm{S}_X \ $ determines a system  $\   (X,\pi, \alpha)  \in \mathrm{L}_3(\varepsilon_1,\varepsilon_2) \ $ with probability
  $$  \binom{|X|}{|\widehat{\pi}_1|,\hspace*{.1cm}\dots{}\hspace*{.1cm},|\widehat{\pi}_o|}^{-1}\sum_{a \in
  \Theta_{\varepsilon_1, \varepsilon_2}^{\mathbb{Z}}} \ \prod_{i=1}^o {|\widehat{\pi}_{i}| \choose a_{1i},...,a_{oi}}. $$
\item  A random permutation in $\ \mathrm{S}_{X,\pi}^0 \ $  determines a system  $\   (X,\pi, \alpha)  \in \mathrm{L}_3(\varepsilon_1,\varepsilon_2) \ $ with probability
  $$  \frac{|\widehat{\pi}_1|!^2\hspace*{.1cm}\dots{}\hspace*{.1cm}|\widehat{\pi}_o|!^2}
  {|\mathrm{S}_{X,\pi}^0|} \sum_{a \in \Omega_{\varepsilon_1, \varepsilon_2}^{\mathbb{Z}}} \big( \prod_{i=1}^o a_i!^2
  (\widehat{\pi}_i - a_{i} - a_{i-1})! \big)^{-1},$$ where  we  set $\ a_0=0\ $ and $\ a_o=0.$
\end{itemize}
}
\end{thm}

The associated matrix of a reversible system  is symmetric
$$a_{ij}  \ =  \ \big| \{ s \in \widehat{\pi}_{j} \  |  \ \alpha(s) \in \widehat{\pi}_{i}  \}  \big| \ = \
 \big| \{ s \in \widehat{\pi}_{i} \  |  \ \alpha(s) \in \widehat{\pi}_{j}  \} \big| \ = \ a_{ji},  $$
since  the map $\ r\alpha:  \{ s \in \widehat{\pi}_{j} \  |  \ \alpha(s) \in \widehat{\pi}_{i}  \} \longrightarrow
\{ s \in \widehat{\pi}_{i} \  |  \ \alpha(s) \in \widehat{\pi}_{j}  \} \ $ is a bijection.
So it is interesting to consider permutations for which the symmetry condition
$\ a_{ij} = a_{ji} \ $ holds. We call such systems symmetric  and let $\ \mathfrak{S}_{X,\pi} \ $ be
 the set of symmetric permutations on $\ X. \ $ Note  that zero jump permutations are symmetric.
Let $\ \Sigma_{\varepsilon_1,\varepsilon_2} \subseteq  \mathbb{R}_{\geq 0}^{{o+1 \choose 2}} \ $ be
 the convex polytope given on  $\ x_{ij}\ $ with $\ 1 \leq i \leq j \leq o \ $ by
$$\sum_{j=1}^{i} x_{ji} + \sum_{j=i+1}^o x_{ij}=   |\widehat{\pi}_{i}|, \ \ \
\sum_{i<j}x_{ij} \geq (1-\varepsilon_2)|\widehat{\pi}_{i}|  \ \ \mbox{for} \ \ i\in [o-1], \ \ \
 \sum_{i<j}x_{ij} \leq  \varepsilon_1|X|.$$

\begin{thm}{\em Let $\ (X,\pi) \ $ be a micro-macro dynamical system with
$\  O_{\pi}  = \{k_1 <   ...  < k_o\}. \ $ A random invertible symmetric system $\ (X, \pi, \alpha) \ $ has property  $\  \mathrm{L}_3(\varepsilon_1,\varepsilon_2) \ $ with probability
 $$ \frac{\prod_{i=1}^o |\widehat{\pi}_{i}|!}{|\mathfrak{S}_{X,\pi}|} \sum_{a \in
 \Sigma_{\varepsilon_1, \varepsilon_2}^{\mathbb{Z}}} \ \prod_{i=1}^o {|\widehat{\pi}_{i}| \choose a_{1i},...,a_{oi}}. $$

}
\end{thm}

\

Let $\ \Phi_{\varepsilon_1,\varepsilon_2}  \subseteq  \mathbb{R}_{\geq 0}^{o^2} \ $ be the convex polytope given by
$$ \sum_{j=1}^o x_{ji}+x_{ij} =  \ 2|\widehat{\pi}_{i}|, \ \ \ \ \ \ \
(1-2\varepsilon_1)|X|  \leq    \sum_{i}x_{ii}   \leq
 2(1-\varepsilon_2) |X^{\mathrm{eq}}|  -   (1-2\varepsilon_2)|X|.$$
The following result follows from Theorem \ref{cf}.

\begin{thm}\label{cg1}
{\em  Let $\ (X,\pi) \ $ be a micro-macro dynamical system with
$\  O_{\pi}  = \{k_1 <   ...  < k_o\}. \ $
\begin{enumerate}
  \item A random permutation $\ \alpha \in \mathrm{S}_X \ $  determines an (invariant, equivariant) reversible system
$ \ IR(X,\pi,\alpha) \ $ in $\ \mathrm{L}_2(\varepsilon_1,\varepsilon_2)\  $ with probability
$$  \binom{|X|}{|\widehat{\pi}_1|,\hspace*{.1cm}\dots{}\hspace*{.1cm},|\widehat{\pi}_o|}^{-1}\sum_{a \in
\Phi_{\varepsilon_1, \varepsilon_2}^{\mathbb{Z}}} \ \prod_{i=1}^o {|\widehat{\pi}_{i}| \choose a_{1i},...,a_{oi}}. $$

  \item If $\ O_{\pi} =  \{k_1 < k_2 \}, \ $ then a random permutation $\alpha \in \mathrm{S}_X$ determines
  an (invariant, equivariant) reversible system
  $ \ IR(X, \pi,\alpha) \ $ in $\ \mathrm{L}_2(\varepsilon_1,\varepsilon_2) \ $ with probability
$$ \binom{|X|}{|\widehat{\pi}_{1}|}^{-1}\underset{(1-\varepsilon_2) |\widehat{\pi}_{1}| \leq  d \leq
\mathrm{min}(|\widehat{\pi}_{1}|, |\widehat{\pi}_{2}|, \varepsilon_1 |X|)}
\sum  {|\widehat{\pi}_{1}| \choose d}{|\widehat{\pi}_{2}| \choose d} .$$
\end{enumerate}
}
\end{thm}

\

 Let $\ \Theta_{\varepsilon_1, \varepsilon_2} \subseteq
 \mathbb{R}_{\geq 0}^{o^2} \ $ be the convex polytope  given by
$$\sum_{j=1}^o x_{ij}+x_{ji}   = 2|\widehat{\pi}_{i}|, \ \ \ \ \ \  \ \
\sum_{j=1}^{o-1}x_{jo}+x_{oj}  \geq     2(1-\varepsilon_2)|X^{\mathrm{neq}}|.$$
Next  result is a consequence of Theorems \ref{jj} and \ref{kk}.

\begin{thm}\label{kkkk}
{\em   Let $\ (X, \pi) \ $ be  micro-macro phase space with  $\  |X^{\mathrm{eq}}|   \geq   (1-\varepsilon_1)|X|. \  $ A random
 permutation $ \ \alpha \in \mathrm{S}_X \ $  determines an (invariant, equivariant) reversible system
 $\ R(X,\pi,\alpha)\ $ in $\ \mathrm{L}_2(\varepsilon_1,\varepsilon_2) \ $ with probability greater than
$$  \binom{|X|}{|\widehat{\pi}_{1}|,\hspace*{.1cm}\dots{}\hspace*{.1cm},|\widehat{\pi}_{o}|}^{-1}\sum_{a \in
\Theta_{\varepsilon_1, \varepsilon_2}^{\mathbb{Z}}}  \prod_{i=1}^o {|\widehat{\pi}_{i}| \choose a_{1i},...,a_{oi}}, $$
and less than
$$  \binom{|X|}{|\widehat{\pi}_{1}|,\hspace*{.1cm}\dots{}\hspace*{.1cm},|\widehat{\pi}_{o}|}^{-1}\sum_{a \in
 \Theta_{\varepsilon_1, 2\varepsilon_2}^{\mathbb{Z}}}  \prod_{i=1}^o {|\widehat{\pi}_{i}| \choose a_{1i},...,a_{oi}}, $$
}
\end{thm}

We close this section describing property $\ \mathrm{L}_4(\varepsilon_1,\varepsilon_2) \ $ in terms of convex polytopes.
Let $\ (X, \pi, E) \ $ be a micro-macro phase-space with $ \ (X,E)\ $  a simple graph. Recall that
$\ (\pi,\mathcal{E})\ $ denotes the  induced simple graph on macrostates. Let
$\ \Gamma_{\varepsilon_1,  \varepsilon_2} \subseteq \mathbb{R}_{\geq 0}^{A\times A} \ $ be the convex polytope given by
$$\sum_{a\in A}^o x_{ab}   = |b|, \ \ \ \ \sum_{b\in A}^o x_{ab}   = |a|,
 \ \ \   \sum_{|a|<|b|, \{a,b\}\in \mathcal{E}}x_{ab}  \leq  \varepsilon_1|X|, \ \ \ \  \sum_{|a|>|b|, \{a,b\}\in \mathcal{E}}x_{ab} \geq   (1-\varepsilon_2)|b|.$$

\begin{thm}\label{kkkk}
{\em   Let $\ (X, \pi, E) \ $ be a micro-macro phase-space with $ \ (X,E)\ $  a simple graph. A random  permutation in
$ \  \mathrm{S}_X^E \ $  determines a  system $\ (X,\pi,\alpha)\ $ in $\ \mathrm{L}_4(\varepsilon_1,\varepsilon_2) \ $ with probability
$$ \frac{\prod_{a\in A}|a|!}{|\mathrm{S}_X^E|} \sum_{c \in \Gamma^{\mathbb{Z}}_{\varepsilon_1, \varepsilon_2}}  \prod_{b\in A} {|b| \choose (c_{ab})_{a\in A}}. $$
}
\end{thm}

\section{Thermodynamic Limits}\label{tl}

Let $\ f:\mathbb{R}_{\geq 0}^o \longrightarrow  \mathbb{R} \ $ be a map. A thermodynamic (or projective) limit for $\ f \ $ is a limit
$$\underset{x\rightarrow \infty}{\mathrm{Lim}} f(xp_1,...,xp_o)$$
where $(p_1,...,p_o)\in \Delta^{o-1} ,$  i.e. $\ p_i \geq 0\ $ and $\ p_1+ \cdots + p_o=1.\ $
 Assuming that $\ f \ $ can be  written asymptotically as
$\ f(xp_1,...,xp_o) =  x^{\alpha}g(p_1,...,p_n) + o(x^{\alpha}), \ $
the thermodynamic limits of $\ f\ $ are controlled by the map
$\ g:\Delta^{o-1} \longrightarrow  \mathbb{R}. \ $ In our computations below $\ \mathrm{ln}(f) \ $ will have such
asymptotic  behaviour, with $ \ \alpha = 1, \ $ thus $\ \displaystyle f(xp_1,...,xp_o) \simeq e^{ xg(p_1,...,p_n)}. \ $\\

The zone proportions $\ p_i\ $ and transition proportions $\ \lambda_{ij}\ $ of a micro-macro dynamical system $\ (X,\pi, \alpha) \ $  with
$\  O_{\pi} =  \{k_1 <   ...  < k_o\} \ $ are given, respectively, by
$$p_i=\frac{|\widehat{\pi}_j|}{|X|} \ \ \ \ \ \ \mbox{and}  \ \  \ \ \ \  \lambda_{ij}=\frac{\big| \{ \alpha(l) \in \widehat{\pi}_i \ | \ l\in \widehat{\pi}_j \} \big|}{|X|}.$$

\begin{defn}{\em
A thermodynamic limit of micro-macro dynamical systems with zone proportions $\ p_i \ $ and zone transition proportions
$\ \lambda_{ij} \ $ is a sequence $\ (X_n,\pi(n), \alpha_n) \ $ such that:
\begin{itemize}
  \item $\displaystyle  O_{\pi(n)} =  \{k_1(n) <   ...  < k_o(n)\}, \ \ \ $ and $\ \ \ \displaystyle |X_n|\rightarrow \infty \  $ as $\ n\rightarrow \infty ,\ $
  \item $\displaystyle \underset{n\rightarrow \infty}{\mathrm{Lim}}\ p_j(n)= p_j,\ \ \ \ \mbox{and}
  \ \ \ \ \underset{n\rightarrow \infty}{\mathrm{Lim}}\
\lambda_{ij}(n)= \lambda_{ij}. \ \ $
\end{itemize}
}
\end{defn}

\

\begin{rem}{\em We have already study  thermodynamics limits with zone proportions
$\ p_o=1\ $ and $\ p_j=0  \ $ for $\ j \neq o \ $ in Corollaries \ref{ct1}  and \ref{ct2} and Theorem \ref{ct3}.
}
\end{rem}

\begin{prop}\label{g}
{\em In a thermodynamic limit with zone proportions $\ p_i> 0 \ $ for $\ i \in [o]\ $ a random
permutation has null probability of being always increasing in entropy.
}
\end{prop}

\begin{proof} By Theorem \ref{parfi} and  Stirling's approximation formula the desired probability
is given by the thermodynamic limit
$$\underset{n\rightarrow \infty }{\mathrm{lim}}
\binom{n}{np_1,\hspace*{.1cm}\dots{}\hspace*{.1cm},np_o}^{-1} \ = \
\underset{n\rightarrow \infty }{\mathrm{lim}} e^{-nH(p_1,...,p_o)} =\ 0,$$
where $ \ H(p_1,...,p_o) = -\sum_{i=1}^{o}p_i\mathrm{ln}(p_i)>0 \ $ is the Shannon entropy of
$\ (p_1,...,p_o) \in \Delta^{o-1}.$
\end{proof}

\begin{thm}\label{g2}
{\em In a thermodynamic limit with zone proportions $\ p_i> 0 \ $ for $\ i \in [o]\ $  an invertible micro-macro dynamical system $\ S \ $
has null probability of having transition proportions $\ \lambda_{ij}\geq 0\ $
unless $\  \lambda_{ij}=p_ip_j\  $ which has full probability. If
$\ 0 <p_1 < p_2 < ... < p_o\leq 1\ $ we have that:
\begin{enumerate}
  \item $S \in \mathrm{L}_1(\varepsilon) \ \ $ if and only if $\ \  \sum_{i<j}p_{i}p_j \leq \varepsilon ; \ \ \ $
  If $\ S \in \mathrm{L}_1(\varepsilon), \  $ then $\  p_o(1-p_o)\leq \varepsilon  .$
  \item $S \in \mathrm{GAT}(\varepsilon) \ \ $ if and only if
  $\ \ \displaystyle  \frac{\sum_{i\leq j <o}p_ip_j}{1-p_o} \leq \varepsilon; \ \ \ $
  $S \in \mathrm{ZAT}(\varepsilon) \ \ $ if and only if $\ \  p_o  \geq 1- \varepsilon.$
\item $ |S^{\mathrm{eq}}|   \geq   (1-\varepsilon)|S| \  $ if and only if
 $\  p_o  \geq 1- \varepsilon; \ \  $ $ |DS^{\mathrm{eq}}|   \geq   (1-\varepsilon)|S^{\mathrm{neq}}| \  $ if and only if
 $\ p_o  \geq 1- \varepsilon.$
\item If $\  p_o  \geq 1- \varepsilon, \ $ then  $\ \ S \in
 \mathrm{L}_3(\varepsilon,\varepsilon). \ $
  \item $S\ $ is symmetric and thus $\ |D|=|I|.\ $ The proportionality constants
   of the (invariant or equivariant) reversible system associated to $\ S \ $ agree with those of
  $\ S.$
\item  A micro-state in the zone $\ j \ $  moves
to the equilibrium with probability $\ p_o. \ $ The mean
  jump for such  micro-states is greater than $\ (o-j-1)p_o.$
  \end{enumerate}
  }
\end{thm}

\begin{proof}
By Theorem \ref{cg} we should consider the thermodynamic limit
$$ \underset{n\rightarrow \infty }{\mathrm{lim}}
\binom{n}{np_1,\hspace*{.1cm}\dots{}\hspace*{.1cm},np_o}^{-1}
 \prod_{j=1}^o {np_j \choose np_j \frac{\lambda_{1j}}{p_j},...,np_j \frac{\lambda_{oj}}{p_j}} \ = \ \underset{n\rightarrow \infty }{\mathrm{lim}}
 e^{n\big[\sum_{j=1}^{o}p_jH(\frac{\lambda_{1j}}{p_j},...,\frac{\lambda_{oj}}{p_j})\ - \ H(p_1,...,p_o)\big]},$$
 $$ \mbox{where }  \ \ \ \ \ \sum_{i=1}^o p_i=1,
 \ \ \ \ \ \ \sum_{i=1}^o \lambda_{ij}   =  p_j,
 \ \ \ \ \ \  \ \sum_{i=1}^o \lambda_{ji}  =  p_j.$$
Our next goal is to maximize
$$\sum_{j=1}^{o}p_jH(\frac{\lambda_{1j}}{p_j},...,\frac{\lambda_{oj}}{p_j})- H(p_1,...p_o)\ = \
H(\lambda_{ij}) - 2H(p_1,...p_o)$$ with respect to $\ \lambda_{ij}. \ $ Omitting the $\lambda_{ij}$-independent summand  $2H(p_1,...p_o),$
we maximize $H(\lambda_{ij})$ subject to the above constrains. Applying the Jaynes' max entropy method
we get that max entropy is achieved by  $\ \lambda_{ij}=p_ip_j, \ $ with entropy  $\ H(p_ip_j)= 2H(p_1,...,p_o). \ $ Indeed the maximum entropy distribution is given by
$$\lambda_{ij}= \frac{e^{-f_i-c_j}}{\sum_{ij}e^{-f_i-c_j}} =  \frac{e^{-f_i}}{\sum_{i}e^{-f_i}} \frac{e^{-c_j}}{\sum_{j}e^{-c_j}} = p_ip_j,$$ where $f_i$ and $c_j$ are the Lagrangian multipliers associated with the constrains.
\end{proof}

\begin{cor}
{\em Let  $ \ s(1) > s(2) > ... > s(o) > 0   \ $ be real numbers.  If $\ \sum_{i=1}^{o} e^{- s(i)}=1 \ $
and $\ s(o) \leq -\mathrm{ln}(1-\varepsilon),\ $ then
an invertible micro-macro dynamical system with transition proportions  $\ e^{- s(i)-s(j)} \ $
belongs to $\ \mathrm{L}_3(\varepsilon,\varepsilon). \ $}
\end{cor}

Fix a probability $\ q \ $ on a finite set $\ X \ $ and  let
$\ \displaystyle p_i  = e^{\mathrm{ln}q_i- \sum_{c}\lambda_cf_c(i) -\mathrm{ln}Z(\lambda)} \ $ be
the probability on $ \ X\ $ of minimum relative entropy
$\ \displaystyle D(p|q) \ $
subject to the constrains $\  \sum_{i\in X}p_if_c(i) = a_c, \ $ where $\ f_c: X \longrightarrow \mathbb{R}, \ $
$\ a_c \in \mathbb{R}, \ $  and  $ \ \displaystyle Z(\lambda) =
 \sum_{i\in X} e^{\mathrm{ln}q_i- \sum_{c}\lambda_cf_c(i)}. \ $
After reordering, assume that
$$ \sum_{c}\lambda_cf_c(1) - \mathrm{ln}q_1 \ > \ \cdots \ > \ \sum_{c}\lambda_cf_c(i) - \mathrm{ln}q_i \ > \
 \cdots \ > \ \sum_{c}\lambda_cf_c(o)-\mathrm{ln}q_o.\ $$

\begin{cor}{\em Under the above conditions assume that  $\
\sum_{c}\lambda_cf_c(o) \geq \mathrm{ln}[\frac{(1-\varepsilon)q_o}{Z(\lambda)}],\ $ then
an invertible micro-macro dynamical system with transition proportions $$\ \displaystyle
e^{\mathrm{ln}(q_iq_j)- \sum_{c}\lambda_c(f_c(i)+f_c(j)) -2\mathrm{ln}Z(\lambda)} \ \ \ \
\mbox{belongs to} \ \ \ \mathrm{L}_3(\varepsilon,\varepsilon). \ $$
}
\end{cor}

Next we characterize the most likely invertible micro-macro dynamical system in a thermodynamic limit.

\begin{thm}\label{pr}
{\em In a thermodynamic limit an invertible micro-macro dynamical system  with $\ o \ $ zones most likely have
zone and transition proportions $\ \displaystyle p_i = \frac{1}{o}\ $ and
$ \  \displaystyle \lambda_{ij} = \frac{1}{o^2}.\ $   For $ \ \varepsilon < \frac{1}{4}, \ $ such a system has property $\ \mathrm{L}_1(\varepsilon) \ $   if and only if
$\ o=1.$
  }
\end{thm}

\begin{proof}
We proceed as in Theorem \ref{g2} letting both  $ \ p_i >0 \ $ and $\ \lambda_{ij} >0\ $ vary subject to
$$\sum_{ij}^o \lambda_{ij}=1, \ \ \ \ \ \ \sum_{i=1}^o \lambda_{ij}   =  p_j,
 \ \ \ \ \ \  \ \sum_{i=1}^o \lambda_{ji}  =  p_j.$$
Maximizing
$ \   \displaystyle \sum_{j=1}^{o}p_jH(\frac{\lambda_{1j}}{p_j},...,\frac{\lambda_{oj}}{p_j})- H(p_1,...p_o)\ = \
H(\lambda_{ij}) - 2H(p_1,...p_o)\  $ we get that $\ \displaystyle p_i = \frac{1}{o}\ $ and
$ \  \displaystyle \lambda_{ij} = \frac{1}{o^2},\ $ i.e.
all zones have the same cardinality and transitions between zones  are uniformly  random. By Theorem \ref{g2} we have that
such a system has property $\ \mathrm{L}_1(\varepsilon) \ $ if and only if $ \  \frac{1}{o} \geq 1-2\varepsilon. \ $
In particular, for $ \ \varepsilon < \frac{1}{4} \ $  property   $\ \mathrm{L}_1(\varepsilon) \ $ holds if and only if
$\ o=1.$
\end{proof}

Theorems \ref{g2} and \ref{pr} show the limitations of the proportionality principle
as the only basis of second law. We proceed to supplement it with the continuity principle, making the
assumption that only permutations with jump bounded by $ \ k-1 \geq 0 \ $ are allowed. Note that
setting $\ k=o \ $ we recover the proportionality model discussed above.

\begin{thm}\label{bb}
{\em Consider an invertible micro-macro dynamical system $\ S\ $ in a thermodynamic limit  with $k$-bounded jumps
and  zone proportions $\ p_j. \ $ The system  have vanishing relative probability unless its transition probabilities are
 $\ \lambda_{ij}=b_i b_j \ $ for $\ |i-j|\leq k \ $  and zero otherwise, where
$\ \  b_{j-k}b_j+ \cdots + b_{j-1}b_j + b_j^2 + b_{j+1}b_j+ \cdots + b_{j+k}b_j =  p_j, \ \ $
$\ b_j>0\ $ for $\ j \in [o],\ $ and  $\ b_j =0 \ $ if $\ j \notin [o].\ $ If $\ 0< b_1<...<b_o\ $ we have that:
\begin{enumerate}
  \item $S \in \mathrm{L}_1(\varepsilon) \ \ $ if and only if $\ \
\sum_{j-k\leq i<j}b_{i}b_{j}  \leq \varepsilon .$
  \item $S\in \mathrm{GAT}(\varepsilon) \ \ $ if and only if
$\ \ \displaystyle \sum_{j-k\leq i\leq j<o}b_{i}b_{j}\leq \varepsilon \sum_{j<o, \ |i-j|\leq k}b_{i}b_{j}.$
  \item $S \in \mathrm{ZAT}(\varepsilon) \ \ $ if and only if $\ \ b_{j-k}+\cdots +  b_{j}\leq
  \varepsilon(b_{j-k}+ \cdots + b_j+ \cdots + b_{j+k} ) \ $ for $\ j \in [o-1].$
 \item The proportionality constants of the (invariant or equivariant) reversible system associated to $\ S \ $ agree with those of
  $\ S.$
  \item The average jump of system $ \ S \ $ is bounded by $\ k-1.$
  \end{enumerate}
}
\end{thm}

\begin{proof} Consider the thermodynamic limit
  $$ \underset{n\rightarrow \infty }{\mathrm{lim}}
 \prod_j {np_j \choose
 np_j \frac{\lambda_{j-k,j}}{p_j},...,np_j \frac{\lambda_{jj}}{p_j},...,np_j \frac{\lambda_{j+k,j}}{p_j}} \ = \
  \underset{n\rightarrow \infty }{\mathrm{lim}}
 e^{n\big[\sum_{j=1}^{o}p_jH(\frac{\lambda_{j-k,j}}{p_j},...,
 \frac{\lambda_{jj}}{p_j},...,\frac{\lambda_{j+k,j}}{p_j}) \big]},$$
 with the convention that $\ \lambda_{ij}=0 \ $ if $\ i \notin [o]. \ $
The strict concavity of Shannon's entropy and the convexity of the constrained domain show that the expression has
a unique maximum. Clearly the quotient of any such probability by  the probability of the maximum are vanishing
as $n$ grows to infinity.
Next we maximize
$$ \  \sum_{j=1}^{o}p_jH(\frac{\lambda_{j-k,j}}{p_j},...,
 \frac{\lambda_{jj}}{p_j},...,\frac{\lambda_{j+k,j}}{p_j})  \ \ \
\mbox{subject to} \ \ \  \lambda_{j-k,j} + \cdots + \lambda_{jj} + \cdots + \lambda_{j+k,j}  =  p_j. $$
Associate to each constrain  its Lagrangian multiplier $ \mu_0,  $
$ \ \mu_j - 1 \ $ for $ \ 1 \leq j \leq o, \ $ respectively.
Applying the method of Lagrange multipliers and setting
$\ b_j = \sqrt{p_j}e^{-\frac{\mu_{0}}{2}-\frac{\mu_{j}}{2}}>0 \ $ so
$\ \lambda_{ij}=0\  $ for $\ |i-j|> k,  \ $ and for $\ |i-j|> k,  \ $ we get that
$$\lambda_{jj} = p_j e^{-\mu_{0}}e^{-\mu_j} =
b_j^2, \ \ \ \ \ \
\lambda_{ij} =  \sqrt{p_ip_{j}}e^{-\mu_{0}}e^{-\frac{\mu_i}{2}}e^{-\frac{\mu_{j}}{2}} =
b_jb_{j+1}.  $$
From the constrains one obtains the desired result.
\end{proof}

Note that $\ b_j^2 = \lambda_{jj} \ $  measures the proportion of transitions from
zone $\ j\ $ to itself.

\begin{thm}\label{last}
{\em Let $\ S \ $ be a micro-macro dynamical system  as in Theorem  \ref{bb} with $\ b_j=q^{j-1}b_1\ $ for
$\ j \in [o], \ o\geq 2, \ $ and $\ q>1.\ $
For $\ q \ $ large enough that  we have that:
\begin{enumerate}
  \item  $ \displaystyle b_1^2 \approx  \frac{1}{q^{2o-2}}; \ $
  $\ \ S \in \mathrm{L}_1(\varepsilon)\  $ if and only if  $ \ \ \displaystyle \frac{1}{q}  \leq \varepsilon.$
  \item  $ S\in \mathrm{GAT}(\varepsilon) \  $ if and only if
  $\ \ \displaystyle \frac{1}{q^2}  \leq \varepsilon; \ $
  $\ \ S \in \mathrm{ZAT}(\varepsilon) \  $ if and only if
  $\ \ \displaystyle \frac{1}{q^k}  \leq \varepsilon.$
  \item  If $ \ \displaystyle \frac{1}{q}  \leq \varepsilon, \ $ then
  $\ S  \in \mathrm{L}_3(\varepsilon,\varepsilon). \ \ $
  \item  If $ k=1  $ or $o=2 $ the average jump is zero. If $ k \geq 2  $ and $ o \geq 3  $
  the average jump is  $\ \ \displaystyle \frac{1}{q^{2}}.$
\end{enumerate}
}
\end{thm}

A subtler approach is to incorporate the
principles of proportionality and continuity by fixing beforehand the average jump,
and  looking for the maximum entropy transitions proportionalities with such jumpiness.
For $\ i,j \in [o] \ $  set $\ J(i,j)=|i-j| + \delta_{ij}-1. \ $  Fix zone proportions
$\ p_i \ $ for  $\ i \in [o], \ $ and fix $\ \delta \in [0,\Delta] \subseteq \mathbb{R}_{\geq 0} \ $
to be regarded as the average jump of a micro-macro dynamical system, where
$\ \Delta = \underset{\lambda}{\mathrm{sup}}\sum_{ij}J(i,j)\lambda_{ij}  \ $ is the largest
average jump for a probability $\ \lambda_{ij} \ $
satisfying the first five constrains below.

\begin{thm}\label{ya}
{\em
 The maximum entropy probability $\ \lambda_{ij} \ $  on
$\ [o]\times [o] \ $ subject to the constrains
$$\lambda_{ij}=\lambda_{ji}\geq 0,  \ \ \ \ \ \sum_{ij}^o \lambda_{ij}=1,
\ \ \ \ \ \sum_{i=1}^o \lambda_{ij}   =  p_j,
 \ \ \ \ \ \sum_{i=1}^o \lambda_{ji}  =  p_j, \ \ \ \ \
\sum_{ij}J(i,j)\lambda_{ij}  =  \delta ,$$
exists and it is given  setting $\ \displaystyle b_j = \frac{e^{-f_i}}{\sqrt{Z}}, \ \ c=e^{-\lambda} \ $ by
$$\lambda_{ij}\ = \ \frac{e^{-f_i-f_j-J(i,j)\lambda}}{\sum_{ij}e^{-f_i-f_j-J(i,j)\lambda}}
\ = \ \frac{e^{-f_i-f_j-J(i,j)\lambda}}{Z}
\ = \ b_ib_jc^{J(i,j)}.$$ Assume that $\ 0 < b_1<...< b_o \ $  and consider an  invertible symmetric micro-macro dynamical system $\ S \ $
 with  transition proportions $\  b_ib_jc^{J(i,j)}.\ $  We have that:
\begin{enumerate}
  \item $S \in \mathrm{L}_1(\varepsilon) \ \ $ if and only if $\ \
\sum_{i<j}b_{i}b_{j}c^{J(i,j)}  \leq \varepsilon .$
  \item $S\in \mathrm{GAT}(\varepsilon) \ \ $ if and only if
$\ \ \displaystyle \sum_{i\leq j<o}b_{i}b_{j}c^{J(i,j)}\leq \varepsilon \sum_{i, \ j<o}b_{i}b_{j}c^{J(i,j)}.$
  \item $S \in \mathrm{ZAT}(\varepsilon) \ \ $ if and only if $\ \  \sum_{i\leq j}b_{i}c^{J(i,j)}\leq
  \varepsilon \sum_{i}b_{i}c^{J(i,j)} \ $ for $\ j \in [o-1].$
 \item $S$ is symmetric and the proportionality constants of the (invariant or equivariant) reversible system associated to $\ S \ $ agree with those of
  $\ S.$
  \item Let $\ q>1\ $  and set $\ b_j=q^{j-1}b_1\ $ for $\ j \in [o]\ $ and $\ o\geq 2. \ $
  Properties 1-4 of Theorem \ref{last} hold (setting $k=1$ in property 4).
  \item for $\ \lambda=0\ $ we recover the  proportionality model from Theorem \ref{g2}, and for
$\ \lambda \rightarrow \infty \ $ we recover the  bounded jump proportionality model from Theorem \ref{bb}
with $\ k=1. \ $
  \end{enumerate}
  }
  \end{thm}

\begin{thm}
{\em Let $\ 0 < c(n)< 1\ $ and $\ 0 < b_1(n)<...< b_o(n) \ $ be $\ o+1\ $ sequences of real numbers  such that
$\ b_o(n) \rightarrow 1\ $ and $\ \displaystyle \frac{b_j(n)}{b_{j+1}(n)} \rightarrow 0\ $ as $\ n \rightarrow \infty. \ $
A sequence of invertible micro-macro dynamical systems with zone transition proportions $\ \lambda_{ij}\ $ given either
as in  Theorem \ref{bb} or as in Theorem \ref{ya}
has property $\ \mathrm{L}_3.$
}
\end{thm}

\begin{proof}
We consider the latter case, the former being similar, by verifying  conditions 1 and 3 from  Theorem \ref{ya}.
As $\ n \rightarrow \infty\  $ we have that
$$0\ \leq \ \sum_{i<j}b_{i}(n)b_{j}(n)c(n)^{J(i,j)} \ \leq \ {o \choose 2}b_{o-1}(n)b_{o}(n) \ \rightarrow \ 0, \ \ \ \ \
 \mbox{and} $$
$$0 \ \leq \ \frac{ \sum_{i\leq j}b_{i}(n)c(n)^{J(i,j)}}{\sum_{i}b_{i}(n)c(n)^{J(i,j)}}\ \leq \
j\frac{b_{j}(n)}{b_{j+1}(n)}\ \rightarrow \ 0.$$
\end{proof}

\

\noindent ragadiaz@gmail.com\\
\noindent Universidad Nacional de Colombia - Sede Medell\'in, Facultad de Ciencias, Escuela de Matem\'aticas, Medell\'in, Colombia\\

\noindent sergiogyoz@hotmail.com\\
\noindent Departamento de Matem\'aticas, Universidad Sergio Arboleda, Bogot\'a, Colombia\\

\end{document}